\documentclass[11pt,a4paper,reqno]{amsart}

\allowdisplaybreaks

\newcommand{\zero}[1]{{#1}_{\scriptscriptstyle{(0)}}}
\newcommand{\one}[1]{{#1}_{\scriptscriptstyle{(1)}}}
\newcommand{\two}[1]{{#1}_{\scriptscriptstyle{(2)}}}

\newcommand{\three}[1]{{#1}_{\scriptscriptstyle{(3)}}}
\newcommand{\four}[1]{{#1}_{\scriptscriptstyle{(4)}}}

\newcommand{\mzero}[1]{{#1}_{\scriptscriptstyle{(0)}}}
\newcommand{\mone}[1]{{#1}_{\scriptscriptstyle{(-1)}}}
\newcommand{\mtwo}[1]{{#1}_{\scriptscriptstyle{(-2)}}}
\newcommand{\mthree}[1]{{#1}_{\scriptscriptstyle{(-3)}}}
\newcommand{\mfour}[1]{{#1}_{\scriptscriptstyle{(-4)}}}

\newcommand{\ot}{\otimes}
\newcommand{\bbC}{\mathbb{C}}
\newcommand{\bbN}{\mathbb{N}}
\newcommand{\id}{\mathrm{id} }
\newcommand{\del}[1]{{}^{#1}\delta}

\renewcommand{\cot}{\gamma}
\newcommand{\coin}[2]{\bar{\cot}\big({#1}\otimes{#2}\big)}
\newcommand{\co}[2]{\cot\big({#1}\ot{#2}\big)}
\newcommand{\cat}{\qMod{A}{B}{}{B}}
\newcommand{\tcat}{\qMod{A_\gamma}{B_\gamma}{}{B_\gamma}}
\newcommand{\hm}[3]{{}_{#1}\mathrm{Hom}(#2, #3)}
\newcommand{\delbar}[1]{\overline{\partial}_{#1}}
\newcommand{\ev}{\mathrm{ev}}
\newcommand{\coev}{\mathrm{coev}}
\newcommand{\ch}{\nabla_{\mathrm{Ch},\cE}}
\newcommand{\tch}{\nabla_{\mathrm{Ch}, \Gamma(\cE)}}

\newcommand{\dc}{(\Omega^\bullet, \wedge, d)}
\newcommand{\dctwisted}{(\Omega_\gamma^{\bullet}, \wedge_\gamma, d_\gamma)}
\newcommand{\complexdc}{(\Omega^{(\bullet, \bullet)}, \wedge, \partial, \delbar{})}
\newcommand{\complexdctwisted}{(\Omega_\gamma^{(\bullet, \bullet)}, \wedge_\gamma, \partial_\gamma, (\delbar{})_\gamma)}
\newcommand{\metric}{(g, ( ~, ~ ))}
\newcommand{\metrictwisted}{(g_\gamma,(~,~)_\gamma)}

\newcommand{\ol}[1]{\overline{#1}}
\newcommand{\Hol}{\mathrm{Hol}}

\newcommand{\Title}{The Levi-Civita connection and Chern connections for cocycle deformations of K\"{a}hler manifolds}%
\newcommand{\ShortTitle}{\Title}%

\newcommand{\AuthorOne}{Bappa Ghosh}%
\newcommand{\AuthorTwo}{Jyotishman Bhowmick}
\newcommand{\AuthorOneAddr}{%
Stat-Math Unit, Indian Statistical Institute, Kolkata, India 700108
}%
\newcommand{\AuthorTwoEmail}{jyotishmanb@gmail.com}
\newcommand{\AuthorOneEmail}{%
bappa0697@gmail.com
}%
\newcommand{\AuthorOneThanks}{%
B.G. is supported by Shyama Prasad Mukherjee fellowship from CSIR, India}%

\newcommand{\SubjectClassText}{46L87, 81R60, 81R50, 17B37, 16T05}
\newcommand{\Keywords}{quantum groups, noncommutative geometry,
complex geometry.
}

\newcommand{\pdfTitle}{PDF Title}
\newcommand{\pdfAuthor}{PDF Author(s)}
\newcommand{\pdfSubject}{PDF Subject}
\newcommand{\pdfKeywords}{\Keywords}
\newcommand{\pdfCreator}{Write a creator name (e.g. MikTeX)}
\newcommand{\pdfCreationDate}{\today}
\newcommand{\pdfColorLink}{true}
\newcommand{\pdfLinkColor}{blue}
\newcommand{\pdfUrlColor}{blue}
\newcommand{\pdfCiteColor}{blue}

\def\clap#1{\hbox to 0pt{\hss#1\hss}}
\def\mathllap{\mathpalette\mathllapinternal}

\def\mathllapinternal#1#2{%
	\llap{$\mathsurround=0pt#1{#2}$}}

\newsavebox\qModFirst
\newsavebox\qModSecond
\newcommand{\Mod}{\mathsf{Mod}}
\newcommand{\modz}[2]{
	\sbox{\qModFirst}{$#1$}
	\sbox{\qModSecond}{$#2$}
	\ifdim\wd\qModFirst>\wd\qModSecond
	{}^{\phantom{#1}\mathllap{#1}}_{\phantom{#1}\mathllap{#2}}%
	\else
	{}^{\phantom{#2}\mathllap{#1}}_{\phantom{#2}\mathllap{#2}}%
	\fi\mathsf{mod}_0}
\newcommand{\Modz}[2]{
	\sbox{\qModFirst}{$#1$}
	\sbox{\qModSecond}{$#2$}
	\ifdim\wd\qModFirst>\wd\qModSecond
	{}^{\phantom{#1}\mathllap{#1}}_{\phantom{#1}\mathllap{#2}}%
	\else
	{}^{\phantom{#2}\mathllap{#1}}_{\phantom{#2}\mathllap{#2}}%
	\fi\mathsf{Mod}_0}
\newcommand{\qMod}[4]{{%
		\sbox{\qModFirst}{$#1$}
		\sbox{\qModSecond}{$#2$}
		\ifdim\wd\qModFirst>\wd\qModSecond
		{}^{\phantom{#1}\mathllap{#1}}_{\phantom{#1}\mathllap{#2}}%
		\else
		{}^{\phantom{#2}\mathllap{#1}}_{\phantom{#2}\mathllap{#2}}%
		\fi
		\Mod^{#3}_{#4}}}

\newcommand{\qmod}[4]{{%
		\sbox{\qModFirst}{$#1$}
		\sbox{\qModSecond}{$#2$}
		\ifdim\wd\qModFirst>\wd\qModSecond
		{}^{\phantom{#1}\mathllap{#1}}_{\phantom{#1}\mathllap{#2}}%
		\else
		{}^{\phantom{#2}\mathllap{#1}}_{\phantom{#2}\mathllap{#2}}%
		\fi
		\mathsf{mod}^{#3}_{#4}}}

\usepackage[top=0.9in, bottom=0.9in, left=0.95in, right=0.95in]{geometry}
\usepackage{amsmath, amssymb, amsthm} 
\usepackage[utf8]{inputenc}
\usepackage[T1]{fontenc}
\usepackage{mathtools}
\usepackage{mathrsfs} 
\usepackage{xfrac} 
\usepackage{array}
\usepackage{verbatim}
\usepackage{graphicx}
\usepackage{enumitem} 
\usepackage{hyperref}
\hypersetup{
	pdftitle={\pdfTitle},
	pdfauthor={\pdfAuthor},
	pdfsubject={\pdfSubject},
	pdfcreationdate={\pdfCreationDate},
	pdfcreator={\pdfCreator},
	pdfkeywords={\pdfKeywords},
	colorlinks=\pdfColorLink,
	linkcolor={\pdfLinkColor},
	urlcolor=\pdfUrlColor,
	citecolor=\pdfCiteColor,
	pdfpagemode=UseOutlines,
}
\usepackage{tikz-cd} 
\usepackage{lipsum}
\usetikzlibrary{matrix,arrows}
\usepackage[english]{babel}
\usepackage{lmodern}
\usepackage{xcolor}

\theoremstyle{plain}
\newtheorem{theorem}{Theorem}[section]
\newtheorem{prop}[theorem]{Proposition}
\newtheorem{lem}[theorem]{Lemma}
\newtheorem{cor}[theorem]{Corollary}

\theoremstyle{plain}
\newtheorem{definition}[theorem]{Definition}
\newtheorem{example}[theorem]{Example}

\theoremstyle{plain}
\newtheorem{remark}[theorem]{Remark}

\numberwithin{equation}{section}

\newtheoremstyle{ser}
{8pt}
{8pt}
{\it}
{}
{\sf}
{:}
{6mm}
{}

\theoremstyle{ser}

\newtheoremstyle{serr}
{8pt}
{8pt}
{\normalfont}
{}
{\sf}
{.}
{6mm}
{}

\theoremstyle{serr}

\theoremstyle{ser}

\theoremstyle{ser}

\definecolor{indraRed}{rgb}{0.593, 0.183, 0.183}
\definecolor{indraPink}{rgb}{0.858, 0.188, 0.478}
\definecolor{indraBlue}{rgb}{0, 0.199, 0.398}
\definecolor{madridBlue}{rgb}{0.199, 0.199, 0.695}
\definecolor{metropolisThemeColor}{rgb}{0.105, 0.214, 0.234}
\definecolor{metropolisBarColor}{rgb}{0.984, 0.0.515, 0.015}
\definecolor{UBCblue}{rgb}{0.04706, 0.13725, 0.26667} 
\definecolor{UBCgrey}{rgb}{0.3686, 0.5255, 0.6235} 

\makeatletter
\def\mathcolor#1#{\@mathcolor{#1}}
\def\@mathcolor#1#2#3{%
	\protect\leavevmode
	\begingroup
	\color#1{#2}#3%
	\endgroup
}
\makeatother

\newcommand{\ra}{\rightarrow} 

\newcommand{\Hom}{{\rm{Hom}}}

\newcommand{\R}{\mathbb{R}}

     \newcommand{\cC}{\mathcal C}		
     \newcommand{\cD}{\mathcal D}		
     \newcommand{\cE}{\mathcal E}		
     \newcommand{\cF}{\mathcal F}		
     		
     		\newcommand{\kH}{\mathscr{H}}

\newcommand{\fN}{\mathfrak N}

\newcommand{\fS}{\mathfrak S}

\newcommand{\bN}{{\mathbb N}}

\begin{document}

\title[\ShortTitle]{\MakeUppercase\Title}

\author{\AuthorTwo}
\author{\AuthorOne} 
{\thanks{\AuthorOneThanks}}
\address{\AuthorOneAddr}
\email{\AuthorTwoEmail}
\address{\AuthorOneAddr}
\email{\AuthorOneEmail}

\subjclass{\SubjectClassText}


\begin{abstract}
	We consider unitary cocycle deformations of covariant $\ast$-differential calculi. We prove that complex structures, holomorphic bimodules and Chern connections on the deformed calculus are twists of their untwisted counterparts. Moreover,  for cocycle deformations of a class of classical K\"ahler manifolds, the Levi-Civita connection on the space of one-forms of the deformed calculus is shown to be a direct sum of the Chern connections on the twisted holomorphic and the anti-holomorphic bimodules. Our class of examples also includes cocycle deformations of the Heckenberger-Kolb calculi.    
\end{abstract}
\maketitle%
\thispagestyle{empty}%
\section{Introduction}

The last decade saw a number of approaches to study connections and curvatures in noncommutative geometry. One of the central themes have been the question of existence and uniqueness of Levi-Civita connections for different classes of examples of noncommutative manifolds. Classically, the existence of the Levi-Civita connection is derived by a Koszul type formula and a similar approach was taken in  \cite{KoszulJDG}. For examples coming from quantum groups and their homogeneous spaces, we refer to  \cite{aschieriweber}, \cite{atiyahseq} and references therein. In the setup of spectral triples (\cite{connes}) with different definitions of metric-compatibility and torsion,   the authors of \cite{meslandrennie1} proved the existence and uniqueness of Levi-Civita connections for toric deformations of classical manifolds. They also proved the existence of Levi-Civita connections for a spectral triple on the Podle\'{s} sphere in \cite{meslandrennie3}. In a very recent work, a  comparison between the frameworks of \cite{article7} for tame differential calculi, the derivation based framework of \cite{arnwil}, \cite{arnwil2} and \cite{meslandrennie1} has been deduced in \cite{flamantMR}. A comparison between the approaches taken in \cite{BeggsMajid:Leabh} and \cite{meslandrennie1} is available in \cite{comparison}. In this article we will follow the setup of \cite{BeggsMajid:Leabh}.

On the other hand, the notion of noncommutative complex structure has been studied by a number of mathematicians (see \cite{PolSchwar2003}, \cite{KLvSPodles}, \cite{BS}, \cite{MMF2} and references therein). This led to the definition of K\"ahler structures in \cite{MMF3} and the subsequent identification of   K\"ahler structures on a family of quantum homogeneous spaces in \cite{MATASSA2019103477}.  In this article, we investigate the relationship between Levi-Civita connections and Chern connections for a class of noncommutative K\"ahler manifolds.

One of the fundamental results in K\"ahler geometry states that if $M$ is a K\"ahler manifold, then the Levi-Civita connection on the real tangent bundle on $M$ can be related to  the Chern connection on the holomorphic tangent bundle. On the complexified space of one-forms $\Omega^1= \Omega^{(1,0)}\oplus \Omega^{(0,1)}$
 the Levi-Civita connection $\nabla$ for a Riemannian metric can be expressed as the direct sum of Chern connections for the associated Hermitian metrics on $\Omega^{(1,0)}$ and $\Omega^{(0,1)}.$ The goal of this article is to establish a noncommutative analogue of this theorem for a class of noncommutative K\"ahler manifolds via deformation of the classical Levi-Civita connection.
 

Let us explain the framework of \cite{BeggsMajid:Leabh} in which we work. We consider a Hopf $\ast$-algebra $A$ and a $\ast$-algebra $B$ which is equipped with an $A$-comodule $\ast$-algebra coaction 
$$  \prescript{}{}{\delta}: B \rightarrow A \otimes B. $$
The noncommutative manifold structure for the triplet $(A, B, \prescript{}{}{\delta})$ is given by an $A$-covariant $\ast$-differential calculus $\dc$ (see Definition \ref{8thmarch256}) and the geometric input data is that of an $A$-covariant metric $\metric,$  the latter being  a choice of a self-dual structure on the space of one-forms $\Omega^1.$ Moreover, in order to make sense of metric-compatibility of a connection $\nabla$ on $\Omega^1,$ we work with bimodule connections, while the torsion of $\nabla$ is the map $T_\nabla: \wedge \circ \nabla - d: \Omega^1 \rightarrow \Omega^2.$  A bimodule connection $\nabla$ on $\Omega^1$ is called a Levi-Civita connection for $\metric$ if $\nabla$ is torsionless and compatible with $\metric$ (Definition \ref{15thdec242}).

An $A$-covariant complex structure on $\dc$ is an $\mathbb{N}_0 \times \mathbb{N}_0$-grading on $\dc$ by $A$-covariant $B$-bimodules $\Omega^{(p,q)}  $ satisfying $\Omega^k = \oplus_{p + q = k} \Omega^{(p,q)}  $ and some more compatibility conditions (see Definition \ref{18thfeb255}). Of particular interest are the bimodules $\Omega^{(1,0)}  $ and $\Omega^{(0,1)} $ which are holomorphic (Definition \ref{11thdec231}) if the complex structure is factorizable for the given and the opposite complex structure respectively. In this case, a choice of a real metric $\metric$  on $\Omega^1$ satisfying 
\begin{equation}\label{diamondnew}
		(\eta_1, \eta_2)=0\text{ for  all } \eta_1, \eta_2 \text{ in }\Omega^{(1,0)}\text{ or }\Omega^{(0,1)},
	\end{equation}
	yields an Hermitian metric (Definition \ref{4thdec231}) $\kH_g$ on $\Omega^1 = \Omega^{(1,0)} \oplus \Omega^{(0,1)} $ such that $\kH_g = \kH_1 \oplus \kH_2 $ for covariant Hermitian metrics $\kH_1$ and $\kH_2$ on $\Omega^{(1,0)}$ and $ \Omega^{(0,1)}$ respectively. Then by a result of Beggs-Majid (\cite{BeggsMajid:Leabh}),  we have unique Chern connections $\nabla_{\text{Ch}, \Omega^{(1,0)}}$ and $\nabla_{\text{Ch}, \Omega^{(0,1)}}$ for the pairs $(\Omega^{(1,0)}, \kH_1)$ and $(\Omega^{(0,1)}, \kH_2). $

Now, let us introduce cocycle deformations in the picture. If $\gamma$ is a unitary $2$-cocycle on the Hopf $\ast$-algebra $A,$ then we have a deformed Hopf $\ast$-algebra $A_\gamma$ and a deformed $A_\gamma$-comodule $\ast$-algebra $(B_\gamma, \prescript{B_\gamma}{}{\delta} ).$ In fact, there is a  well-known monoidal equivalence $ \Gamma: \cat \rightarrow \tcat $ of relative Hopf-modules (see \eqref{21stnov243}).  It follows that  the $A$-covariant $\ast$-differential calculus $\dc$ is deformed to an $A_\gamma$-covariant $\ast$-differential calculus $\dctwisted.$ If $\nabla_{\Omega^1}$ is the unique  $A$-covariant Levi-Civita connection for a metric $\metric$ on $\Omega^1,$ then it is already known (see Theorem \ref{thm:15thdec243}) that a $\gamma$-twist $\nabla_{\Omega^1_\gamma}$ is the unique Levi-Civita connection for the twisted metric $\metrictwisted$ on $\Omega^1_\gamma.$

The goal of this article is to answer the following question:

\noindent If an $A$-covariant $\ast$-differential calculus is equipped with a real metric $\metric$ and a factorizable complex structure so that  
$$ \nabla = \nabla_{\text{Ch}, \Omega^{(1, 0)}} \oplus \nabla_{\text{Ch}, \Omega^{(0, 1)}}$$
 with respect to the decomposition $\Omega^1 = \Omega^{(1,0)} \oplus  \Omega^{(0,1)}, $ then is the twisted Levi-Civita connection $\nabla_{\Omega^1_\gamma}$ related with the Chern connections on the twisted bimodules $\Omega^{(1,0)}_\gamma$ and $\Omega^{(0,1)}_\gamma.$

We give an affirmative answer to this question for real metrics satisfying \eqref{diamondnew} by showing that
\begin{equation} \label{11thsep251} 
\nabla_{\Omega^1_\gamma} = \nabla_{\text{Ch}, \Omega^{(1, 0)}_\gamma} \oplus \nabla_{\text{Ch}, \Omega^{(0, 1)}_\gamma}.
\end{equation}
However, in order to make sense of \eqref{11thsep251}, we need to prove a number of results. 

To begin with, we show that under cocycle deformations, a covariant Hermitian metric $\kH$ on an object $\cE$ in $\cat$ can be deformed to an $A_\gamma$-covariant Hermitian metric $\kH_\gamma$ on the  $B_\gamma$-bimodule  $\Gamma  ( \cE ) $ (see Section \ref{sec:cocyledef_Rel}).   With a bit more work, it follows that the cocycle deformation of the Hermitian metric $\kH_g$ on $\Omega^1$ is equal to $\kH_{g_{\gamma}},$ i.e, the Hermitian metric obtained from the twisted metric $\metrictwisted.$ Next, we turn our attention to prove that  covariant complex structures and holomorphic bimodules can be twisted in the presence of a unitary $2$-cocycle.  Moreover, if $\ch$ is the Chern connection for an Hermitian metric on a holomorphic bimodule $\cE,$  then it is proven that the Chern connection for the pair $(\Gamma(\cE), \kH_\gamma )$ is given by the cocycle deformation of $\ch.$ 

Summarizing, if $\dc$ has an $A$-covariant factorizable complex structure, then we get Chern connections $\nabla_{\text{Ch}, \Omega^{(1, 0)}_\gamma}$ and $ \nabla_{\text{Ch}, \Omega^{(0, 1)}_\gamma}$ for the pairs $(\Omega^{(1,0)}_\gamma, (\kH_1)_\gamma)$ and $(\Omega^{(0,1)}_\gamma, (\kH_2)_\gamma),$ where $\kH_1$ and $\kH_2$ are the Hermitian metrics mentioned above.  Finally,  we combine all the results obtained so far to prove \eqref{11thsep251}. This is done in  Theorem \ref{29thjan253}
which is the main result of the article. In Subsection \ref{4thmarch254}, we have collected many examples of cocycles coming from compact quantum group algebras. For more examples of cocycles, we refer to \cite[Chapter 3]{NeshveyevTuset} and references therein.

In \cite{LeviCivitaHK}, the existence and uniqueness of Levi-Civita connection on the Heckenberger-Kolb calculi on the quantized irreducible flag manifolds for every real covariant metric was proved. It was shown that the Levi-Civita connection is the sum of the Chern connections on the spaces of  holomorphic and anti-holomorphic one-forms.
The class of examples tackled in \cite{LeviCivitaHK} are a class of quantum homogeneous spaces and a key role in the proof of the main theorem in that article was played by Takeuchi's monoidal equivalence between certain categories of relative Hopf modules.
In this article,  we appeal to the well-known monoidal equivalence $ \Gamma: \cat \rightarrow \tcat $ of relative Hopf-modules  and careful applications of cocycle identities. 
We should mention that in the particular case of toric deformations (studied in \cite{meslandrennie1}), the deformed Hopf $\ast$-algebra is still the untwisted function algebra on $\mathbb{T}^n$ and the $\ast$-structure on the deformed algebra $B_\gamma$ coincides with the $\ast$-structure on $B$ (see Example \ref{8thmarch257}). This leads to a simplification of the proofs. But since we are dealing with arbitrary unitary $2$-cocycles, the proofs need much more care.  

Now, let us come to the range of applications. Our first class of examples are unitary cocycle deformations of a class of K\"ahler manifolds. Indeed, if we have an action of a smooth affine real algebraic group  $G$ on a smooth  affine real algebraic variety $X$ such that the set of real points $X(\mathbb{R})$ of $X$ has a covariant  K\"ahler structure in the sense of Definition \ref{def:kahler-manifold}, then it is well-known that the hypotheses of Theorem \ref{29thjan253} are satisfied. However, for the convenience of the reader, we have explained the classical situation clearly in Subsection \ref{11thsep252}.  Secondly, our result Theorem \ref{29thjan253} is applicable to cocycle deformations of genuine noncommutative manifolds too. Indeed we establish the existence and uniqueness of Levi-Civita connection on unitary cocycle deformations of the Heckenberger-Kolb calculi via deformation of the connection constructed in \cite{LeviCivitaHK}.


We have collected most of the necessary definitions, notations and relevant results in the first two sections of the article. In Section \ref{8thmarch251}, we recall the notions of differential calculi, metrics and connections, while Section \ref{31stjan251} deals with cocycle deformations of relative Hopf modules and differential calculi. All the results quoted in Section \ref{31stjan251} are well-known except the possible exception of the second assertion of Proposition \ref{prop:5thdec241}. Section \ref{8thmarch252} deals with the deformation of covariant Hermitian metrics on covariant bimodules. In Section \ref{8thmarch253}, we prove that complex structures, holomorphic bimodules and Chern connections can be deformed to their twisted counterparts. The main result of this article and two classes of examples are discussed in Section \ref{8thmarch254}.  

Our article ends with two appendices. Since our work makes heavy use of cocycle identities,  we have proved  these equations in the first subsection of Appendix \ref{8thmarch255} for the ease of readability.  Another technical tool that we have used is the language of bar categories. In Appendix \ref{B}, we have  studied  interactions of bar categories with the unitary $2$-cocycle deformations.

\vspace{6mm}

\paragraph{\textbf{Acknowledgments:}} We are grateful to Ulrich Kr\"ahmer for discussions related to Section \ref{8thmarch254}. It is our pleasure to thank Debashish Goswami for posing a question which eventually led to this article and for fruitful discussions.

\section{Preliminaries} \label{8thmarch251}


Throughout this article, unless otherwise mentioned, our ground field will be $\mathbb{C}$. All algebras are assumed to be unital. All unadorned tensor products are over $\mathbb{C}$.

Let $ (\mathcal{C}, \otimes_\mathcal{C}) $ be a monoidal category with unit object $1_{\mathcal{C}}$. An object $M$ in  ${\mathcal{C}}$ is said to have a~right dual if there exists an object $\prescript{\ast}{}{M}$ in ${\mathcal{C}}$ and  morphisms $\ev: M\otimes_\mathcal{C} \prescript{\ast}{}{M} \to 1_{\mathcal{C}}$, $\coev: 1_{\mathcal{C}}\to \prescript{*}{}{M}\otimes_\mathcal{C} M$ such that the equations
  \begin{align*}
      (\ev\otimes_\mathcal{C} \id_M) (\id_M\otimes_\mathcal{C}  \coev) =\id_M, \quad (\id_{\prescript{\ast}{}{M}}\otimes_\mathcal{C}\ev)(\coev\otimes_\mathcal{C} \id_{\prescript{\ast}{}{M}})  = \id_{\prescript{\ast}{}{M}}
  \end{align*}
are satisfied.


If $B$ be an~algebra,  then category of all $B$-bimodules will be denoted by $\qMod{}{B}{}{B}$.
For objects $M, N$  in the category $\qMod{}{B}{}{B},$ we define 
\begin{align*}
	\hm{B}{M}{N}:=  \{f: M \to N : f \text{ is  left  $B$-linear}\}.
\end{align*}

Let $A$ be a Hopf-algebra with coproduct $\Delta: A\to A\otimes A$, counit $\epsilon: A\to \mathbb{C}$ and antipode $S:A\to A$. The antipode is always assumed to be bijective and moreover, we will use Sweedler notation to write 
\[
\Delta (a) = a_{(1)} \otimes a_{(2)}.
\]

For a Hopf algebra $A$ and a left $A$-comodule $V$, the coaction on an element
$v\in V$ is written in Sweedler notation as 
$$\del{V}(v) = \mone{v}\ot\zero{v}.$$
An element $v \in V$ is called coinvariant if $\del{V}(v) = 1 \otimes v. $
We denote by $\qMod{A}{}{}{}$ the monoidal category of left $A$-comodules.

For a Hopf algebra $A$ and a left $A$-comodule algebra $B$, the symbol $\qMod{A}{B}{}{B}$ is reserved for the category of relative Hopf-modules. Thus an object $(\cE, \del{\cE})$ of $\cat$ is a $B$-bimodule and left $A$-comodule such that
$$\del{\cE}(aeb)= \del{B}(a)\del{\cE}(e) \del{B}(b)= \mone{a}\mone{e} \mone{b}\ot \zero{a}\zero{e}\zero{b} \quad\text{for all } e \in \cE, a, b\in B .$$
A morphism $f: \cE \to \cF $  of the category $\cat$ is a $B$-bilinear map which is also $A$-covariant, i.e., $\del{\cF} (f ( e )) = (\id\ot f )\del{\cE} ( e ) $.

If $\cE, \cF$ are two objects of the category $\cat,$ then  $\prescript{}{B}{\Hom(\cE, \cF)}$ is again a $B$-bimodule with structure maps defined as: 
\begin{align*}
	(b\cdot f)(e)=f(e\cdot b); \quad (f \cdot b)(e) = f(e)b,
\end{align*}
where $f \in \prescript{}{B}{\Hom(\cE,\cF)}, e\in \cE $ and $ b \in B$. If $\cE$ is finitely generated as a left $B$-module, then following \cite[Section 2]{Ulb90} or \cite[Proposition 2.11]{aschieriweber}, we have a left $A$-coaction $\del{\hm{B}{\cE}{\cF}}$ on $\hm{B}{\cE}{\cF}$. For $f \in \hm{B}{\cE}{\cF}$,  $\del{\hm{B}{\cE}{\cF}}(f):= \mone{f} \ot \zero{f}$, where for all $e \in \cE$, we have 
\begin{equation} \label{eq:27thnov241}
	\mone{f} \ot \zero{f}(e)= S(\mone{e}) \mone{[f(\zero{e})]} \ot \zero{[f(\zero{e})]}.
\end{equation}    

Then one has the following  result:
\begin{prop} \label{28thnov241jb} (\cite[Section 2]{Ulb90}, \cite[Proposition 2.11]{aschieriweber})
Let $B$ be a left $A$-comodule algebra. If $\cE$ and $\cF$ are objects in $\cat$ and $\cE$ is finitely generated as a left $B$-module, then $\hm{B}{\cE}{\cF}$ is an object of $\cat.$ Moreover, if $\cE$ is finitely generated and projective as a left $B$-module, then $\hm{B}{\cE}{B}$ is a right dual of $\cE$ in $\cat.$
\end{prop}

Let us also recall the definition of a Hopf $*$-algebra. A Hopf algebra $A$ is called a Hopf $\ast$-algebra if $A$ is a $\ast$-algebra and  $\Delta$ is a $\ast$-homomorphism. Then it follows that the antipode $S$ is invertible and satisfies the following identity (see \cite[Subsection 1.2.7]{KSLeabh}):  	
\begin{equation}\label{19thsept20245}
	S(S(a^*)^*)=a 
\end{equation}
for all $a \in A$.

\subsection{Bar Categories} \label{2ndfeb261}

The theory of bar categories was developed in \cite{BarCategory}. In this subsection, we collect the relevant notations and definitions which have been used throughout the article. For a monoidal category $ (\mathcal{C}, \otimes_\mathcal{C})$ and objects $X, Y $ of $ \mathcal{C}, $ we denote by flip the functor from  $\mathcal{C}\times \mathcal{C} $ to $\mathcal{C}  \times \mathcal{C}$ which sends the pair $(X,Y)$ to $ (Y,X)$. We also denote by $1_{\mathcal{C}}$ the unit object. As usual, we will suppress the notations for the left unit, right unit as well as the associator of $\mathcal{C}$.

\begin{definition}(\cite[Definition 2.1]{BarCategory}) \label{15thjuly241}
	A bar category is a monoidal category $(\mathcal{C},\otimes_\mathcal{C},1_{\mathcal{C}})$ together with  a functor $\mathrm{bar}: \mathcal{C}\to \mathcal{C}$ (written as $X\mapsto \overline{X}$),
	a natural equivalence $\mathrm{bb}: \id_ {\mathcal{C}} \to \mathrm{bar}\circ \mathrm{bar}  $ between the identity and the  $\mathrm{bar}\circ \mathrm{bar}$ functors on $\mathcal{C}$,
	an invertible morphism $ \star:1_{\mathcal{C}} \to \overline{1_{\mathcal{C}}}$ and
	a natural equivalence $\Upsilon$  between $\mathrm{bar}\circ \otimes_\mathcal{C}$ and $\otimes_\mathcal{C} \circ (\mathrm{bar}\times \mathrm{bar})\circ \mathrm{flip}$ from $\mathcal{C}\times \mathcal{C}$ to $\mathcal{C}$	
	such that the following compositions of morphisms are both equal to $1_{\overline{X}}:$
	$$ \overline{X} \xrightarrow{\cong} \overline{X \otimes_\mathcal{C} 1_{\mathcal{C}}} \xrightarrow{\Upsilon_{X, 1_{\mathcal{C}}}} \overline{ 1_{\mathcal{C}}}\otimes_\mathcal{C}  \overline{X} \xrightarrow{\star^{-1} \otimes_\mathcal{C} \id}  1_{\mathcal{C}} \otimes_\mathcal{C} \overline{X} \xrightarrow{\cong} \overline{X}, $$
	$$ \overline{X} \xrightarrow{\cong} \overline{ 1_{\mathcal{C}} \otimes_\mathcal{C} X} \xrightarrow{\Upsilon_{ 1_{\mathcal{C}},X}}  \overline{X} \otimes_\mathcal{C} \overline{1_{\mathcal{C}}}\xrightarrow{ \id \otimes_\mathcal{C}  \star^{-1}}  \overline{X} \otimes_\mathcal{C}  1_{\mathcal{C}} \xrightarrow{\cong} \overline{X},$$ 
	and moreover, the following equations hold:
	$$ (\Upsilon_{Y,Z} \otimes_\mathcal{C} \id) \Upsilon_{X, Y \otimes_\mathcal{C} Z} = (\id \otimes_\mathcal{C} \Upsilon_{X, Y}) \Upsilon_{X \otimes_\mathcal{C} Y, Z}, ~ \overline{\star} \star = \mathrm{bb}_{1_{\mathcal{C}}}: 1_{\mathcal{C}} \to \overline{\overline{1_{\mathcal{C}}}}, ~ \overline{\mathrm{bb}_X} =\mathrm{bb}_{\overline{X}}: \overline{X} \to \overline{\overline{\overline{X}}}$$
	for all objects $X, Y, Z$ in $\mathcal{C}$.
	
	An object $X$ in a bar category is  called a star object if there is a morphism $\star_X: X\to \overline{X}$ such that $\overline{\star_X}\circ \star_X = \mathrm{bb}_X $.
\end{definition}


If $M$ is a vector space, then the symbol $\overline{M}$ will stand for the vector space defined as 
$$ \overline{M}:= \{ \overline{m}: m \in M \}.$$ 
Moreover, if ${M}$ is $B$-bimodule, then $\overline{M}$ is equipped with the following $B$-bimodule structure:
\begin{align}\label{28thnov231}
	b\cdot \ol{m}= \ol{m\cdot b^{\ast}};\quad \ol{m}\cdot b= \ol{b^{\ast}\cdot m} \quad \text{for all $b\in B$, $m\in M,$}
\end{align} and 
if $(M, \prescript{M}{}{\delta})$ is a~left $A$-comodule, then $\overline{M}$ has a~left $A$-comodule structure defined by
\begin{align}\label{25thnov232}
	\prescript{\ol{M}}{}{\delta}(\ol{m})= m_{(-1)}^{\ast}\otimes \ol{m_{(0)}},
\end{align}
where we have used Sweedler's notation $\prescript{M}{}{\delta}(m)= m_{(-1)}\otimes m_{(0)}$.

\begin{example} (\cite[Section 2.8]{BeggsMajid:Leabh}) \label{5thdec241jb}
	If A is a Hopf $\ast$-algebra and $B$ a left $A$-comodule $\ast$-algebra, then the category $\qMod{A}{B}{}{B}$ of relative Hopf modules is a bar category.
	Indeed, if $M$ is an object of $\qMod{A}{B}{}{B},$ then  \eqref{28thnov231} and \eqref{25thnov232} make   $\overline{M}$  an object of $\qMod{A}{B}{}{B}.$  
	Now we define   $\mathrm{bar} (M) := \overline{M}$ and 
	for $f \in \Hom(M,N)$, we define 
	$ \ol{f}\in \Hom(\ol{M}, \ol{N})$ by $\ol{f}(\ol{x})=\ol{f(x)}$.
	
	Moreover, the natural equivalence $\mathrm{bb}$ is given by
	$ \mathrm{bb}_M ( m ) = \overline{\overline{m}} $ for all $m \in M.$
	Finally, for objects $M, N$ in $\qMod{A}{B}{}{B},$
	$$ \Upsilon_{M, N} (  \overline{m \otimes_B n}  ) = \overline{n} \otimes_B \overline{m}.  $$

	Let us note that if $(\Omega^{\bullet},\wedge, d)$ is an $A$-covariant $\ast$-differential calculus on $B$, then the map
	\begin{equation} \label{8thjuly241}
		\star_{\Omega^1}: \Omega^1 \rightarrow \overline{\Omega^1}, ~ \star_{\Omega^1} (\omega) := \overline{\omega^*} 
	\end{equation}
	makes  $\Omega^1 $  a star object of $\qMod{A}{B}{}{B}$.
	
\end{example}

\subsection{Differential calculi and metrics}

\begin{definition} \label{Covariant Calculi}
Let $(B, \prescript{B}{}{\delta})$ be a left $A$-comodule algebra for a Hopf-algebra $A$. An $A$-covariant differential calculus over $B$ is a differential graded algebra $( \bigoplus_{k\ge 0} \Omega^k(B),  \wedge, d)$ such that $\Omega^0(B) = B,$  $ \Omega^\bullet (B):= \bigoplus_{k\ge 0} \Omega^k(B) $ is generated as an algebra by $B$ and $dB$ and moreover,  the coaction $ \prescript{B}{}{\delta}$ extends to a (necessarily unique) comodule algebra map $\prescript{\Omega^\bullet}{}{\delta}: \Omega^\bullet (B) \rightarrow A \otimes \Omega^\bullet (B) $ such that   $\Omega^k ( B ) $ is an object of $\qMod{A}{B}{}{B}$ for each $k \geq 0$ and the map $d$ is $A$-covariant. 
\end{definition}

If $A$ is taken to be the trivial one dimensional Hopf-algebra, then we recover the definition of a differential calculus. We note that if $(\Omega^{\bullet} (B), \wedge, d)$ is any left $A$-covariant calculus, then clearly, the canonical projection map $\Omega^\bullet (B) \rightarrow \Omega^k (B) $ as well as $\wedge: \Omega^k (B) \otimes_B \Omega^l (B) \rightarrow \Omega^{k + l} (B) $ are left $A$-covariant.

Let $B$ be a $\ast$-algebra and  $A$  a Hopf $\ast$-algebra such that  $(B, \prescript{B}{}{\delta})$ is an $A$-comodule algebra. If $~^{B}\delta: B \to A \otimes B$ is a $*$-algebra homomorphism, then the pair  $(B, \prescript{B}{}{\delta})$ is called an  $A$-comodule $\ast$-algebra.

In this case,  one is led to consider the following subclass of differential calculi.

\begin{definition} \label{8thmarch256}
If $A$ is a Hopf $\ast$-algebra and $B$ is a comodule $\ast$-algebra, then an $A$-covariant differential calculus $(\Omega^{\bullet}(B), \wedge, d) $ on   $B$ is called an $A$-covariant $\ast$-differential calculus if there exists a conjugate linear involution $\ast: \Omega^{\bullet}(B) \to \Omega^{\bullet}(B) $ which extends the map $\ast: B \rightarrow B $ such that 
$$ \ast(\Omega^k(B))\subseteq \Omega^k(B), ~ (d\omega)^{\ast} =d(\omega^{\ast}), ~ (\omega \wedge \nu)^{\ast} = (-1)^{kl} \nu^{\ast} \wedge \omega^{\ast} $$
for  all $ \omega \in \Omega^{k}(B), \nu \in \Omega^l (B)$ and moreover, if $\Omega^\bullet (B)$ is a comodule $\ast$-algebra.
\end{definition}

From now on, while referring to a differential calculus on an algebra $B$, we will often use the notations $\Omega^k$ and $\Omega^{\bullet}$ to denote $\Omega^k(B)$ and $\Omega^{\bullet}(B)$ respectively.

\begin{definition} (\cite[Definition 1.15]{BeggsMajid:Leabh}) \label{4thmay242}
	A metric on a differential calculus $(\Omega^\bullet, \wedge, d)$ on an algebra $B$ is a pair $(g, (~ , ~))$ where $g $ is an element of $\Omega^1 \otimes_B \Omega^1$ and $(~ , ~): \Omega^1 \otimes_B \Omega^1 \rightarrow B$ is a $B$-bilinear map such that the following conditions hold:
	\begin{equation*} 
     ((\omega, ~) \otimes_B \id) g = \omega = (\id \otimes_B (~ , \omega)) g.
     \end{equation*}
	If the differential calculus is $A$-covariant, then we will say that the metric is covariant if  $(~ , ~)$ and the map 
	\begin{equation} \label{23rdmay241}
		\coev_g: B \rightarrow \Omega^1 \otimes_B \Omega^1 ~  \text{defined by} ~ \coev_g (b) = b g 
	\end{equation}
	are $A$-covariant.
\end{definition}

If $ (g, (~ , ~)) $ is a metric on $\Omega^1(B)$, then  by \cite[Lemma 1.16]{BeggsMajid:Leabh}, $g$ is central, i.e, $b g = g b$ for all $b \in B$. This implies that the map $\coev_g$ is actually $B$-bilinear. If $ (g, (~ , ~)) $ is a covariant metric, then the covariance of the map $\coev_g$ 
implies that
\begin{equation} \label{18thfeb252}
{}^{\Omega^1 \otimes_B \Omega^1} \delta ( g ) = 1 \otimes g.
\end{equation}

The following well-known characterization of metrics will be useful for us.

\begin{remark} \cite[Page 311]{BeggsMajid:Leabh} \label{24thjuly241}
	If $ (g, (~ , ~)) $ is a covariant metric on $\Omega^1$ and $\coev_g$ is the map defined in \eqref{23rdmay241}, then $ (\Omega^1, (~ , ~), \coev_g) $ is a right dual of $\Omega^1$ in the category $\qMod{A}{B}{}{B}$. 
	
	Conversely, if a triplet $ (\Omega^1, \ev, \coev) $ is a right dual of $\Omega^1$ in $\qMod{A}{B}{}{B}$, then the pair $ (\coev (1), \ev) $  is a covariant metric on $\Omega^1$.

    In particular, if $\dc$ is a differential calculus such that $\Omega^1$ admits a metric $ \metric, $ then $\Omega^1$ is finitely generated and projective as a left $B$-module (see \cite{etingof2015tensor}, for example). 
\end{remark}

If $\dc$ is a $*$-differential calculus, then we can make sense of the following definition:
\begin{definition} (\cite[Chapter 8]{BeggsMajid:Leabh} )  \label{23rddec24jb2}
Suppose that  $\metric$  is a metric  on a $\ast$-differential calculus $\dc$ with $g = \sum_i \omega_i \otimes_B \eta_i. $
 Then $\metric$    is said to be a real metric if 
 $$g= \sum_i \eta^*_i \otimes_B \omega^*_i.  $$
\end{definition}

\subsection{Connections} \label{4thaugust241}

If  $(\Omega^{\bullet}, \wedge, d) $ is a differential calculus over an algebra $B$, then a left connection on a $B$-bimodule $\mathcal{E}$ is a $\mathbb{C}$-linear map
$\nabla: \mathcal{E} \rightarrow \Omega^1 \otimes_B \mathcal{E}$ such that
$$\nabla (b e) = b \nabla (e)  + db \otimes_B e $$
for all $e \in \mathcal{E}$ and for all $b \in B$.
	A left connection $\nabla$ on an object $\cE$ of $\cat$ is said to be covariant if $\nabla$ is left $A$-covariant i.e., 
    $${}^{\Omega^1 \otimes_B\cE}\delta\circ \nabla= (\id \otimes_B \nabla) \circ {}^\cE\delta. $$

A left connection $\nabla$ on a $B$-bimodule $\mathcal{E}$ is called a left $\sigma$-bimodule connection if there exists a $B$-bimodule map 
$\sigma: \mathcal{E} \otimes_B \Omega^1 \rightarrow \Omega^1 \otimes_B \mathcal{E}$
such that
\begin{equation} \label{19thoct237}
	\nabla (e b) = \nabla (e) b + \sigma (e \otimes_B db)
\end{equation}
for all $e \in \mathcal{E}$ and for all $b \in B$.

If \eqref{19thoct237} is satisfied, then we will sometimes say that $ (\nabla, \sigma) $ is a left bimodule connection. It is well-known that if $ ( \nabla, \sigma_1 ) $ and $( \nabla, \sigma_2 )$ are bimodule connections on $\cE,$ then $\sigma_1 = \sigma_2.$

Suppose that $(\nabla_{\mathcal{E}}, \sigma_{\mathcal{E}}) $ is a bimodule connection and $ \nabla_ {\mathcal{F}} $ is a connection on $B$-bimodules $\mathcal{E}$ and $\mathcal{F}$ respectively.  Then by \cite[Theorem 3.78]{BeggsMajid:Leabh},  we have a  connection $ \nabla_{\mathcal{E} \otimes_B \mathcal{F}} $ on $\mathcal{E} \otimes_B \mathcal{F} $ defined as
\begin{equation*}
	\nabla_{\mathcal{E} \otimes_B \mathcal{F}}:= \nabla_{\mathcal{E}} \otimes_B \id + (\sigma_{\mathcal{E}} \otimes_B \id) (\id \otimes_B \nabla_{\mathcal{F}}).     
\end{equation*}

Now we introduce the definition of the Levi-Civita connection for a metric. 

\begin{definition} (\cite[Definition 8.3]{BeggsMajid:Leabh}) \label{15thdec242}
	Suppose that $B$ is an algebra and $\dc$ is a differential calculus on $B$. Let $\metric$ be a metric on the space of one-forms $\Omega^1(B)$. A left bimodule connection $(\nabla, \sigma) $ is said to be a Levi-Civita connection for the metric $\metric$ if the following two conditions are satisfied:
    \begin{enumerate}
    \item $\nabla$ is torsionless, i.e, the map
    $\wedge \circ \nabla - d: \Omega^1 \rightarrow \Omega^2 $ is the zero map, 
    
    \item $\nabla$ is compatible with $\metric,$ i.e, $\nabla_{\Omega^1 \otimes_B \Omega^1} g = 0.$
    \end{enumerate}
	\end{definition}

Thus, whenever we say that $\nabla$ is a Levi-Civita connection for a metric $\metric,$ it will be assumed that $\nabla$ is already a   bimodule connection. We should also mention that for connections on the space of one-forms on a spectral triple (\cite{connes}), a spectral definition of torsion has been studied, for which we refer to \cite{dabrowskisitarzadv}.

In the classical geometry, one of the fundamental theorems for Riemannian geometry is  the uniqueness and existence of the Levi-Civita connection (see \cite[Theorem 5.4]{johnmlee} but establishing an analogous result for noncommutative manifolds remains a significant challenge.

For examples of differential calculi and metrics which admit Levi-Civita connection in the sense of Definition \ref{15thdec242}, we refer to the monograph \cite{BeggsMajid:Leabh} and references therein. More recently, Matassa proved the existence uniqueness theorem for an analogue of the Fubini-Study metric on quantum projective spaces (see \cite{matassalevicivita}) and this was extended to the case of covariant real metrics on all quantized irreducible flag manifolds in \cite{LeviCivitaHK}. 

\section{Cocycle deformation of relative Hopf modules and differential calculi} \label{31stjan251}

Let $(A, \Delta, S, \epsilon)$ be a Hopf algebra. We recall that $A\ot A$ is canonically
a coalgebra with coproduct $\Delta_{A\ot A}(h\ot k)=\one{h}\ot
\one{k}\ot \two{h}\ot \two {k}$ and counit $\epsilon_{A\otimes
	A}(h\otimes k)=\epsilon(h)\epsilon(k)$ for all $ h,k\in A$. If $\phi, \psi$ are linear maps from $A\ot A$ to $\bbC$, the convolution $\phi * \psi$ is defined as 
$$(\phi * \psi) (a\ot b )= \phi(\one{a}\ot \one{b}) \psi(\two{a}\ot \two{b}) $$
A linear functional $\phi$ on $A\ot A$ is said to be convolution invertible if there exists a linear functional $\psi$ on $A\ot A$ such that $\phi * \psi = \epsilon_{A\ot A}= \psi * \phi.$

\begin{definition} \label{18thjan252}
	A  convolution invertible map $\gamma:A \otimes A \ra \mathbb{C}$ is called a \textbf{$2$-cocycle} if	
	\begin{equation}
		\label{25thaug24} \gamma({{g}_{(1)}}\otimes {\one{h}}) \co{\two{g}\two{h}}{k} =  \co{\one{h}}{\one{k}} \co{g}{\two{h}\two{k}},   
	\end{equation}
	for all $g,h, k \in A $ and  it is said to be  unital if 
	$\gamma({h\otimes1})= \epsilon(h) = \gamma({1\otimes h})$ for all  $ h\in A $.
\end{definition}

\begin{remark}
	All cocycles in this article will be assumed to be unital unless otherwise mentioned.
\end{remark}

Throughout this article, $\overline{\gamma}$ will denote the convolution inverse of a $2$-cocycle $\gamma.$ 
We  recall the following result from the literature (for example see \cite[Lemma 2.16]{TwistPAschieriMain}).
\begin{lem}\label{lem:formula}
	Let   $\cot:A \otimes A \to \bbC$ be a  $\bbC$-linear map 
	with convolution  inverse $\bar{\cot}$. Let $g, h, k \in A.$ Then the
	following statements are equivalent:
	\begin{enumerate}
		\item
		$\cot$ is a cocycle, i.e, it satisfies \eqref{25thaug24},
		\item\label{ii}
		$\coin{\one{g}\one{h}}{k} \coin{\two{g}}{\two{h}}=  \coin{g}{\one{h}\one{k}} \coin{\two{h}}{\two{k}}\,,\;$ 
		\item\label{iii}
		$
		\co{\one{g}\one{h}}{\one{k}} \coin{\two{g}}{\two{h}\two{k}} =
		\coin{g}{\one{h}} \co{\two{h}}{k}\,,\;$ 
		\item\label{iv}
		$
		\co{\one{g}}{\one{h}\one{k}} \coin{\two{g}\two{h}}{\two{k}}=  \co{g}{\two{h}} \coin{\one{h}}{k}\,.$
	\end{enumerate}
\end{lem}

The following proposition recalls the deformation of Hopf algebras under a $2$-cocycle. 

\begin{prop} (\cite[Subsection 2.3]{SMFounds}) {\label{prop/defn:twisted-Hopf-algebra}}
	Suppose that $(A, \cdot, \Delta, \epsilon, S)$ is a Hopf  algebra and $\gamma$  a  $2$-cocycle, then we can define a Hopf algebra $(A_\gamma,\cdot_\gamma, \Delta, \epsilon, S_\gamma)$  such that 
	\begin{enumerate}
		\item The identity map defines a coalgebra isomorphism between  $(A_\gamma, \epsilon) $ and  $(A,\epsilon).$ 
		\item      
		$
		h\cdot_\gamma k:= \co{\one{h}}{\one{k}} \two{h}\two{k} \coin{\three{h}}{\three{k}},
		$
		\item  $S_{\gamma}(h):= U(\one{h})S(\two{h})\bar{U}(\three{h})$, where $U, \overline{U}: A \rightarrow \mathbb{C} $ are defined as  $U(k):= \co{\one{k}}{S(\two{k})},$ $ \bar{U}(k)= \coin{S(\one{k})}{\two{k}},$
        \end{enumerate}
        for all $h,k\in A.$
\end{prop}

The Hopf algebra $A_\gamma$ is called the $2$-cocycle twist of the Hopf algebra $A$.

\subsection{Cocycle deformation of relative Hopf modules} \label{sec:cocyledef_Rel}
If $\gamma$ is  a $2$-cocycle on a Hopf algebra $A,$ then the cocycle deformation of an object $(M, \del{M})$  of $\cat$  is the pair  $( \Gamma ( M ), \del{M_\gamma})$ where $\Gamma ( M ) = M $ as a vector space equipped with $B_\gamma$-bimodule structures given by
\begin{align*}
	b\cdot_\cot m = \co{\mone{b}}{\mone{m}}\zero{b} \cdot \zero{m},	\quad m\cdot_\cot b =\co{\mone{m}}{\mone{b}} \zero{m}\cdot \zero{b} 
\end{align*}
 and $\del{M_\gamma}= \del{M}$. Then $\tcat$ is a monoidal category with the monoidal structure to be denoted by $\ot_{B_\gamma}$.

In this case, it is well-known that the canonical functor 
\begin{equation} \label{21stnov243}
	\Gamma: \cat \to \tcat
\end{equation}
 defines a monoidal equivalence between the categories $(\cat, \ot_B)$ and $(\tcat, \ot_{B_\gamma}).$ The associated natural isomorphism  $\varphi$ between the functors $\ot_{B_\gamma} \circ(\Gamma\times \Gamma)$ and $\Gamma \circ \ot_B$  is given by
\begin{eqnarray}\label{nt}
	\varphi_{V,W}: \Gamma(V) \ot_{B_\gamma} \Gamma(W) &\longrightarrow&  \Gamma(V \ot_B W)  ~,
	\\
	v \ot_{B_\gamma} w &\longmapsto & \co{\mone{v}}{\mone{w}}  \zero{v} \ot_B \zero{w}  ~, \nonumber
\end{eqnarray}
for  objects $V,W$ in $\cat$. The inverse  $\varphi_{V,W}^{-1}$ is given by
\begin{equation} \label{eq:29thnov243}
	\varphi_{V,W}^{-1} (v \ot_B w)= \coin{\mone{v}}{\mone{w}} \zero{v} \ot_{B_\gamma} \zero{w}. 
\end{equation}

Let us also note that since $\varphi$ is a natural isomorphism, we have the equality
\begin{equation}\label{15thdec244}
	\left(\Gamma(f) \ot_{B_\cot} \Gamma(g)\right) \circ \varphi_{V_1,W_1}^{-1}= \varphi^{-1}_{V_2, W_2} \circ \Gamma(f \ot_B g)
\end{equation}
for morphisms $f \in \hm{}{V_1}{V_2},  g\in \hm{}{W_1}{W_2} $ in the category $\cat$.

If $B$ is a left $A$-comodule algebra,  then we will denote the $A_\gamma$-comodule $\Gamma ( B ) $ by the symbol $B_\gamma.$  Then from \cite[Subsection 2.3]{SMFounds}, we know that  $B_\gamma$ is an algebra via the formula 
$$ a\cdot_\cot b = \co{\mone{a}}{\mone{b}} \zero{a}\zero{b}, $$
 making $(B_\gamma, \del{B_\gamma})$  a left $A_\gamma$-comodule algebra.

Let us note that if we take $B = \mathbb{C} $ in \eqref{21stnov243}, we have a monoidal equivalence
\begin{equation} \label{16thNov241}
	\Gamma: \qMod{A}{}{}{} \to \qMod{A_\gamma}{}{}{}.
	\end{equation}

Let us also recall the well-known definition of cocycle deformation of morphisms in the category $\cat.$ If $\cE_1, \cE_2, \cF_1, \cF_2$ are objects of the category $\cat$ and $T: \cE_1 \ot_{B} \cE_2 \to \cF_1\ot_{B} \cF_2$ and $S: \cE_1 \ot_{B}\cE_2 \to B $ are morphisms, then we define 
$$ T_\gamma: \Gamma(\cE_1)\ot_{B_\cot}\Gamma(\cE_2) \to \Gamma(\cF_1)\ot_{B_\cot}\Gamma(\cF_2)  \text{ and } S_\gamma: \Gamma(\cE_1)\ot_{B_\cot}\Gamma(\cE_2) \to B_\gamma $$
as
\begin{equation} \label{1stdec241jb}
T_\gamma:= \varphi_{\cF_1, \cF_2}^{-1}\circ \Gamma(T) \circ \varphi_{\cE_1, \cE_2} \text{ and } S_\gamma := \Gamma(S) \circ \varphi_{\cE_1, \cE_2}.
\end{equation}
It follows that $T_\gamma$ and $S_\gamma$ are morphisms of $\tcat$.

\subsection{Cocycle deformation of Hopf  \texorpdfstring{$*$}{}-algebras}

For the purpose of dealing with $*$-differential calculi, we will need a compatibility between the $*$-structure and the cocycle. Beggs and Majid considered reality conditions 
$ \overline{\gamma ( a \otimes b )} = \gamma ( S^2 ( b )^\ast \otimes S^2 ( a )^\ast  ) $ (\cite[Subsection 2.3]{SMFounds}) as well as $ \overline{\gamma ( a \otimes b )} = \gamma ( a^\ast \otimes b^\ast ) $ (\cite[equation (8) ]{beggsmajidtwisting}). However, we will need a unitarity condition on the cocycle (instead of a reality condition) which has been studied in the operator algebra literature and in fact coincides with the definition of dual unitary $2$-cocycles on compact quantum groups (see Remark \ref{18thdec24jb1}).

\begin{definition} \label{18thjan253} (\cite[Definition 4.4]{SadeDeformSpecTrip})
	A $2$-cocycle $\gamma$ on a Hopf $*$-algebra is said to be unitary if 
	\begin{equation} 
		\label{19thSept20241} \overline{\co{a}{b}}= \coin{S(a)^*}{S(b)^*}.
	\end{equation}
\end{definition}

\noindent In this case we also have 
\begin{equation} \label{19thSept20242}
\overline{\coin{a}{b}}= \co{S(a)^*}{S(b)^*}
\end{equation}
 due to the identity $S(S(a^*)^*)=a$ (see \eqref{19thsept20245}).

Then we have the following result in which $U$ and $\overline{U}$ are as in Proposition \ref{prop/defn:twisted-Hopf-algebra}:

\begin{prop} (\cite[Subsection 4.2]{SadeDeformSpecTrip}) \label{10thmarch251}
If $(A, \cdot, \Delta, \epsilon, S, *)$ is a Hopf $*$-algebra and $\gamma$  a unitary $2$-cocycle on $A,$ consider the maps $V, \bar{V}: A \rightarrow \mathbb{C}$ as
        \begin{equation*}  \label{19thsept20244}
V(k):= U(S^{-1}(k))= \co{S^{-1}(\two{k})}{\one{k}}, ~ \bar{V}(k):= \bar{U}(S^{-1}(k))= \coin{\two{k}}{S^{-1}(\one{k})}.
        \end{equation*}
Then $\ast_\gamma: A \rightarrow A$ defined as 
		 $$h^{*_{\gamma}}:= \bar{V}({\one{h}^*}){\two{h}^*}V({\three{h}^*})$$
         is an antilinear involution on $A$ and moreover $(A_\gamma,\cdot_\gamma, \Delta, \epsilon,  S_\gamma, *_\gamma )$ is a Hopf $\ast$-algebra.

 If $B$ is a left $A$-comodule $\ast$-algebra, then the algebra $B_\gamma$ is equipped with a $*$-algebra structure given by
\begin{equation} \label{4thfeb251}
 b^{*_\cot} = \bar{V}(\mone{b}^*) ~ \zero{b}^*
\end{equation}
making $(B_\gamma, \del{B_\gamma})$  a left $A_\gamma$-comodule $*$-algebra.
\end{prop}
 We should mention  that \cite[Definition 4.7]{SadeDeformSpecTrip} is for right comodule $\ast$-algebras while we are working with left comodule structures. The following observation was made in \cite[Subsection 4.2]{SadeDeformSpecTrip}:

\begin{remark} \label{25thnov241}
	If $\gamma$ is any cocycle, then the maps $\bar{U}$ and $\bar{V}$ are (respectively) the convolution inverses of $U$ and $V.$ This can be checked by using the equations  \eqref{iii} and \eqref{iv} of Lemma \ref{lem:formula}. 
\end{remark}

\begin{remark} \label{9thsep251}
If $A$ is a cocommutative Hopf $\ast$-algebra, then it is well-known that $A$ and $A_\gamma$ are isomorphic Hopf $\ast$-algebras.   Indeed, since $h_{(1)} \otimes h_{(2)} = h_{(2)} \otimes h_{(1)}, $ the formula for the deformed product in Proposition \ref{prop/defn:twisted-Hopf-algebra} gives 
\begin{align*}
 h\cdot_\gamma k &= \co{\one{h}}{\one{k}} \two{h}\two{k} \coin{\three{h}}{\three{k}} = \one{h} \one{k} \co{\two{h}}{\two{k}} \coin{\three{h}}{\three{k}}\\
 &= \one{h} \one{k} (\gamma \ast \overline{\gamma}) (\two{h} \otimes \two{k}) = \one{h} \one{k} \epsilon (\two{h}) \epsilon (\two{k}) = hk.
 \end{align*}

Similarly, 
$$ h^{\ast_\gamma} = \overline{V} ( h^\ast_{(1)} ) h^*_{(2)}  V ( h^\ast_{(3)} ) = h^*_{(1)} \overline{V} ( h^\ast_{(2)} ) V ( h^\ast_{(3)} )  =  h^\ast  $$
since $\overline{V}$ is the convolution inverse of $V.$ 
\end{remark}

If $B$ is a left $A$-comodule $\ast$-algebra and $\gamma$ a unitary $2$-cocycle, then Theorem \ref{theorem:21stnov242} proves that the monoidal equivalence $\Gamma$ of  \eqref{21stnov243} is actually a bar functor. We refer to Definition \ref{3rdmarch251} for the definition.  

We record the following observation which will be used in the sequel:
\begin{lem} \label{lem:22ndnov241}
	Suppose that $A$ is a Hopf $*$-algebra and $\gamma$ is a unitary  $2$-cocycle then $\overline{\overline{V}(h^*)}=V(h)$.
\end{lem}
\begin{proof}
	We apply \eqref{19thsept20245} and \eqref{19thSept20242} to compute
    
	$
		\overline{\overline{V}(h^*)}= \overline{\coin{\two{h}^*}{S^{-1}(\one{h}^*)}}= \co{S(\two{h}^*)^*}{(\one{h}^*)^*}= \co{S^{-1}(\two{h})}{\one{h}}= V(h).
	$
\end{proof}

\begin{lem} \label{lem:28thnov242}
	If $A$ is a Hopf $*$-algebra and $\gamma$ is a two cocycle on $A$ with convolution inverse $\bar{\gamma},$ then the following equation holds:
	\begin{equation} \label{19thsept20243}
		\bar{V}(\one{k}^*)\bar{V}(\one{h}^*) \co{\two{k}^*}{\two{h}^*} =\coin{S(\one{h})^*}{S(\one{k})^*}~ \bar{V}(\two{k}^*\two{h}^*). 
    \end{equation}
\end{lem}

\begin{cor}
If $\gamma$ is a  $2$-cocycle on $A,$ then for all $h, k \in A,$ the following equation holds:
\begin{equation} \label{3rdfeb253}
\co{S(\one{h})^*}{S(\one{k})^*}
 \bar{V} ( k^\ast_{(2)} ) \bar{V} ( h^\ast_{(2)} ) = \bar{V} ( k^\ast_{(1)} h^\ast_{(1)} ) 
 \coin{ k^\ast_{(2)}}{h^\ast_{(2)}}
\end{equation}
\end{cor}

\subsection{Twisting Covariant Differential Calculi}

Following \cite{beggsmajidtwisting}, we recall the cocycle deformation of an $A$-covariant differential calculus $\dc$ over an algebra $B.$

In this case, $\Omega^k(B)$ is an object of $\cat$ and $d: \Omega^k(B) \to \Omega^{k+1}(B)$ is a morphism of $\qMod{A}{}{}{}$, and therefore, $\Gamma(\Omega^k(B))$ is an object of $\tcat$ and $d_\gamma:= \Gamma ( d ): \Gamma(\Omega^k(B)) \to \Gamma(\Omega^{k+1}(B))$ is an $A_\gamma$-covariant map. Let us define 
$$ \Omega^k(B_\gamma):= \Gamma(\Omega^k(B)), ~ \Omega^\bullet(B_\gamma):= \oplus_{k \geq 0}\Omega^k(B_\gamma).$$ 
Moreover, we define
$$  \wedge_\gamma: \Omega^k(B_\gamma)\ot_{B_\cot} \Omega^l(B_\gamma)\to \Omega^{k+l}(B_\gamma) $$
by \eqref{1stdec241jb}. It follows that for all $\omega \in \Omega^k(B_\gamma), \eta \in \Omega^l(B_\gamma),$
\begin{align}\label{29thnov24jb1}
	\omega\wedge_\cot\eta = \co{\mone{\omega}}{\mone{\eta}}\zero{\omega}\wedge \zero{\eta}.
\end{align}

From now on, we will use a shorthand notation to denote $\Omega^k ( B_\gamma) $ and $\Omega^\bullet ( B_\gamma) $ by the symbols $\Omega^k_\gamma $ and $\Omega^\bullet_\gamma $ respectively. 

Since $(\Omega^\bullet, \wedge)$ is a left $A$-comodule $*$-algebra, we have a left $A_\gamma$-comodule $*_\gamma$-algebra $(\Omega^\bullet_\gamma, \wedge_\gamma).$ Then we have the following result:

\begin{prop} (\cite[Section 5]{beggsmajidtwisting}) \label{prop:5thdec241}
	Let $A$ be a Hopf-algebra and $(\Omega^\bullet, \wedge, d)$  an $A$-covariant differential calculus on an $A$-comodule $B.$ If $\gamma$ is a $2$-cocycle on $A,$ then  $(\Omega^\bullet_\gamma, \wedge_\cot, d_\gamma)$ is an $A_\gamma$-covariant differential calculus on $B_\gamma$.
	
	Moreover, if $A$ is a Hopf $\ast$-algebra, $B$ an $A$-covariant $\ast$-algebra, $(\Omega^\bullet, \wedge, d)$ a $\ast$-calculus and $\gamma$ unitary, then
	$(\Omega^\bullet_\gamma, \wedge_\cot, d_\gamma)$ is an $A_\gamma$-covariant $*$-differential calculus.
\end{prop}

\begin{proof}
  Since the first statement is proved in \cite{beggsmajidtwisting}, we are left to prove the second assertion for which   we assume that $ ( \Omega^\bullet, \wedge, d  ) $ is a $\ast$-differential calculus and $\gamma$ is a unitary $2$-cocycle.
  
  The containment $*_\gamma(\Omega_\gamma^k(B_\gamma)) \subseteq \Omega_\gamma^k( B_\cot)$ follows from \eqref{4thfeb251}.  Next, for $\omega \in \Omega_\gamma^k,$	we have 
	$$d_\gamma(\omega^{*_\gamma})= \bar{V}(\mone{\omega}^*)~d(\zero{\omega}^*)= \bar{V}(\mone{\omega}^*)~(d\zero{\omega})^*=(d_\gamma\omega)^{*_\gamma} $$ by using $A$-covariance of $d$.  Finally, if $ \omega \in \Omega^k_\gamma,~\eta \in \Omega_\gamma^l$, we get
	\begin{align*}
		(\omega\wedge_\gamma \eta)^{*_\gamma} & = \overline{\co{\mone{\omega}}{\mone{\eta}}}(\zero{\omega}\wedge \zero{\eta})^{*_\gamma}\\
		&= {\coin{S(\mtwo{\omega})^*}{S(\mtwo{\eta})^*}}\bar{V}(\mone{\eta}^* \mone{\omega}^*)(\zero{\omega}\wedge \zero{\eta})^{*} \text{ (by \eqref{19thSept20241})}\\
		&=\bar{V}(\mtwo{\omega}^*) \bar{V}(\mtwo{\eta}^*) \co{\mone{\eta}^*}{\mone{\omega}^*} (-1)^{kl} \zero{\eta}^*\wedge \zero{\omega}^*\quad\text{(by using \eqref{19thsept20243})}\\
		&= \bar{V}(\mone{\omega}^*) \bar{V}(\mone{\eta}^*)(-1)^{kl} \zero{\eta}^*\wedge_\gamma \zero{\omega}^*\\
		&= (-1)^{kl} {\eta^{*_\gamma}}\wedge_{\gamma} {\omega^{*_\gamma}}.
	\end{align*}
	Hence, $\dctwisted$ is a covariant $*_\gamma$-differential calculus.
\end{proof}

It is well-known that if $\gamma$ is a cocycle on $A,$ then its convolution inverse $\bar{\gamma}$ is a cocycle on $A_\gamma.$ Moreover, if $M$ is an object of $\cat,$ then the $\overline{\gamma}$-deformation of  $  \Gamma ( M ) $ is equal to  $   M. $ In fact, more is true:

\begin{remark} \label{27thdec241jb}
If $\dc$ is an $A$-covariant differential calculus on $B,$ then the $\bar{\gamma}$-deformation of $\dctwisted$ is $\dc.$
\end{remark}

We end this section by recalling the cocycle twisting of connections. If $\nabla_{\cE}$ is an $A$-covariant left connection on an object $\cE$ in $\cat$, then we define a $\bbC$-linear map $\nabla_{\Gamma(\cE)}: \Gamma(\cE) \to \Omega^1_\gamma\ot_{B_\cot} \Gamma(\cE)$ as 
\begin{equation}
	\label{eq:3rddec241} \nabla_{\Gamma(\cE)}:= \varphi_{  \Omega^1, \cE}^{-1} \circ \Gamma(\nabla_\cE),
\end{equation} 
where $\Gamma(\nabla_\cE)$ is  as in \eqref{16thNov241}. Then we have the following proposition:

\begin{prop} (\cite[Proposition 9.28]{BeggsMajid:Leabh}) \label{prop:29thnov242}
	Suppose that $(\Omega^\bullet, \wedge, d)$ is an $A$-covariant differential calculus on a left $A$-comodule algebra $B$ and $\nabla_\cE$ an $A$-covariant connection on an object $\cE$ in $\cat$.  Then $\nabla_{\Gamma(\cE)}$ is an $A_\gamma$-covariant  left connection on $\Gamma(\cE)$. 
    
    Furthermore, if $\sigma: \cE\ot_{B} \Omega^1 \to \Omega^1 \ot_{B} \cE$ is a morphism in $\cat$ such that  $(\nabla_\cE, \sigma)$ is a left $A$-covariant bimodule connection on $\cE$, then $(\nabla_{\Gamma(\cE)}, \sigma_\gamma)$ is a  left $A_\gamma$-covariant bimodule connection on $\Gamma(\cE)$, where $\sigma_\gamma$ is the morphism in $\tcat$  defined in \eqref{1stdec241jb}.
	\end{prop}

\section{Hermitian metrics and their cocycle deformations} \label{8thmarch252}

In this section, we  prove that covariant  Hermitian metrics on an object $\cE$ in $\cat$ can be deformed to Hermitian metrics on the object $\Gamma ( \cE ) $ in $\tcat.$ In the next section, we will apply this fact to obtain Chern connections on the twisted holomorphic bimodules $\Omega^{(1,0)}_\gamma$ and $\Omega^{(0, 1)}_\gamma.$ Hermitian metrics are defined for bimodules over a $\ast$-algebra. In order to deform these, we will require a compatibility between the $\ast$-structure and the cocycles and therefore, we will restrict ourselves to unitary cocycles (see Definition \ref{18thjan253}).

For the rest of this article, we will need the language of bar-categories. We refer to subsection \ref{2ndfeb261} for the definitions and other necessary details. In particular, for a vector space $M$, the symbol $\ol{M}$ will denote the vector space defined as 
$$ \overline{M}:= \{ \overline{m}: m \in M \}.$$
Throughout this section, $A$ will denote a Hopf $\ast$-algebra and $B$ a left $A$-comodule $\ast$-algebra. Then the monoidal category $\cat$ is a bar category (see Example \ref{5thdec241jb}). In particular,  if $M$ is an object of $\cat$, then $\ol{M}$ is a $B$-bimodule by \eqref{28thnov231} and a left $A$-comodule by \eqref{25thnov232}.

We start with the definition of an Hermitian metric, where,
$$ \ev: \cE \otimes_B \hm{B}{\cE}{B} \rightarrow B ~ \text{is defined as} ~ \ev ( e \otimes_B f ) = f ( e ) $$
for all $e \in \cE$ and for all $f \in \hm{B}{\cE}{B}.$

\begin{definition} (\cite[Definition 8.33]{BeggsMajid:Leabh}) \label{4thdec231}
	Let $\cE$ be an object of $\cat$ such that $\cE$ is finitely generated and projective as a left $B$-module. An $A$-covariant Hermitian metric on  $\cE$   is an isomorphism $ \kH: \overline{\cE} \to \hm{B}{\cE}{B}$  in the category $\cat$ such that 
	\begin{equation}\label{15thdec246}
		\langle y,\overline{x}\rangle^*=\langle x, \overline{y}\rangle \quad \text{for all } y, x\in \cE,
	\end{equation}
	where the map 
	\begin{equation} \label{metric:Hermitian}
		\langle~,~\rangle: \cE \otimes_B \overline{\cE} \to B \quad  \text{ is defined as }  \langle x, \overline{y}\rangle:= \ev(x \otimes_B \kH(\overline{y})).
	\end{equation} 
\end{definition}
We note that $\langle~, ~ \rangle$ is a morphism in the category $\cat$ as $\kH$ and $\ev$ are morphisms in $\cat$.

\subsection{Cocycle twists of Hermitian metric}

Let $\kH$ be an Hermitian metric on an object $\mathcal{E}$ in the category $\cat$ which is moreover
 finitely generated and projective as a left $B$-module. Thus, $\kH$ is an isomorphism from $\overline{\mathcal{E}}$ to $\hm{B}{\cE}{B}$ satisfying \eqref{15thdec246}. In the presence of a unitary $2$-cocycle $\gamma$ on $A,$ we want to construct an isomorphism 
$$ \kH_{\gamma}: \overline{\Gamma ( \mathcal{E} )} \rightarrow \hm{B_\gamma}{ \Gamma ( \mathcal{E} )} {B_\gamma }. $$
To this end, we will need to introduce two isomorphisms. 

If $\cE$ and $\cF$ are objects of $\cat$ with $\cE$ being finitely generated as a left $B$-module and $\gamma$ is a $2$-cocycle on $A,$ then by Proposition \ref{28thnov241jb}, $\hm{B_\cot}{\Gamma(\cE)}{\Gamma(\cF)}$ is an object of the category $\tcat$.
The authors of \cite{TwistPAschieriMain} defined the map 
$$\mathfrak{S}: \Gamma (\hm{B}{\mathcal{E}}{\mathcal{F}}) \to \hm{B_\gamma}{\Gamma(\mathcal{E})}{\Gamma(\mathcal{F})}  ~ {\rm as} $$
\begin{align}\label{fS}
\fS(f)(v)&=\co{\mtwo{v}}{S(\mone{v})\mone{[f(\zero{v})]}}
\zero{[f(\zero{v})]}\\
\label{eq:28thnov241} &= U(\mtwo{v}) \coin{S(\mone{v})}{\mone{[f(\zero{v})]}}\zero{[f(\zero{v})]}.
\end{align}
 The equality of \eqref{fS} and \eqref{eq:28thnov241} was observed in the proof of \cite[Proposition 3.17]{TwistPAschieriMain} and moreover, the same result shows that $\mathfrak{S}$ is a vector space isomorphism. In fact, we have that  $\fS$ is actually an isomorphism in $\tcat$.

\begin{lem} \label{26thsept20241}
Let $\cE$ be an object of $\cat$ such that $\cE$ is finitely generated and projective as a left $B$-module. The map $\mathfrak{S}: \Gamma (\hm{B}{\mathcal{E}}{\mathcal{F}}) \to \hm{B_\gamma}{\Gamma(\mathcal{E})}{\Gamma(\mathcal{F})}$ defined in \eqref{fS}
is an isomorphism in $\tcat$.
\end{lem}
\begin{proof}
See Lemma \ref{14thfeb252}.
\end{proof}

Now, if $A$ is a Hopf $\ast$-algebra, $B$ a left $A$ comodule $\ast$-algebra and $\gamma$ a unitary $2$-cocycle, we define

\begin{equation} \label{5thdec242jb}
\mathfrak{N}_\cE : \overline{\Gamma(\mathcal{E})} \to \Gamma(\overline{\mathcal{E}}),~ \fN_\cE(\overline{e})= \bar{V}(\mone{e^*}) \overline{\zero{e}}.
\end{equation}

Then we have the following result which has been proved in the appendix (see Lemma \ref{14thfeb251}).

\begin{lem} \label{8thoct2024}
 $\fN : \mathrm{bar}\circ \Gamma \to \Gamma \circ \mathrm{bar}$  
is a natural isomorphism.
\end{lem}

Now we are in a position to prove the main theorem of this section:

\begin{theorem} \label{thm:twisted-hermitian}
	Suppose that $A$ is a Hopf $*$-algebra, $B$ an $A$-comodule $*$-algebra and $\cE$ an object in the category $\cat$ which is finitely generated and projective as a left $B$-module.  If $\kH: \overline{\cE} \to \hm{B}{\cE}{B}$ is an $A$-covariant Hermitian metric on $\cE,$  then for a unitary $2$-cocycle $\gamma$ on $A,$ the map  
    $$\kH_\gamma: \overline{\Gamma{(\cE)}} \to \hm{~B_\gamma}{\Gamma(\cE)}{B_\gamma} ~ {\rm defined} ~ {\rm by} ~ \kH_\gamma  = \fS \circ \Gamma( \kH )\circ \fN_\cE$$
    is an $A_\gamma$-covariant Hermitian metric on $\Gamma{(\cE)}.$
\end{theorem}
\begin{proof}
	Since $\fS$ and $\fN_\cE$ are isomorphims in $\tcat,$ it follows that $\kH_\gamma$ is an isomorphism. Consider the map $\langle~,~\rangle_\gamma: \Gamma(\cE) \ot_{B_\cot} \overline{\Gamma(\cE)} \to B_\cot$, defined by $$  \langle x , \overline{y}\rangle_\gamma= \ev \big(x \ot_{B_\cot} \kH_\cot \left(\overline{y}\right)\big).$$ 
	Then for $x, y \in \cE,$ we get
	\begin{align*}
		\langle x,\overline{y} \rangle_\gamma &= \ev\left(x \ot_{B_\cot} \bar{V}(\mone{y}^*) \fS \circ \Gamma (\kH) (\overline{\zero{y}}) \right)\\
		&= \bar{V}(\mone{y}^*) ~ \fS \left( \kH(\ol{\zero{y}}) \right) (x) \\
		&= \bar{V}(\mone{y}^*) \co{\mtwo{x}}{S(\mone{x})\mone{(\kH(\overline{\zero{y}})(\zero{x}))}}\zero{(\kH(\overline{\zero{y}})(\zero{x}))} \text{ (by \eqref{fS})}\\
		&= \bar{V} ( \mone{y}^\ast ) \coin{\mtwo{x}}{S ( \mone{x} )\mone{\langle \zero{x}, \ol{\zero{y}}\rangle}} \zero{\langle \zero{x}, \ol{\zero{y}}\rangle}  \\
        &= \bar{V}(\mtwo{y}^*) \co{\mthree{x}}{S(\mtwo{x})\mone{x}\mone{y}^*}\langle\zero{x}, \ol{\zero{y}} \rangle \text{ (as $\langle ~, ~ \rangle$ is covariant)}\\
		&= \bar{V}(\mtwo{y}^*) \co{\mone{x}}{\mone{y}^*}\langle\zero{x}, \ol{\zero{y}} \rangle. 
	\end{align*}
	Hence we have that 
	\begin{equation} \label{eq:relation}
		\langle x,\overline{y} \rangle_\gamma = \bar{V}(\mtwo{y}^*) \co{\mone{x}}{\mone{y}^*}\langle\zero{x}, \ol{\zero{y}} \rangle.
	\end{equation}
	So, for $ x, y \in \cE$, we have,
	\begin{align*}
		&\langle x,\overline{y} \rangle_\gamma^{*_\cot}\\ 
		&= \big(\bar{V}(\mtwo{y}^*) \co{\mone{x}}{\mone{y}^*}\langle\zero{x}, \ol{\zero{y}} \rangle \big)^{*_\gamma} \text{ (by \eqref{eq:relation})}\\
		&= \overline{\bar{V}(\mthree{y}^*) \co{\mtwo{x}}{\mtwo{y}^*}} ~\bar{V}((\mone{x}\mone{y}^*)^*) \langle\zero{x}, \ol{\zero{y}} \rangle^* \text{(as $\langle~,~\rangle$ is covariant and by \eqref{4thfeb251})}\\
		&= {V}(\mthree{y}) \coin{S(\mtwo{x})^*}{S(\mtwo{y}^*)^*}\bar{V}(\mone{y}\mone{x}^*) \langle\zero{y}, \ol{\zero{x}} \rangle \text{(by Lemma \ref{lem:22ndnov241}, \eqref{19thSept20241} and \eqref{15thdec246})}\\
		& 
		={V}(\mthree{y}) \bar{V}(\mtwo{y}) \bar{V}(\mtwo{x}^*) \co{\mone{y}}{\mone{x}^*}\langle\zero{y}, \ol{\zero{x}} \rangle \text{ (by \eqref{19thsept20243})}\\
		&= \bar{V}(\mtwo{x}^*) \co{\mone{y}}{\mone{x}^*} \langle\zero{y},\ol{\zero{x}}\rangle \text{ (by Remark \ref{25thnov241})}\\
		&= \langle y, \ol{x}\rangle_\gamma \text{ (by \eqref{eq:relation})}.
	\end{align*}  
	Therefore, $\kH_\gamma$ satisfies \eqref{15thdec246} which completes the proof.
\end{proof}
\begin{remark} \label{03rdoct242} 
  \begin{enumerate}
  \item We will call $\kH_\gamma$ the $\gamma$-deformation of the Hermitian metric $\kH.$
  
	\item The equation  \eqref{eq:relation} implies the following relation:  
    \begin{equation} \label{18thfeb251}
    \langle~, ~\rangle_\cot = \Gamma(\langle~,~\rangle) \varphi_{\cE, \overline{\cE}}(\id \ot_{B_\gamma}\fN_\cE).
    \end{equation}
    \end{enumerate}
\end{remark}

\subsection{Hermitian metrics on the bimodule of one-forms}

In this subsection, we deal with the deformation of covariant Hermitian metrics on the bimodule of one-forms of a $\ast$-differential calculus by relating it with the deformation of covariant real metrics (see Definition \ref{23rddec24jb2}) on one-forms. n this subsection, we will tacitly use the fact (see Remark \ref{24thjuly241}) that if $\Omega^1$ admits a metric $\metric,$ then it is finitely generated and projective as a left module. 

The following result follows from \cite[Section 8.4]{BeggsMajid:Leabh} (also see \cite[Proposition 4.3]{LeviCivitaHK}):

\begin{prop}  \label{27thdec242jb}
	Let $\dc$ be an $A$-covariant $\ast$-differential calculus on a left $A$-comodule $\ast$-algebra $B.$  Given a real $A$-covariant metric $ ( g, ( ~, ~ )  ) $ on $\Omega^1,$ define
	$$ \kH_g: \overline{\Omega^1} \rightarrow \hm{B}{\Omega^1}{B} ~ {\rm as} ~ \kH_g ( \overline{\omega}  ) ( \eta ) = ( \eta, \omega^* ). $$
	Then $\kH_g$ is an $A$-covariant Hermitian metric on $\Omega^1.$ In fact,  $ ( g, ( ~ , ~ ) ) \mapsto \kH_g $ is a bijective correspondence between real $A$-covariant metrics and $A$-covariant Hermitian metrics on $\Omega^1.$
\end{prop}

In fact, we have that 
\begin{equation} \label{29thjan251}
 \langle \omega,  \overline{\eta} \rangle = ( \omega, \eta^\ast ) ~  {\rm for} ~ {\rm all} ~ \omega, \eta \in \Omega^1.
\end{equation}

We will need to recall the cocycle deformation of covariant metrics from \cite{BeggsMajid:Leabh}.
Let $\dc$ be an $A$-covariant differential calculus on a left $A$-comodule algebra $B$ and  $\metric$ an $A$-covariant metric. In particular, by virtue of Remark \ref{24thjuly241}, $g= \coev_g(1)$ for a morphism $\coev_g: B \to \Omega^1 \ot_{B} \Omega^1$ in $\cat$ such that $(\Omega^1, (~,~), \coev_g)$ is a right dual of $\Omega^1$ in $\cat$. Now if $\gamma$ is a $2$-cocycle on $A,$ consider the maps
\begin{equation}
	\label{1stdec242jb}(~,~)_\gamma: \Omega^1_\gamma \ot_{B_\cot} \Omega^1_\gamma\to B_\gamma,~ (\coev_g)_\gamma: B_\gamma \to \Omega^1_\gamma\ot_{B_\cot}  \Omega_\gamma^1
\end{equation}
defined by \eqref{1stdec241jb} and $g_\gamma= (\coev_g)_\gamma(1).$ 

Since $g_\gamma$ is an $A_\gamma$-coinvariant central element of $\Omega^1_\gamma \otimes_{B_\gamma} \Omega^1_\gamma,$ it follows that there is a morphism 
$$\text{coev}_{g_\gamma}: B_\gamma \rightarrow \Omega^1_\gamma \otimes_{B_\gamma} \Omega^1_\gamma ~ \text{in} ~ \tcat ~ \text{defined by} ~ \text{coev}_{g_\gamma} ( b ) = b g_\gamma.  $$

\begin{prop} (\cite[Proposition 9.28]{BeggsMajid:Leabh}) \label{prop:15thdec241}
	If $(\Omega^\bullet, \wedge, d)$ is an $A$-covariant differential calculus on a left $A$-comodule algebra $B$
	and $\metric$  an $A$-covariant metric on  $\Omega^1(B),$  then $\metrictwisted$ is an $A_\gamma$-covariant metric on $\Omega^1(B_\gamma).$ 

    Moreover, the $\overline{\gamma}$-deformation of $\metrictwisted$ is $\metric$ and hence  $\metric \mapsto \metrictwisted $ defines a one to one correspondence between $A$-covariant metrics on $\Omega^1$ and $A_\gamma$-covariant metrics on $\Omega^1_\gamma. $
\end{prop}

In order to state our results, we make the following observation:

\begin{prop} \label{realmetric}
	If $ \metric $ is a real $A$-covariant metric on $\Omega^1 ( B ) $ and $\gamma$  a unitary  $2$-cocycle on $A,$ then the cocycle-deformed metric $ \metrictwisted $ of Proposition \ref{prop:15thdec241} for the differential calculus $\dctwisted$ is also real.  
\end{prop}
\begin{proof}
We will use the notation $\dagger_\gamma$ to denote the antilinear map ${\rm flip} ( \ast_\gamma \otimes_{B_\gamma} \ast_\gamma ): \Omega^1_\gamma \otimes_{B_\gamma} \Omega^1_\gamma \rightarrow \Omega^1_\gamma \otimes_{B_\gamma} \Omega^1_\gamma. $  Thus, we need to prove that $g^{\dagger_\gamma}_\gamma = g_\gamma.$ 

Let us use a shorthand notation to write 
$$g = g^{[1]} \otimes_B g^{[2]}, ~ \text{where summation is understood}.$$
Applying the antilinear map $\ast \otimes \dagger $ to the equation $ {}^{\Omega^1 \otimes_B \Omega^1} \delta ( g ) = 1 \otimes g $ (see \eqref{18thfeb252}) we get
\begin{equation} \label{3rdfeb252}
 g^{[2]\ast}_{(-1)} g^{[1]\ast}_{(-1)} \otimes (  g^{[2]\ast}_{(0)} \otimes_B g^{[1]\ast}_{(0)}  ) = 1 \otimes g^\dagger.
 \end{equation}
    
   We claim that the following equation holds for all $\omega, \eta \in \Omega^1:$
\begin{equation} \label{3rdfeb251}
\dagger_\gamma ( \varphi^{-1}_{\Omega^1, \Omega^1} ( \omega \otimes_B \eta ) ) = \varphi^{-1}_{\Omega^1, \Omega^1} ( \eta^*_{(0)} \otimes_B \omega^\ast_{(0)}  ) \bar{V} ( \eta^\ast_{(-1)} \omega^\ast_{(-1)} ).
\end{equation} 
 The proposition follows from \eqref{3rdfeb251} and \eqref{3rdfeb252}. Indeed, it is easy to see that $g_\gamma = \varphi^{-1}_{\Omega^1, \Omega^1} ( g ) $ and hence, 
 \begin{eqnarray*}
g^{\dagger_\gamma}_\gamma &=& \dagger_\gamma ( \varphi^{-1}_{\Omega^1, \Omega^1} ( g ) )\\
&=& \varphi^{-1}_{\Omega^1, \Omega^1} (   g^{[2]\ast}_{(0)} \otimes_B g^{[1]\ast}_{(0)} ~ \bar{V} ( (  g^{[1]}_{(-1)} g^{[2]}_{(-1)}  )^\ast  )  ) ~ {\rm (}  {\rm by} ~ \eqref{3rdfeb251} {\rm )}\\
&=& \varphi^{-1}_{\Omega^1, \Omega^1} ( g^\dagger ) ~ {\rm (} ~ {\rm by} ~ \eqref{3rdfeb252} {\rm )}  \\
&=& \varphi^{-1}_{\Omega^1, \Omega^1} ( g )\\
&=& g_\gamma.
 \end{eqnarray*}
 Thus, we are left to prove \eqref{3rdfeb251}. To this end, we compute:
 \begin{align*}
 \dagger_\gamma ( \varphi^{-1}_{\Omega^1, \Omega^1} ( \omega \otimes_B \eta ) ) &= \dagger_\gamma ( \bar{\gamma} ( \omega_{(-1)} \otimes \eta_{(-1)}  ) \omega_{(0)} \otimes_{B_\gamma} \eta_{(0)} ) ~ {\rm (}  {\rm by} ~ \eqref{eq:29thnov243}  {\rm )} \\
&= \gamma ( S ( \omega_{(-1)} )^\ast \otimes S ( \eta_{(-1)} )^\ast  ) \eta^{\ast_\gamma}_{(0)} \otimes_{B_\gamma} \omega^{\ast_\gamma}_{(0)} ~ {\rm (}  {\rm by} ~ \eqref{19thSept20242} {\rm)}\\
&= \gamma ( S ( \omega_{(-2)} )^\ast \otimes S ( \eta_{(-2)} )^\ast  ) \bar{V} ( \omega^\ast_{(-1)} ) \bar{V} ( \eta^\ast_{(-1)} ) \eta^{\ast}_{(0)} \otimes_{B_\gamma} \omega^{\ast}_{(0)} ~ {\rm (} {\rm by} ~ \eqref{4thfeb251}  {\rm )}\\
&= \bar{\gamma} ( \eta^\ast_{(-1)} \otimes \omega^\ast_{(-1)}  ) \bar{V} ( \eta^\ast_{(-2)} \omega^\ast_{(-2)} ) \eta^\ast_{(0)} \otimes_{B_\gamma} \omega^\ast_{(0)} ~ {\rm (}  {\rm by} ~ \eqref{3rdfeb253}  {\rm )} \\
&= \varphi^{-1}_{\Omega^1, \Omega^1} ( \eta^*_{(0)} \otimes_B \omega^\ast_{(0)}  ) \bar{V} ( \eta^\ast_{(-1)} \omega^\ast_{(-1)} ).
 \end{align*}
  This proves \eqref{3rdfeb251} and hence  the proposition.  
\end{proof}

We will continue to denote the  $\gamma$-deformation of a Hermitian metric $\kH$ introduced in Theorem \ref{thm:twisted-hermitian} by the symbol ${\kH}_\gamma.$  

By a combination of Proposition \ref{27thdec242jb} and Proposition \ref{realmetric}, we can construct an $A_\gamma$-covariant Hermitian metric on $\Omega^1 ( B_\gamma ) $ from a Hermitian metric on $\Omega^1 ( B ). $ Indeed,  Proposition \ref{27thdec242jb} implies that any covariant Hermitian metric  on $\Omega^1 ( B ) $ is of the form $ {\kH}_g $ for a real $A$-covariant metric $\metric$ on $\Omega^1 ( B ). $ By Proposition \ref{realmetric}, we know that $\metrictwisted$ is a real $A_\gamma$-covariant metric on $\Omega^1 ( B_\gamma ). $ Therefore, Proposition  \ref{27thdec242jb} applied to $\metrictwisted$ yields an $A_\gamma$-covariant Hermitian metric ${\kH}_{g_\gamma}$ on $\Omega^1 ( B_\gamma ). $ Our next result states that  ${\kH}_{g_\gamma}$ is precisely the $\gamma$-deformation of $ {\kH}_g $ for which we will use the fact that $\star_{\Omega^1}: \Omega^1 \rightarrow \overline{\Omega^1} $ defined in \eqref{8thjuly241} is an isomorphism in $\cat.$

\begin{theorem} \label{5thfeb253}
	If $\metric$ is a real $A$-covariant metric on $\Omega^1 ( B ) $ and $\gamma$ a unitary $2$-cocycle, then $ {\kH}_{g_\gamma} = ( {\kH}_g )_\gamma. $
\end{theorem}
\begin{proof}
	    Let $ \langle ~ , ~ \rangle, \langle ~ , ~ \rangle^\prime $ and $\langle ~ , ~ \rangle_\gamma$ be the morphisms associated to the Hermitian metrics $\kH_g, \kH_{g_\gamma}$ and $ ( \kH_g )_\gamma $ respectively. Then by   \eqref{18thfeb251}, we have that,   
	\begin{align*}
		\langle~,~\rangle_\gamma &= \Gamma(\langle~,~\rangle)\varphi_{\Omega^1, \overline{\Omega^1}} (\id \ot_{B_\cot} \fN_{\Omega^{1}})\\
		&= \Gamma \big( (~, ~)\big) \Gamma\big((\id \ot_{B} \star_{\Omega^1}^{-1}) \big) \varphi_{\Omega^1, \overline{\Omega^1}} (\id \ot_{B_\cot} \fN_{\Omega^{1}}) ~ {\rm (}  {\rm by} ~ \eqref{29thjan251}  {\rm )}\\
		&= \Gamma \big( (~, ~)\big) \varphi_{\Omega^1, \Omega^1} \big(\Gamma (\id) \ot_{B_\gamma}\Gamma( \star_{\Omega^1}^{-1}) \big)  (\id \ot_{B_\cot} \fN_{\Omega^{1}})  \text{ (by \eqref{15thdec244})}\\
		&=  (~, ~)_\gamma  \big(\id \ot_{B_\gamma}\Gamma( \star_{\Omega^1}^{-1})  \fN_{\Omega^{1}}\big) \text{ (by the definition of ~} (~, ~)_\gamma \text{ in  \eqref{1stdec242jb})}\\
		&=  (~, ~)_\gamma  \big(\id \ot_{B_\gamma} ( \star_{\Omega_\gamma^1} )^{-1}) \big) \text{ (by \eqref{eq:21stnov245})}\\
        &= \langle ~ , ~ \rangle^\prime 
	\end{align*}
	by  \eqref{29thjan251}. Then by \eqref{metric:Hermitian}, it can be easily checked  that $ {\kH}_{g_\gamma} = ( {\kH}_g )_\gamma. $
    	\end{proof}

We have the following corollary:

\begin{cor} \label{28thmarch251}
	The following sets are in one to one correspondence:
	\begin{enumerate}
		\item $A$-covariant real metrics on $\Omega^1, $
		
		\item $A$-covariant Hermitian metrics on $\Omega^1, $
		
		\item $A_\gamma$-covariant real metrics on $\Omega^1_\gamma, $
		
		\item $A_\gamma$-covariant Hermitian metrics on $\Omega^1_\gamma. $
	\end{enumerate}
\end{cor}
\begin{proof}
The bijection between the sets in (1) and (2) (and similarly, between (3) and (4)) follow from Proposition \ref{27thdec242jb}. On the other hand,  a combination of the second statement of Proposition \ref{prop:15thdec241} with Proposition \ref{realmetric} yields a bijection between the sets in (1) and (3).
\end{proof}

We end this section by recalling the following result by Beggs and Majid:

\begin{theorem} (\cite[Proposition 9.28]{BeggsMajid:Leabh}) \label{thm:15thdec243}
Let $(\Omega^\bullet, \wedge, d)$ be an $A$-covariant differential calculus on a left $A$-comodule algebra $B.$  Moreover, assume that  $(\nabla_{\Omega^1}, \sigma)$ is the unique  $A$-covariant Levi-Civita connection for a covariant metric $\metric$.
   
    Then $\nabla_{\Omega_\cot^1}:= \varphi^{-1}_{\Omega^1, \Omega^1} \circ \Gamma ( \nabla_{\Omega^1} ) $ is a bimodule connection and is  the unique $A_\gamma$-covariant Levi-Civita connection for the metric $\metrictwisted.$
\end{theorem}

Here, the definition of the Levi-Civita connection is as in Definition \ref{15thdec242}. The uniqueness follows by combining Remark \ref{27thdec241jb} with the second assertion of Proposition \ref{prop:15thdec241}.

\section{Chern connections on  twisted holomorphic bimodules} \label{8thmarch253}

This section relates Chern connections on holomorphic bimodules  with cocycle deformations. In order to prove the main result of this section (Theorem \ref{thm:twisted_chern}), we have to go through two steps. Firstly, in Subsection \ref{18thfeb253}, we quickly observe that under unitary cocycle deformations, complex structures get deformed to complex structures. Then in Subsection \ref{18thfeb254}, we prove a similar result about holomorphic bimodules. Thus, if $\mathcal{E}$ is a covariant holomorphic $B$-bimodule equipped with an $A$-covariant Hermitian metric $\kH$ (in the sense of Definition \ref{4thdec231}) and $\gamma$ an unitary $2$-cocycle on $A,$ then by a result of Beggs and Majid (Theorem \ref{thm:Chern}), there exists a unique Chern connection for the pair $ (\Gamma ( \mathcal{E}), {\kH}_\gamma  ),$ where $\kH_\gamma$ is the twisted Hermitian metric of Theorem \ref{thm:twisted-hermitian}. We prove that this connection is nothing but the cocycle deformation of the Chern connection for the pair $( \mathcal{E}, \kH  ).$ 

We will recall the definitions of complex structures, holomorphic bimodules and Chern connections in the relevant subsections. As in the previous section, we will repeatedly need the language of bar categories which has been defined in Section \ref{8thmarch251}.

\subsection{Complex structures} \label{24thfeb251}

We begin by recalling the definition of a complex structure for a $\ast$-differential calculus. 
The symbol $\mathbb{N}_0$ will denote the set $\mathbb{N} \cup \{ 0 \}. $

\begin{definition} \cite{MMF2} \label{18thfeb255}
	A {\em  complex structure} $\Omega^{(\bullet,\bullet)}$ for a $\ast$-differential calculus~$(\Omega^{\bullet}, \wedge, d)$ on a $\ast$-algebra $B$
	is an $\bbN^2_0$-algebra grading $\bigoplus_{(p,q)\in \bN^2_0} \Omega^{(p,q)}$
	for $\Omega^{\bullet}$ such that for all $k \in \bbN_0$ and $(p,q) \in \bN^2_0$, we have 
	\[
	\Omega^k = \bigoplus_{p+q = k} \Omega^{(p,q)},\qquad
	\big(\Omega^{(p,q)}\big)^* = \Omega^{(q,p)},\qquad
	d( \Omega^{(p,q)}) \subseteq \Omega^{(p+1,q)} \oplus \Omega^{(p,q+1)}.
	\]
	Here, $\Omega^{(0,0)} = \Omega^0 = B$.
\end{definition}

The notion of complex structures in noncommutative geometry has been studied by many mathematicians for which we refer to \cite{PolSchwar2003}, \cite{KLvSPodles}, \cite{BS}, and references therein.
By a combination of \cite[Lemma 2.15 and Remark 2.16]{MMF2}, the above definition  is equivalent to those in \cite{BS} and  \cite{KLvSPodles}.

In the setup of Definition \ref{18thfeb255}, it is clear that  $\Omega^{(p,q)}$ is a $B$-bimodule.
Moreover, for $p, q \in \bN_0$,  $\pi^{p,q}: \Omega^{p + q} \rightarrow \Omega^{(p, q)} $ will denote the canonical projections associated to the decomposition $  \Omega^k = \bigoplus_{p+q = k} \Omega^{(p,q)},$ while $\partial, \overline{\partial}$ will denote the maps
\begin{equation}\label{2nddec241jb}
	\partial:= \pi^{(p + 1, q)} \circ d: \Omega^{(p, q)} \rightarrow \Omega^{(p + 1, q)} ~ \text{and} ~ \overline{\partial}:= \pi^{(p, q + 1)} \circ d: \Omega^{(p,q)} \rightarrow \Omega^{(p, q + 1)}.
\end{equation}

In the sequel, we will denote a complex structure on a differential calculus $(\Omega^\bullet, \wedge, d)$ by the quadruple $(\Omega^{(\bullet, \bullet)}, \wedge, \partial, \overline{\partial})$.

\begin{definition} \label{12thdec234}
	A complex structure $ (\Omega^{(\bullet, \bullet)}, \wedge, \partial, \overline{\partial}) $ on a left $A$-comodule $\ast$-algebra $B$ is said to be left $A$-covariant if  $ (\Omega^\bullet, \wedge, d) $ is a covariant differential $\ast$-calculus in the sense of Definition \ref{8thmarch256} and moreover, if $\Omega^{(p,q)}$ is a left $A$-comodule for all $ (p, q) \in \bN^2_0. $
\end{definition}

In this case, the covariance of  $(\Omega^\bullet, \wedge, d)$ implies that the projections $\pi^{p,q}$ and $d$ are $A$-covariant and hence the bimodules $\Omega^{(p,q)}$ are objects in $\cat$ and the maps $\partial$ and $\overline{\partial}$ are also $A$-covariant.

As a consequence of Remark \ref{24thjuly241}, we make the following observation:

\begin{remark} \label{5thmarch251}
 If $\complexdc$ is a complex structure on $B$ such that $\Omega^1$ admits a metric $\metric,$ then $\Omega^{(1,0)}$ and $\Omega^{(0,1)}$ are finitely generated and projective as left $B$-modules.   
\end{remark}

\subsection{Twisting a Covariant Complex Structure} \label{18thfeb253}
Let $\complexdc$ be an $A$-covariant complex structure on a left $A$-comodule $*$-algebra $B$ and $\gamma$ a unitary $2$-cocycle on $A$. Consider the cocycle twisted $A_\gamma$-covariant $\ast$-differential calculus $\dctwisted $ of Proposition \ref{prop:5thdec241}. Since $\Omega^{(p,q)}$ is an object of $\cat$ and $\partial, \delbar{}$ are morphisms of the category $\qMod{A}{}{}{}$, we can make the following definitions:
$$ \Omega^{(p,q)}_\gamma:= \Gamma(\Omega^{(p,q)}),~ \partial_\gamma := \Gamma(\partial),~ (\delbar{})_\gamma:= \Gamma(\delbar{}),$$ 
where $\partial$ and $\delbar{}$ are the maps defined in \eqref{2nddec241jb}. Then we have the following result:

\begin{prop} \label{26thfeb251}
	Let $\complexdc$ be an $A$-covariant complex structure on a left $A$-comodule $*$-algebra $B$ and $\gamma$ a unitary $2$-cocycle on $A$. Then $\complexdctwisted$ is an $A_\gamma$-covariant  complex structure for the $\ast$-differential calculus $\dctwisted.$
\end{prop}

\begin{proof}
	Since $\gamma$ is unitary, we have that $\Omega_\cot^\bullet$ is an $A_\gamma$-covariant differential $*_\gamma$-algebra on $B_\gamma$. Moreover, since the functor $\Gamma: \cat \to \tcat$ is a categorical equivalence, we have $\Omega^k_\gamma=\bigoplus_{p+q=k} \Gamma(\Omega^{(p,q)})= \oplus_{p + q = k} \Omega^{(p,q)}_\gamma$. Then  $d_\gamma(\Omega_\gamma^{(p,q)})\subseteq \Omega_\cot^{(p+1,q)}\oplus \Omega_\cot^{(p,q+1)}.$	

    Next, we verify that $\left(\Omega^{(p,q)}_\gamma\right)^{*\gamma}= \Omega^{(q,p)}_\gamma.$
    If $\omega\in \Omega^{(p,q)}_\gamma,$ then by \eqref{4thfeb251}, 
    $$\omega^{*_\gamma} = \bar{V}(\mone{\omega^*})~\omega_{(0)}^* \in \Omega^{(q,p)}_\gamma.$$
    Conversely, for $\eta \in \Omega^{(q,p)}_\gamma,$ we have
    $(\bar{V}(\mone{\eta^*})~\eta_{(0)}^*)^{*_\gamma} = \eta$
    by virtue of Remark \ref{25thnov241} and Lemma \ref{lem:22ndnov241}. Since $ \bar{V}(\mone{\eta^*})~\eta_{(0)}^*\in \Omega^{(p,q)}_\gamma, $  we conclude that $ \Omega^{(q,p)}_\gamma \subseteq \left(\Omega^{(p,q)}_\gamma\right)^{*\gamma}.$
    The rest of the properties of a complex structure can be checked easily.
\end{proof}

\subsection{Deformation of holomorphic bimodules} \label{18thfeb254}

Now we are in a position to define and twist holomorphic bimodules.

\begin{definition} \label{11thdec231}
	Suppose that $ (\Omega^{(\bullet, \bullet)}, \wedge, \partial, \overline{\partial}) $ is an $A$-covariant complex structure on an $A$-comodule $\ast$-algebra $B$. A covariant holomorphic structure on an object  $\mathcal{E}$ of $\cat$ is the choice of a covariant left $\overline{\partial}$-connection $\overline{\partial}_{\mathcal{E}}$ on $\mathcal{E}$ whose holomorphic curvature vanishes. Concretely, this means that we have an $A$-covariant  $\mathbb{C}$-linear map
    \begin{equation} \label{21stfeb251}
    \overline{\partial}_{\mathcal{E}}: \mathcal{E}\to \Omega^{(0,1)}\otimes_B \mathcal{E} ~ \text{such that} ~ \overline{\partial}_{\mathcal{E}}(be)= b \overline{\partial}_{\mathcal{E}}(e)+ \overline{\partial}(b)\otimes_B e 
    \end{equation}
 for all $b \in B, e \in \cE,$   and moreover, the map 
     $$ R_{\mathcal{E}}^{\Hol}: \mathcal{E} \rightarrow \Omega^{(0, 2)} \otimes_B \mathcal{E}, ~ R_{\mathcal{E}}^{\Hol}(e)=(\ol{\partial}\otimes_B id- id\wedge \ol{\partial}_{\mathcal{E}})\ol{\partial}_{\mathcal{E}} (e)$$
     is identically zero.
    
    In this case, we say that the pair $(\mathcal{E}, \overline{\partial}_{\mathcal{E}})$ is a covariant holomorphic $B$-bimodule. 
	\end{definition}

With the setup and notations of Definition \ref{11thdec231}, we have the following result:

\begin{theorem} \label{24thfeb252}
	Suppose that $(\cE, \delbar{\cE})$ is an $A$-covariant holomorphic bimodule in $\cat$ and $\gamma$ a unitary $2$-cocycle on $A.$ We define
    \begin{equation}\label{16thdec241}
	\delbar{\Gamma(\cE)}:= \varphi_{\Omega^{(0,1)}, \cE}^{-1} \Gamma(\delbar{\cE}): \Gamma(\cE) \to \Omega^{(0,1)}_\gamma \ot_{B_\cot} \Gamma(\cE).
\end{equation} 
    Then the pair $(\Gamma(\cE), \delbar{\Gamma(\cE)})$ is an $A_\gamma$-covariant holomorphic bimodule in $\tcat$.
\end{theorem}
\begin{proof}
	By a verbatim adaptation of the proof of Proposition \ref{prop:29thnov242}, it follows  that $\delbar{\Gamma(\cE)}$is a left $\delbar{}$-connection on $\Gamma(\cE),$ i.e, \eqref{21stfeb251} is satisfied. So, we are left to show that 
	$$R^{\Hol}_{\Gamma(\cE)}= \left( (\delbar{})_\gamma\ot_{B_\cot} \id - \id \wedge_\cot\delbar{\Gamma(\cE)}\right)\delbar{\Gamma(\cE)}=0.$$
    
    We begin by observing that the map  $\delbar{}\ot_{B} \id - \id \wedge\delbar{\cE}:\Omega^{(0,1)}\ot_{B}\cE \rightarrow \Omega^{(0,1)}\ot_{B}\cE $ is a morphism in the category $\qMod{A}{}{}{}$. Hence,  $\Gamma\left(\overline{\partial}\ot_{B}\id - \id \wedge \delbar{\cE}\right)$ is a morphism in the category $\qMod{A_\gamma}{}{}{}$. 
     
     We claim that  for all $e\in \cE$ and  $\omega \in \Omega^{(0,1)},$ the following equation is satisfied:  
    \begin{equation} \label{21stfeb252}
    \varphi_{\Omega^{(0,2)}, \cE} \left( (\delbar{})_\gamma\ot_{B_\cot} \id - \id \wedge_\cot\delbar{\Gamma(\cE)}\right)\varphi_{\Omega^{(0,1)}, \cE}^{-1}(\omega \ot_{B} e)= \Gamma\left( \overline{\partial}\ot_{B}\id - \id \wedge \delbar{\cE} \right) (\omega \ot_{B} e).
    \end{equation}

    The theorem follows easily from this equation. Indeed, for all $e\in \cE$, we have,
	\begin{align*}
		\varphi_{\Omega^{(0,2)}, \cE}R_{\Gamma(\cE)}^{\text{Hol}}(e)& = \varphi_{\Omega^{(0,2)}, \cE}\left( (\overline{\partial} )_\gamma \ot_{B_\cot} \id- (\id \wedge_\cot \delbar{\Gamma(\cE)})\right) \delbar{\Gamma(\cE)}(e)\\
		&= \Gamma \left(\overline{\partial}\ot_{B}\id - \id \wedge \delbar{\cE}\right) \varphi_{\Omega^{(0,1)}, \cE} ~ \delbar{\Gamma ( \cE )}(e)  ~ \text{ ( by \eqref{21stfeb252} ) } \\
		&=\Gamma \left(\overline{\partial}\ot_{B}\id - \id \wedge \delbar{\cE}\right) \Gamma ( \overline{\partial}_{\cE} ) (e) ~ \text{ ( by \eqref{16thdec241}  ) } \\
        &= \Gamma(R^{\text{Hol}}_\cE)(e) =0.
	\end{align*}

	Now we prove \eqref{21stfeb252}. For $e\in \cE,$ let us use the notation  
	\begin{equation}\label{17thdec241}
    \mone{e} \ot \delbar{\cE}(\zero{e}) = \mone{e} \ot \sum_i \omega_i \ot_{B} e_i.
	\end{equation} Since $\delbar{\cE}$ is $A$-covariant we have $\mone{e}\ot \del{\Omega^1\ot_{B} \cE}\big(\delbar{\cE}(\zero{e})\big)= \mone{e}\ot \big(\id \ot \delbar{\cE}\big)\del{\cE}(\zero{e})$, i.e. 
	\begin{equation}\label{16thdec242} 
		\mone{e} \ot  \sum_i  \mone{\omega_i} \mone{e_i} \ot \zero{\omega_i} \ot_{B} \zero{e_i} =  \mtwo{e}\ot \mone{e} \ot \delbar{\cE}(\zero{e}).
	\end{equation}

	Now, for $e\in \cE$ and $\omega \in \Omega^{(0,1)}$, we have
	\begin{align*}
		&\varphi_{\Omega^{(0,2)}, \cE} \left((\overline{\partial})_\gamma \ot_{B_\cot} \id- (\id \wedge_\cot \delbar{\Gamma(\cE)})\right)\varphi^{-1}_{\Omega^{(1,0)}, \cE}(\omega\ot_B e)\\
		&= \coin{\mone{\omega}}{\mone{e}}	\varphi_{\Omega^{(0,2)}, \cE} \left((\overline{\partial})_\gamma(\zero{\omega}) \ot_{B_\cot} \zero{e}- (\zero{\omega} \wedge_\cot \delbar{\Gamma(\cE)}(\zero{e}))\right) \text{ (by \eqref{eq:29thnov243})}\\ 
        &=  \coin{\mtwo{\omega}}{\mtwo{e}}	\co{\mone{\omega}}{\mone{e}} \overline{\partial}(\zero{\omega}) \ot_{B} \zero{e}- \sum_i\coin{\mone{\omega}}{\mone{e}}   \varphi_{\Omega^{(0,2)}, \cE} \\
        &\big( \zero{\omega} \wedge_\cot  \varphi^{-1}_{\Omega^{(0,1)}, \cE}({\omega_i} \ot_{B} {e_i})\big)~ (\text{ by \eqref{nt}, \eqref{16thdec241} and \eqref{17thdec241}} )\\ 
		&= \coin{\mtwo{\omega}}{\mtwo{e}}	\co{\mone{\omega}}{\mone{e}} \overline{\partial}(\zero{\omega}) \ot_{B} \zero{e}- \sum_i\coin{\mone{\omega}}{\mone{e}} \coin{\mone{\omega_i}}{\mone{e_i}}\\
		& \varphi_{\Omega^{(0,2)}, \cE}( \zero{\omega} \wedge_\cot \zero{\omega_i} \ot_{B_\gamma} \zero{e_i}) \text{ (by \eqref{eq:29thnov243})}\\
		&= \overline{\partial}({\omega}) \ot_{B} {e}- \sum_i\coin{\mtwo{\omega}}{\mone{e}} \coin{\mtwo{\omega_i}}{\mone{e_i}}\co{\mone{\omega}}{\mone{\omega_i}} \varphi_{\Omega^{(0,2)}, \cE}\\
		&( \zero{\omega} \wedge \zero{\omega_i} \ot_{B_\gamma} \zero{e_i}) \text{ (as $\bar{\cot} \ast \gamma = \epsilon \otimes \epsilon $   and \eqref{29thnov24jb1})}\\	
        &= \overline{\partial}({\omega}) \ot_{B} {e}- \sum_i\coin{\mthree{\omega}}{\mone{e}} 	\coin{\mthree{\omega_i}}{\mtwo{e_i}} \co{\mtwo{\omega}}{\mtwo{\omega_i}} \\ 
		& \co{\mone{\omega}\mone{\omega_i}}{\mone{e_i}}( \zero{\omega} \wedge \zero{\omega_i} \ot_{B} \zero{e_i}) \text{ (by \eqref{nt})}\\
		&= \overline{\partial}({\omega}) \ot_{B} {e}- \sum_i\coin{\mfour{\omega}}{\mone{e}} 	\co{\mthree{\omega}}{\mthree{\omega_i}\mthree{e_i}} \coin{\mtwo{\omega}\mtwo{\omega_i}}{\mtwo{e_i}} \\ 
		& \co{\mone{\omega}\mone{\omega_i}}{\mone{e_i}}( \zero{\omega} \wedge \zero{\omega_i} \ot_{B} \zero{e_i}) \text{ (by \eqref{iv} of Lemma \ref{lem:formula})}\\ 
		&= \overline{\partial}({\omega}) \ot_{B} {e}- \sum_i\coin{\mtwo{\omega}}{\mone{e}} 	\co{\mone{\omega}}{\mone{\omega_i}\mone{e_i}}( \zero{\omega} \wedge \zero{\omega_i} \ot_{B} \zero{e_i})\\
		& \text{(as $\bar{\cot} \ast \gamma = \epsilon \otimes \epsilon$)}\\
		&= \overline{\partial}({\omega}) \ot_{B} {e}- \sum_i\coin{\mtwo{\omega}}{\mtwo{e}} \co{\mone{\omega}}{\mone{e}} ( \zero{\omega} \wedge \delbar{\cE}(\zero{e})) \text{ (by \eqref{16thdec242})}\\	
		&= \overline{\partial}({\omega}) \ot_{B} {e}-  ( {\omega} \wedge \delbar{\cE}(e)) \text{ (as  $\bar{\cot} \ast \gamma = \epsilon \otimes \epsilon$)}\\
		& = \Gamma\left(\overline{\partial}\ot_{B}\id - \id \wedge \delbar{\cE}\right)(\omega \ot_{B}e).
	\end{align*}
	This proves \eqref{21stfeb252} and hence completes the proof.
\end{proof}

\subsection{Deformation of Chern connections}

From  \cite[Section 2.8]{BeggsMajid:Leabh}, we know that if $A$ is a Hopf $\ast$-algebra, then  $\qMod{A}{}{}{},$ i.e, the monoidal category of left $A$-comodules is a bar-category with structure maps as in Example \ref{5thdec241jb}. If $\mathcal{E}$ is an object of $\cat,$ then a covariant connection $\nabla: \mathcal{E} \rightarrow \Omega^1  \otimes_B \mathcal{E} $ is a morphism in $\qMod{A}{}{}{}.$ Thus, the map
\begin{equation*} 
\overline{\nabla}: \overline{\mathcal{E}} \rightarrow \overline{\Omega^1 \otimes_B \mathcal{E}} ~ {\rm defined} ~ {\rm by} ~ \overline{\nabla} ( \overline{e} ) = \overline{\nabla ( e )}
\end{equation*}
is a morphism in $\qMod{A}{}{}{}.$ We prove the following lemma in which we will use the notations of Example \ref{5thdec241jb}.

\begin{lem} \label{lem:25thnov242}
	Suppose that $\dc$ is a covariant $\ast$-differential calculus on a comodule $\ast$-algebra $B$ and  $\nabla$  a left covariant connection on  an object  $\cE$ of $\cat$. Then the map
   \begin{equation*} 
\widetilde{\nabla}= (\id \ot_{B}\star_{\Omega^1(B)}^{-1})\Upsilon_{\Omega^1, \cE}\overline{\nabla} 
   \end{equation*}
  is an $A$-covariant  right connection on $\ol{\cE}.$ In particular, if $\nabla(e)= \sum_i \omega_i \ot_{B} e_i,$ then
\begin{equation} \label{ref:25thnov243}
\mone{e}^* \ot \widetilde{\nabla}(\ol{\zero{e}}) = \sum_i \mone{e_i}^* \mone{\omega_i}^* \ot \overline{\mzero{e_i}}\ot_{B}\zero{\omega_i}^*.
\end{equation}
  \end{lem}
\begin{proof}
The fact that $\widetilde{\nabla}$ is a right connection has been proved in \cite[Lemma 3.82]{BeggsMajid:Leabh}. Now, as noted above, the map $\overline{\nabla}$ is covariant and moreover, the same is true about $ \star_{\Omega^1(B)}$ and $ \Upsilon_{\Omega^1, \cE}$  (see Example \ref{5thdec241jb}). Hence, $\widetilde{\nabla}$ is a composition of covariant maps and hence covariant. The equation \eqref{ref:25thnov243} follows from the covariance of $\widetilde{\nabla}$ and the observation that $\widetilde{\nabla}(\ol{e})= \sum_i \ol{e_i}\ot_{B} \omega_i^*.$ 
	\end{proof}

    Now we are in a position to recall the following theorem which generalizes the existence of Chern connections on holomorphic vector bundles on complex manifolds. We will be using the map $\pi^{0,1}: \Omega^1 \rightarrow \Omega^{(0, 1)} $ introduced in Subsection \ref{24thfeb251}, while the notations  $\langle ~ , ~ \rangle$ and $\widetilde{\nabla}$ will be as in \eqref{metric:Hermitian} and  Lemma \ref{lem:25thnov242}.

\begin{theorem} (\cite[Theorem 8.53]{BeggsMajid:Leabh}) \label{thm:Chern}
	Suppose that $\complexdc$ is an $A$-covariant complex structure on an $A$-comodule $*$-algebra $B.$ If $\kH$ is an $A$-covariant Hermitian metric on an $A$-covariant holomorphic $B$-bimodule  $(\cE, \delbar{\cE})$ with $\cE$ finitely generated and projective as a left $B$-module,  then there exists a unique $A$-covariant left connection $\ch$ on $\cE$ satisfying the following two conditions:
    \begin{enumerate}
    \item $\ch$ is compatible with $\kH,$ i.e,
    \begin{align} \label{8thsep251}
		d\langle ~ , ~ \rangle = (\id \otimes_B \langle ~ , ~\rangle)(\nabla \otimes_B \id) + (\langle ~ , ~ \rangle \otimes_B \id) (\id\otimes_B \widetilde{\nabla})
	\end{align}
	as maps from $ \mathcal{E} \otimes_B \overline{\mathcal{E}} $ to $ \Omega^1,$
    
    \item  $(\pi^{0,1} \otimes_B \id) \circ \nabla = \delbar{\cE} $.
    \end{enumerate}
	\end{theorem}

The covariance of $\ch$ follows from \cite[Theorem 7.11]{LeviCivitaHK}. The connection $\ch$ is called the Chern connection for the pair $ ( \cE, \delbar{\cE}). $

We continue with the assumptions of Theorem \ref{thm:Chern} and assume that $\gamma$ is a unitary $2$-cocycle on $A.$ Then by Theorem \ref{thm:twisted-hermitian}, $\kH_\gamma$ is an $A_\gamma$-covariant Hermitian metric on $\Gamma ( \cE ). $ On the other hand, Theorem \ref{24thfeb252} implies that  $(\Gamma(\cE), \delbar{\Gamma(\cE)})$ is an $A_\gamma$-covariant holomorphic bimodule in $\tcat$. Thus, by Theorem \ref{thm:Chern}, the pair $(\Gamma(\cE), \delbar{\Gamma(\cE)})$ admits a Chern connection for the Hermitian metric $\kH_\gamma.$ We will prove that this connection is the cocycle twist of the connection $\ch.$    

Let us recall (Proposition \ref{prop:29thnov242}) that given a left $A$-covariant connection $\nabla_{\cE}$   on an object  $\cE$ of $\cat,$  we have a left $A_\gamma$-covariant connection $\nabla_{\Gamma(\cE)}$ on $\Gamma(\cE).$ Then by Lemma \ref{lem:25thnov242}, there is a right $A_\gamma$-covariant connection $\widetilde{\nabla_{\Gamma(\cE)}}$ on $\ol{\Gamma(\cE)}.$
The following lemma establishes a relation  between $\widetilde{\nabla_{\Gamma(\cE)}}$ and $\widetilde{\nabla_{\cE}}$ which will be needed in the proof of Theorem \ref{thm:twisted_chern}.

\begin{lem} \label{lem:17thdec242}
	 Suppose that $\nabla_\cE: \cE \to \Omega^1 \ot_{B} \cE$ is a left $A$-covariant connection on an object $\cE$ of $\cat$ and $\widetilde{\nabla_\cE}: \overline{\cE} \to \overline{\cE} \ot_B \Omega^1$ is a right connection on $\cE$ defined as in Lemma \ref{lem:25thnov242}. Then the twisted connection $\nabla_{\Gamma(\cE)}$ of $\nabla_\cE$ as in \eqref{eq:3rddec241} satisfies $$\widetilde{\nabla_{\Gamma(\cE)}}= (\fN_\cE^{-1}\ot_{B_\gamma}\id)\varphi^{-1}_{\overline{\cE}, \Omega^1}\circ \Gamma(\widetilde{\nabla_\cE})\fN_\cE$$ where $\fN_\cE $ is as defined in \eqref{5thdec242jb}.
\end{lem}

\begin{proof}
	Suppose that  for $e \in \cE,~\nabla(e)= \sum_i \omega_i \ot_{B} e_i$. Then by the definition of $\widetilde{\nabla_{\Gamma(\cE)}},$  we get
	\begin{align*}
		&\widetilde{\nabla_{\Gamma(\cE)}}(\overline{e})\\
		&= (\id \ot_{B_\cot} \star_{\Omega^1_\gamma}^{-1})\Upsilon_{\Omega^1_\gamma, \Gamma(\cE)}\overline{\nabla_{\Gamma(\cE)}}(\overline{e})\\
		&=  (\id \ot_{B_\cot} \star_{\Omega^1_\gamma}^{-1})\Upsilon_{\Omega^1_\gamma, {\Gamma(\cE)}}  \overline{\varphi_{\Omega^1, \cE}^{-1} \big(\sum_i \omega_i \ot_{B} e_i \big) } \text{ (by \eqref{eq:3rddec241})}\\
		&=  (\id \ot_{B_\cot} \star_{\Omega^1_\gamma}^{-1}) \sum_i \overline{\coin{\mone{\omega_i}}{\mone{e_i}}} \Upsilon_{\Omega^1_\gamma, {\Gamma(\cE)}}  \overline{ \big( \zero{\omega_i} \ot_{B_\gamma} \zero{e_i} \big) } \text{ (by \eqref{eq:29thnov243})}\\
		&= \left(\id \ot_{B_\cot} \star_{\Omega^1_\gamma}^{-1} \right) \sum_i \overline{\coin{\mone{\omega_i}}{\mone{e_i}}}   \left( \overline{\zero{e_i}} \ot_{B_\cot} \overline{\zero{\omega_i}}\right) ~ \text{(by Example \ref{5thdec241jb})} \\
		&= \sum_i \overline{\coin{\mone{\omega_i}}{\mone{e_i}}}  \left( \overline{\zero{e_i}} \ot_{B_\cot} { \zero{\omega_i}^{\star_\gamma}}\right) \text{ (by \eqref{eq:18thdec241})}\\
		&= \sum_i {\co{S(\mtwo{\omega_i})^*}{S(\mone{e_i})^*}}    \left( \overline{\zero{e_i}} \ot_{B_\cot} { \bar{V}(\mone{\omega_i}^*) \zero{\omega_i}^*}\right) \text{ (by \eqref{19thSept20242} and \eqref{4thfeb251})}\\
		&= \sum_i {\co{S(\mtwo{\omega_i})^*}{S(\mtwo{e_i})^*}} \bar{V}(\mone{\omega_i}^*)  \bar{V}(\mone{e_i}^*) \left( \fN_\cE^{-1}\ot_{B_\cot} \id \right)   \left( \overline{\zero{e_i}} \ot_{B_\cot} {  \zero{\omega_i}^*}\right) \text{(by \eqref{11thdec24jb1})}\\
		&= \sum_i \left(\fN_\cE^{-1} \ot_{B_\cot} \id \right) \bar{V}(\mtwo{e_i}^* \mtwo{\omega_i}^*)\coin{\mone{e_i}^*}{\mone{\omega_i}^*}  \left( \overline{\zero{e_i}} \ot_{B_\cot} {  \zero{\omega_i}^*}\right) (\text{by \eqref{3rdfeb253}})\\
		&= \sum_i \left(\fN_\cE^{-1} \ot_{B_\cot} \id \right) \bar{V}(\mone{e_i}^* \mone{\omega_i}^*) \varphi_{ \overline{\cE}, \Omega^1}^{-1} \left(\overline{\zero{e_i}} \ot_{B} \zero{\omega_i^*}\right) \text{ (by \eqref{eq:29thnov243})}\\
		&= \left(\fN_\cE^{-1} \ot_{B_\cot} \id \right)  \varphi_{ \overline{\cE}, \Omega^1}^{-1} ( \bar{V}(\mone{e}^*) \Gamma(\widetilde{\nabla_\cE} ) (\overline{\zero{e}}) ) ~(\mathrm{by ~ \eqref{ref:25thnov243}})\\
		&= \left(\fN_\cE^{-1} \ot_{B_\cot} \id \right)  \varphi_{ \overline{\cE}, \Omega^1}^{-1}  \Gamma(\widetilde{\nabla_\cE} ) \fN_\cE\left(\overline{e}\right) 
	\end{align*}
	by \eqref{5thdec242jb}. This completes the proof of the lemma.
\end{proof}

Now we can prove the main theorem of this section:

\begin{theorem}\label{thm:twisted_chern}
Suppose that $\gamma$ is a unitary cocycle on a Hopf $*$-algebra $A$ and  $B$ is an $A$-comodule $*$-algebra. If $\kH$ is a covariant Hermitian metric on a holomorphic bimodule $(\mathcal{E}, \overline{\partial}_{\cE})$ in $\cat$ such that $\cE  $ is finitely generated and projective as a left $B$-module, then the connection 
   $$\nabla^\prime:= \varphi_{\Omega^1, \cE}^{-1} \Gamma ( \ch )$$
  is the Chern connection for the Hermitian metric $\kH_\cot$ (as  in Theorem \ref{thm:twisted-hermitian}) on the holomorphic bimodule $(\Gamma (\cE ), \delbar{\Gamma(\cE)})$ in $\tcat$. Thus, $\tch = \nabla^\prime. $
  \end{theorem}
\begin{proof}
 Firstly, as $\Gamma$ is a categorical equivalence, $\Gamma ( \cE ) $ is finitely generated and projective as a left $B_\gamma$-module and moreover, $\nabla^\prime$ is a connection by virtue of Proposition \ref{prop:29thnov242}. Thus, by the uniqueness of the Chern connection in Theorem \ref{thm:Chern}, it is enough to prove that the connection $\nabla^\prime$ satisfies the conditions (1) and (2) of that theorem.	Firstly,
    \begin{eqnarray*}
(\pi^{0,1}_\gamma\ot_{B_\cot} \id)\nabla^\prime &=& (\pi^{0,1}_\gamma\ot_{B_\cot} \id) \varphi^{-1}_{\Omega^1, \cE} \Gamma ( \ch )\\
&=& \varphi^{-1}_{\Omega^{(0,1)}, \cE} \Gamma ( \pi^{0,1} \otimes_B \text{id}  ) \Gamma ( \ch )  ~  \text{(by \eqref{15thdec244})}\\
&=& \varphi^{-1}_{\Omega^{(0,1)}, \cE} \Gamma ( ( \pi^{0,1} \otimes_B \text{id}  )  \ch )\\
&=& \varphi^{-1}_{\Omega^{(0,1)}, \cE} \Gamma ( \delbar{\cE} )\\
&=&\delbar{\Gamma({\cE})} \text{~(by \eqref{16thdec241}).}
  \end{eqnarray*}

    So, we are left to show that $\nabla^\prime$ is compatible with $\kH_\cot$. For the rest of the proof, we will use the notation $\cdot_\gamma$ to denote both the left and the right $B_\gamma$-bimodule actions on  $\Omega^1_{\gamma}.$ For example, for the right $B_\gamma$ action, this will mean
    $$ \cdot_\gamma = \Gamma ( \cdot ) \varphi_{\Omega^1, B}. $$
    We compute
	\begin{align*}
		&\cdot_\gamma \big((\id_{\Omega^1_\gamma} \ot_{B_\gamma} \langle~,~\rangle_\cot )(\nabla^\prime \ot_{B_\gamma} \id_{\overline{\Gamma ( \cE ) }})\big) + \cdot_\gamma\big( (\langle~,~\rangle_\gamma\ot_{B_\gamma} \id_{\Omega^1_\gamma} ) (\id_{\Gamma ( \cE ) } \ot_{B_\gamma} \widetilde{\nabla^\prime})  \big) \\
        &= 	\cdot_\gamma \left((\id_{\Omega^1_\gamma} \ot_{B_\gamma}  \Gamma(\langle~,~\rangle) \varphi_{\cE, \overline{\cE}}(\id_{\Gamma ( \cE )} \ot_{B_\gamma}\fN_\cE) )(\nabla^\prime \ot_{B_\gamma} \id_{\overline{\Gamma ( \cE ) }}) \right)+\\
		&  \cdot_\gamma \left( ( \Gamma(\langle~,~\rangle) \varphi_{\cE, \overline{\cE}}(\id_{\Gamma ( \cE ) } \ot_{B_\gamma} \fN_\cE)\ot_{B_\gamma} \id_{\Omega^1_\gamma} ) (\id_{\Gamma ( \cE ) } \ot_{B_\gamma} (\fN_\cE^{-1}\ot_{B_\gamma}\id_{\Omega^1_\gamma} )\varphi^{-1}_{\overline{\cE}, \Omega^1}\circ \Gamma(\widetilde{\ch})\fN_\cE)  \right)\\
        &\text{ (by \eqref{18thfeb251} and Lemma \ref{lem:17thdec242})} \\
		&= 	\cdot_\gamma \left((\id_{\Omega^1_\gamma} \ot_{B_\gamma}  \Gamma(\langle~,~\rangle)) (\id_{\Omega^1_\gamma} \ot_{B_\gamma}\varphi_{\cE, \overline{\cE}}) (\nabla^\prime \ot_{B_\gamma} \id_{\Gamma ( \overline{\cE} ) } )(\id_{\Gamma ( \cE ) } \ot_{B_\gamma} \fN_\cE) \right)+\\
		& \cdot_\gamma \left( ( \Gamma(\langle~,~\rangle) \varphi_{\cE, \overline{\cE}}\ot_{B_\gamma} \id_{\Omega^1_\gamma} ) (\id_{\Gamma ( \cE ) } \ot_{B_\gamma} \varphi^{-1}_{\overline{\cE}, \Omega^1})(\id_{\Gamma ( \cE )} \ot_{B_\gamma} \Gamma(\widetilde{\ch}))(\id_{\Gamma ( \cE )} \ot_{B_\gamma} \fN_\cE)  \right) \\	
		&= 	\cdot_\gamma \Big((\id_{\Omega^1_\gamma} \ot_{B_\gamma}  \Gamma(\langle~,~\rangle)) (\id_{\Omega^1_\gamma} \ot_{B_\gamma}\varphi_{\cE, \overline{\cE}}) (\varphi^{-1}_{\Omega^{1}, \cE}\ot_{B_\gamma}\id_{\Gamma ( \overline{\cE} ) }) \varphi^{-1}_{\Omega^1 \ot_{B}\cE, \overline{\cE}} \Gamma(\ch \ot_B \id_{\overline{\cE}}) \varphi_{\cE, \overline{\cE}}\\
		&(\id_{\Gamma ( \cE ) } \ot_{B_\gamma}\fN_\cE) \Big) +  \cdot_\gamma \Big( ( \Gamma(\langle~,~\rangle) \ot_{B_\cot} \id_{\Omega^1_\gamma})(\varphi_{\cE, \overline{\cE}}\ot_{B_\gamma} \id_{\Omega^1_\gamma}) (\id_{\Gamma ( \cE ) } \otimes_{B_\gamma} \varphi^{-1}_{\overline{\cE}, \Omega^1})  \varphi^{-1}_{\cE, \overline{\cE}\ot_{B}\Omega^1} \\
        &\Gamma(\id_{\cE} \ot_B \widetilde{\ch})
        \varphi_{\cE, \overline{\cE}}    (\id_{\Gamma ( \cE ) } \otimes_{B_\gamma} \fN_\cE)  \Big) \text{ (by definition of $\nabla^\prime$, \eqref{15thdec244})}\\
		&= \cdot_\gamma \Big((\id_{\Omega^1_\gamma} \ot_{B_\gamma}  \Gamma(\langle~,~\rangle)) (\id_{\Omega^1_\gamma} \ot_{B_\gamma}\varphi_{\cE, \overline{\cE}}) (\id_{\Omega^1_\gamma} \ot_{B_\gamma}\varphi_{\cE, \overline{\cE}}^{-1}) \varphi^{-1}_{\Omega^1,\cE \otimes_B \overline{\cE}} \Gamma(\ch \ot_B \id_{\overline{\cE}})\\
        &\varphi_{\cE, \overline{\cE}}(\id_{\Gamma (\cE ) } \ot_{B_\gamma}\fN_\cE) \Big)+
		  \cdot_\gamma \Big( ( \Gamma(\langle~,~\rangle) \ot_{B_\cot} \id_{\Omega^1_\gamma} )(\varphi_{\cE, \overline{\cE}}\ot_{B_\gamma} \id_{\Omega^1_\gamma})( \varphi^{-1}_{\cE,\overline{\cE}} \ot_{B_\cot} \id_{\Omega^1_\gamma} )  \varphi^{-1}_{\cE \ot_B \overline{\cE}, \Omega^1}   \\
        &\Gamma(\id_{\cE} \ot_B \widetilde{\ch})  \varphi_{\cE, \overline{\cE}}    (\id_{\Gamma ( \cE ) } \ot_{B_\gamma} \fN_\cE)  \Big)  \text{ (as  $\Gamma$ is a monoidal equivalence)}\\
		&=\Gamma(\cdot)\varphi_{\Omega^1 ,B}  \Big((\id_{\Omega^1_\gamma} \ot_{B_\gamma}  \Gamma(\langle~,~\rangle)) \varphi^{-1}_{\Omega^1 ,\cE \otimes_B \overline{\cE}} \Gamma(\ch \otimes_B \id_{\overline{\cE}}) \varphi_{\cE, \overline{\cE}}(\id_{\Gamma ( \cE ) } \ot_{B_\gamma}\fN_\cE) \Big)+\\
		&  \Gamma(\cdot) \varphi_{B, \Omega^1} \Big( ( \Gamma(\langle~,~\rangle) \ot_{B_\cot} \id_{\Omega^1_\gamma} )  \varphi^{-1}_{\cE \ot_B \overline{\cE}, \Omega^1}   \Gamma(\id_{\cE} \ot_B \widetilde{\ch})  \varphi_{\cE, \overline{\cE}}    (\id_{\Gamma ( \cE ) } \ot_{B_\gamma}\fN_\cE)  \Big) \text{ (by \eqref{1stdec241jb})}\\
        &=\Gamma(\cdot ) \Gamma ( (\id_{\Omega^1} \ot_{B}  \langle~,~\rangle) ) \Gamma (\ch \otimes_B \id_{\overline{\cE}})) \varphi_{\cE, \overline{\cE}}(\id_{\Gamma ( \cE ) } \ot_{B_\gamma}\fN_\cE) +\\
		&  \Gamma\Big( \cdot  (\langle~,~\rangle \ot_B \id_{\Omega^1}) (\id_{\cE} \ot_B \widetilde{\ch})\Big)  \varphi_{\cE, \overline{\cE}}    (\id_{\Gamma ( \cE ) } \ot_{B_\gamma}\fN_\cE)   \text{ (by \eqref{15thdec244})}\\
		&=\Gamma\Big(\cdot (\id_{\Omega^1} \ot_{B}  \langle~,~\rangle) (\ch \otimes_B \id_{\overline{\cE}}) + \cdot  (\langle~,~\rangle \ot_B \id_{\Omega^1} ) (\id_{\cE} \ot_B  \widetilde{\ch})\Big)  \varphi_{\cE, \overline{\cE}}    (\id_{\Gamma ( \cE ) } \ot_{B_\gamma}\fN_\cE)   \\
		&= \Gamma(d\langle~,~ \rangle)\varphi_{\cE, \overline{\cE}}    (\id_{\Gamma ( \cE ) } \ot_{B_\gamma}\fN_\cE)\\
        &= d_\gamma (\langle~,~ \rangle_\gamma) 
	\end{align*}
	\noindent
  by \eqref{18thfeb251}.  Hence, $\nabla^\prime$ is the Chern connection on $\Gamma(\cE).$
	\end{proof}

We end this section by spelling out a sufficient condition for $\tch$ to be a bimodule connection. Let us recall that a left $\overline{\partial}$-connection $\overline{\partial}_{\mathcal{E}}$ on $\mathcal{E}$ is said to be a left $\sigma$-bimodule $\overline{\partial}$-connection if there exists a $B$-bimodule map 
$\sigma: \mathcal{E} \otimes_B \Omega^{(0,1)} \rightarrow \Omega^{(0, 1)} \otimes_B \mathcal{E}$
such that
$$\overline{\partial}_{\mathcal{E}} (e b) = \overline{\partial}_{\mathcal{E}} (e) b + \sigma (e \otimes_B \overline{\partial} b)$$
for all $e \in \mathcal{E}$ and for all $b \in B$. 

\begin{cor}
With the notations and assumptions of Theorem \ref{thm:Chern}, let us assume that  $\overline{\partial}_{\mathcal{E}}$ is a left bimodule $\overline{\partial}$-connection. Then for any unitary $2$-cocycle $\gamma$ on $A,$ the Chern connection $\tch$ is again a bimodule connection.
\end{cor} 
\begin{proof}
Since $\overline{\partial}_{\mathcal{E}}$ is a bimodule $\overline{\partial}$-connection, \cite[Proposition 8.54]{BeggsMajid:Leabh} implies that  $\ch$ is also a bimodule connection.
Now by Theorem \ref{thm:twisted_chern},  $\tch$ is a $2$-cocycle twist of $\ch$ and therefore, an application of  the second assertion of Proposition \ref{prop:29thnov242} yields that $\tch$ is a bimodule connection.
\end{proof}

\section{Levi-Civita connection and K\"ahler structures} \label{8thmarch254}

The goal of this section is to state and prove the main result of this article for which we recall the definition of noncommutative Hermitian and K\"ahler structures from \cite{MMF3}. 

Let $( \Omega^\bullet, \wedge, d )$ be a differential calculus on $B.$ As before, we will tacitly use the fact (see Remark \ref{24thjuly241}) that if $\Omega^1$ admits a metric $\metric,$ then it is finitely generated and projective as a left $B$-module.  If there exists a natural number $n$ such that $\Omega^n \neq 0$ and $\Omega^m = 0$ for all $m > n,$ then we say that $( \Omega^\bullet, \wedge, d )$ has total dimension $n.$ A central form in $( \Omega^\bullet, \wedge, d )$ is an element of $\Omega^\bullet$ which belongs to the center of the algebra $( \Omega^\bullet, \wedge ).$ If $( \Omega^\bullet, \wedge, d )$ is a $\ast$-differential calculus, then a form $\omega$ is called real if $\omega^\ast = \omega.$ Finally, for a central real form $\kappa$ on a $\ast$-differential calculus, the Lefschetz operator 
$$L_\kappa: \Omega^\bullet \rightarrow \Omega^\bullet ~ {\rm is} ~ {\rm defined} ~ {\rm as} ~ L_\kappa ( \omega ) = \kappa \wedge \omega. $$
Since $\kappa$ is a central real-form, $L$ is a $B$-bimodule map satisfying $L ( \omega^\ast ) = ( L ( \omega ) )^\ast. $

\begin{definition}   \cite[Definition 7.1]{MMF3} 
	A covariant Hermitian form for a $2n$-dimensional covariant  complex structure on a left $A$-comodule  $\ast$-algebra $B$ is an  $A$-coinvariant central, real, $(1, 1 )$-form $\kappa$ such that the Lefschetz operator induces $B$-bimodule isomorphisms 
	$$L^{n - k}_\kappa: \Omega^k \rightarrow \Omega^{2n - k}, ~ \omega \mapsto \kappa^{n - k} \wedge \omega  $$
	for all $0 \leq k < n.$ 
	
A covariant Hermitian form $\kappa$ is called a K\"ahler form if $d \kappa = 0.$
	\end{definition}

The tuple $ ( \Omega^{\bullet, \bullet}, \wedge, \partial, \overline{\partial}, \kappa ) $ is called an Hermitian (respectively,  K\"ahler  structure) on the differential calculus. For classical complex manifolds, the above definition coincides with the usual definition of Hermitian and K\"ahler  structures (see  Definition 3.1.1 and Definition 3.1.6  of \cite{DanHuybrechts}).

Before proving the main theorem, let us make a couple of observations. The following is an easy consequence of the fact that $\kappa$ is an $A$-coinvariant element in $\Omega_\cot^{(1,1)}.$ 
\begin{prop}
 If $ ( \Omega^{(\bullet, \bullet)}, \wedge, \partial, \overline{\partial}, \kappa ) $ is an $A$-covariant Hermitian (respectively, K\"ahler) structure and $\gamma$ is a unitary $2$-cocycle on $A,$ then $ ( \Omega^{(\bullet, \bullet)}_\gamma, \wedge_\gamma, \partial_\gamma,  (\overline{\partial})_\gamma, \kappa ) $ is also a covariant Hermitian (respectively, K\"ahler) structure.
\end{prop}
\begin{proof}
Since $\gamma$ is a unitary $2$-cocycle, Proposition \ref{26thfeb251} implies that $\complexdctwisted$   is an $A_\gamma$-covariant complex structure.	Now since $\kappa$ is an $A_\gamma$-coinvariant element in $\Omega^{(1,1)}_\gamma$, \eqref{29thnov24jb1} implies that for all $\omega \in \Omega^\bullet_\gamma,$ 
    $$\omega\wedge_\gamma \kappa = \co{\mone{\omega}}{1}\zero{\omega} \wedge \kappa = \omega \wedge \kappa ~ \text{and similarly,} ~ \kappa \wedge_\gamma \omega =\kappa \wedge \omega. $$ 
  Therefore,   as $\kappa$ is central,
  $$\omega\wedge_\gamma \kappa = \omega \wedge \kappa = \kappa \wedge \omega = \kappa \wedge_\gamma \omega.$$
  The form $\kappa$ in $\Omega_\gamma^{(1,1)}$ is real as $\kappa^{*_\gamma}= \bar{V}(1) \kappa^*= \kappa.$ Moreover, since $ \kappa \wedge_\gamma \omega = \kappa \wedge \omega $ as noted above, it follows that $ L^{n - k}_\kappa: \Omega^k_\gamma \rightarrow \Omega^{2n - k}_\gamma $ is an isomorphism.
  
  Finally, if $ d \kappa = 0, $ then $ d_\gamma \kappa = 0. $
\end{proof}

Our next observation is a  generalization of \cite[Theorem 4.4]{LeviCivitaHK}.

\begin{prop} \label{lem:27thdec241}
	Let  $\complexdc$ be a covariant complex structure on an $A$-comodule $\ast$-algebra $B$ and $\metric$ be a covariant real metric on $\Omega^1$ such that
	\begin{equation}\label{diamond}
		(\eta_1, \eta_2)=0\text{ for  all } \eta_1, \eta_2 \text{ in }\Omega^{(1,0)}\text{ or }\Omega^{(0,1)}.
	\end{equation} 
    Then the covariant Hermitian metric $\kH_g$ on $\Omega^{1}$ obtained from Proposition \ref{27thdec242jb} is of the form  $\kH_g= \kH_1 \oplus \kH_2$, where $\kH_1$ and $\kH_2$ are the restrictions of $\kH$ on $\ol{\Omega^{(1,0)}}$ and $\ol{\Omega^{(0,1)}}$ respectively. In fact, $\kH_1$ and $\kH_2$ are covariant Hermitian metrics on $\Omega^{(1,0)}$ and $\Omega^{(0,1)}$ respectively.
    \end{prop}
\begin{proof}
We note that by definition,  $\kH_1$ and $\kH_2$ are morphisms in $\cat$ satisfying \eqref{15thdec246} and moreover,   $\kH_g = \kH_1 \oplus \kH_2.$ 
    
    As   $\kH_g(\ol{\eta_1})(\eta_2)= (\eta_2, \eta_1^*)$ for all
    $\eta_1, \eta_2 \in \Omega^1,$ \eqref{diamond} implies that
     $$\kH_1(\ol{\eta_1}) (\eta_2)= (\eta_2, \eta_1^*) = 0 = \kH_2 ( \overline{\eta_2} ) ( \eta_1 ) $$ 
 for all $\eta_1\in \Omega^{(1,0)}, \eta_2 \in \Omega^{(0,1)}.$    Hence we have that  $\kH_g ( \overline{\Omega^{(1,0)}} ) \subseteq \hm{B}{\Omega^{(1,0)}}{B}  $ and 
$\kH_g ( \overline{\Omega^{(0,1)}} ) \subseteq \hm{B}{\Omega^{(0,1)}}{B}.$

Since $\kH_g: \overline{\Omega^1} \rightarrow \hm{B}{\Omega^{1}}{B}  $ is an isomorphism, it follows that the maps  $\kH_1:  \overline{\Omega^{(1,0)}} \rightarrow \hm{B}{\Omega^{(1,0)}}{B}  $ and 
$\kH_2:  \overline{\Omega^{(0,1)}}  \rightarrow \hm{B}{\Omega^{(0,1)}}{B}$ are also isomorphisms.
\end{proof}

Now we present one last ingredient for proving Theorem \ref{29thjan253}. 

\begin{definition} \cite[Definition 7.21]{BeggsMajid:Leabh}   \label{13thdec231}
 A complex structure $ (\Omega^{(\bullet, \bullet)}, \wedge, \partial, \overline{\partial}) $ over a $\ast$-algebra $B$ is called factorizable if  
the restriction of the wedge maps
$$ \wedge_{(0, q), (p,  0)}: \Omega^{(0, q)} \otimes_{B} \Omega^{(p, 0)} \rightarrow \Omega^{(p,q)} ~ \text{and} ~  \wedge_{(p, 0), (0,  q)}: \Omega^{(p, 0)} \otimes_{B} \Omega^{(0, q)} \rightarrow \Omega^{(p,q)}      $$
are invertible for all $ (p, q) \in \mathbb{N} \times \mathbb{N}. $
\end{definition}
In the sequel, we will denote the inverse of the map $ \wedge_{(0, 1),( 1, 0)} $ by the symbol $\theta^{(1, 1)}_l.$

Factorizable complex structures give rise to holomorphic bimodules in the following way:

\begin{theorem} (\cite[Proposition 6.1]{BS}) \label{28thfeb251}
If $\complexdc$ is an $A$-covariant factorizable complex structure as above and we define the $A$-covariant map
$$\overline{\partial}_{\Omega^{(1,0)}}:= \theta^{(1, 1)}_l \circ \overline{\partial}: \Omega^{(1, 0)} \rightarrow \Omega^{(0, 1)} \otimes_B \Omega^{(1, 0)},$$
then the pair $ (\Omega^{(1,0)}, \overline{\partial}_{\Omega^{(1,0)}}) $ is an $A$-covariant holomorphic $B$-bimodule.
\end{theorem}

We recall that the opposite complex structure on $ (\Omega^{(\bullet, \bullet)} \wedge, \partial, \overline{\partial}) $   is given by the $\bN^2_0$-graded algebra
 $\bigoplus_{(a,b)\in \bN^2_0} \Omega^{(a,b), \text{op}}$, where
$$ \Omega^{(a,b), \text{op}}:= \Omega^{(b, a)}.  $$
The $\partial$ and $\overline{\partial}$-operators for the opposite complex structure are given by:
$$ \partial_{\text{op}}:= \overline{\partial}: \Omega^{(a, b), \text{op}} \rightarrow \Omega^{(a + 1, b), \text{op}}, ~  \overline{\partial}_{\text{op}}:= \partial: \Omega^{(a, b), \text{op}} \rightarrow \Omega^{(a, b + 1), \text{op}}. $$

Thus, for a covariant factorizable complex structure, Theorem \ref{28thfeb251} shows that $\Omega^{(0, 1)}$ is again a holomorphic covariant $B$-bimodule for the opposite complex structure. 

Now let us assume that $\complexdc$ is an $A$-covariant factorizable complex structure equipped with a real covariant metric $\metric$ satisfying \eqref{diamond} and let us collect some observations that follow by combining various results that we have obtained so far.  

Firstly, Proposition \ref{lem:27thdec241} shows that the Hermitian metric $\kH_g$ on $\Omega^1$ is of the form $ \kH_g = \kH_1 \oplus \kH_2 $ for some covariant Hermitian metrics $\kH_1$ and $\kH_2$ on $\Omega^{(1,0)}$ and $\Omega^{(0, 1)}$ respectively. Remark \ref{5thmarch251} implies that $\Omega^{(1,0)}$ and $\Omega^{(0,1)}$ are finitely generated and projective as  left $B$-modules. Therefore, by Theorem \ref{thm:Chern} and the discussion above, we have covariant Chern connections $\nabla_{\text{Ch}, \Omega^{(1, 0)}}$ and $\nabla_{\text{Ch}, \Omega^{(0, 1)}}$ for the pairs  $ ( \Omega^{(1, 0)}, \kH_1 ) $ and $( \Omega^{(0,1)}, \kH_2  )$ respectively.

On the other hand, if $\gamma$ is a unitary $2$-cocycle on $A,$ then by Proposition \ref{realmetric}, $\metrictwisted$ is a real covariant metric on $\Omega^1_\gamma$ and it can be easily verified that $\metrictwisted$ satisfies \eqref{diamond} for the covariant complex structure $\complexdctwisted.$ Once again, by Proposition \ref{lem:27thdec241}, we have a covariant Hermitian metric $\kH_{g_\gamma}$ on $\Omega^1_\gamma$ with $ \kH_{g_\gamma} = \kH^\prime_1 \oplus \kH^\prime_2 $ for some covariant Hermitian metrics $\kH^\prime_1, \kH^\prime_2$ as in that Proposition. Moreover, Theorem \ref{24thfeb252} proves that we have covariant holomorphic bimodule structures on $\Omega^{(1,0)}_\gamma$ and $\Omega^{(0, 1)}_\gamma.$ We will denote the Chern connections for the pairs $ (\Omega^{(1,0)}_\gamma, \kH^\prime_1  ) $ and  $ (\Omega^{(0,1)}_\gamma, \kH^\prime_2  ) $ by the symbols $\nabla_{\text{Ch}, \Omega^{(1, 0)}_\gamma}$ and $\nabla_{\text{Ch}, \Omega^{(0, 1)}_\gamma}$ respectively. 

With these notations, we can now state the following theorem:

\begin{theorem} \label{29thjan253}
Let $\complexdc$ be an $A$-covariant factorizable complex structure on a left $A$ comodule $\ast$-algebra $B$ and $\metric$ be a real covariant metric on $\Omega^1$ satisfying \eqref{diamond}. Assume that $\nabla_{\Omega^1}$ is the unique covariant Levi-Civita connection for $\metric$ such that
\begin{equation} \label{1stmarch251}
 \nabla = \nabla_{\text{Ch}, \Omega^{(1, 0)}} \oplus \nabla_{\text{Ch}, \Omega^{(0, 1)}}
 \end{equation}
 with respect to the decomposition $\Omega^1 = \Omega^{(1,0)} \oplus  \Omega^{(0,1)}. $
If $\gamma$ is a unitary $2$-cocycle on $A,$ then there exists a unique covariant Levi-Civita connection $\nabla_{\Omega^1_\gamma}: = \varphi^{-1}_{\Omega^1, \Omega^1} \circ \Gamma ( \nabla_{\Omega^1} )$ for $\metrictwisted$ and moreover  $ \nabla_{\Omega^1_\gamma} =\nabla_{\text{Ch}, \Omega^{(1, 0)}_\gamma} \oplus \nabla_{\text{Ch}, \Omega^{(0, 1)}_\gamma}. $
\end{theorem}
\begin{proof}
The existence and uniqueness follows from Theorem \ref{thm:15thdec243} and we know that
 $\nabla_{\Omega^1_\gamma} = \varphi^{-1}_{\Omega^1, \Omega^1} \circ \Gamma ( \nabla_{\Omega^1} ). $

For proving the second assertion, we will be using the notations introduced before the statement of this theorem. Moreover, $ \mathfrak{S}_{\cE}: \Gamma (  \hm{B}{\cE}{B}  ) \rightarrow \hm{B_\gamma}{\Gamma ( \cE ) }{B_\gamma}  $ and  $ \mathfrak{N}_{\cE} : \overline{\Gamma ( \cE ) } \rightarrow \Gamma ( \overline{\cE} ) $ will denote the isomorphisms introduced in \eqref{fS} and \eqref{5thdec242jb} respectively. It is easy to observe that 
$$ \mathfrak{N}_{\Omega^1} =   \mathfrak{N}_{\Omega^{(1,0)}} \oplus  \mathfrak{N}_{\Omega^{(1, 0)}} $$
and for all $f \in \hm{B}{\Omega^1}{B}, $
$$ \mathfrak{S}_{\Omega^1} (  f|_{\Omega^{(1,0)}} ) = \mathfrak{S} _{\Omega^{(1,0)}} (  f|_{\Omega^{(1,0)}} ) $$
and a similar equation holds for $\Omega^{(0,1)}.$

Then by Theorem \ref{thm:twisted-hermitian},
\begin{eqnarray*}
(  \kH_g  )_\gamma &=& \mathfrak{S}_{\Omega^1} \circ \Gamma (  \kH_g ) \circ \mathfrak{N}_{\Omega^1}\\
&=&  ( \mathfrak{S}_{\Omega^{(1, 0)}} \oplus \mathfrak{S}_{\Omega^{(0, 1)}}   ) (  \Gamma (  \kH_1 ) \oplus \Gamma ( \kH_2 )  ) (  \mathfrak{N}_{\Omega^{(1, 0)}} \oplus  \mathfrak{N}_{\Omega^{(0, 1)}}   )\\
&=& \mathfrak{S}_{\Omega^{(1, 0)}} \circ \Gamma (  \kH_1 ) \circ \mathfrak{N}_{\Omega^{(1, 0)}} \oplus \mathfrak{S}_{\Omega^{(0, 1)}} \circ \Gamma (  \kH_2 ) \circ \mathfrak{N}_{\Omega^{(0, 1)}}\\
&=& ( \kH_1  )_\gamma \oplus (  \kH_2  )_\gamma.
\end{eqnarray*}

Since $ \kH_{g_\gamma} = ( \kH_g )_\gamma $ by Theorem \ref{5thfeb253} and $\kH_{g_\gamma} = \kH^\prime_1 \oplus \kH^\prime_2, $ it follows that 
$$ \kH^\prime_1 = ( \kH_1 )_\gamma ~ \text{and} ~ \kH^\prime_2 = ( \kH_2 )_\gamma. $$
Hence, by Theorem \ref{thm:twisted_chern}, we get
$$ \nabla_{\rm{Ch}, \Omega^{(1,0)}_\gamma} = \varphi_{  \Omega^{1}, \Omega^{(1,0)}}^{-1} \Gamma \big(\nabla_{\rm{Ch}, \Omega^{(1,0)}}\big) ~ \text{and} ~  \nabla_{\rm{Ch}, \Omega^{(0,1)}_\cot} = \varphi_{  \Omega^1, \Omega^{(0,1)}}^{-1} \Gamma \big( \nabla_{\rm{Ch},\Omega^{(0,1)}}\big). $$

Finally, by virtue of \eqref{1stmarch251}, we can write
	\begin{align*}
	\nabla_{\Omega^1_\gamma}  & =	\varphi_{\Omega^1, \Omega^1}^{-1} \Gamma ( \nabla_{\Omega^1} )\\
        &= \varphi_{  \Omega^1, \Omega^1}^{-1} \Gamma \big(\nabla_{\rm{Ch}, \Omega^{(1,0)}} \oplus \nabla_{\rm{Ch},\Omega^{(0,1)}}\big)\\
		&= \varphi_{  \Omega^1, \Omega^1}^{-1} \Gamma \big(\nabla_{\rm{Ch}, \Omega^{(1,0)}}\big) \oplus  \varphi_{  \Omega^1, \Omega^1}^{-1} \Gamma \big( \nabla_{\rm{Ch},\Omega^{(0,1)}}\big)\\
		&=  \varphi_{  \Omega^{1}, \Omega^{(1,0)}}^{-1} \Gamma \big(\nabla_{\rm{Ch}, \Omega^{(1,0)}}\big) \oplus  \varphi_{  \Omega^1, \Omega^{(0,1)}}^{-1} \Gamma \big( \nabla_{\rm{Ch},\Omega^{(0,1)}}\big)\\
		&= \nabla_{\rm{Ch}, \Omega^{(1,0)}_\gamma} \oplus \nabla_{\rm{Ch}, \Omega^{(0,1)}_\cot}. 
	\end{align*}
 This proves the theorem.	
\end{proof}

\subsection{On examples of unitary cocycles} \label{4thmarch254}

In this article, our main examples of cocycles come from cocycles associated with compact quantum group algebras. Let us start by recalling the definition of a compact quantum group in which $\otimes_{{\rm min}}$ will denote the minimal tensor product of two $C^*$-algebras. 

A compact quantum group is a pair $ ( \mathbb{G}, \Delta ) $ where $ C ( \mathbb{G} ) $ is a unital $C^*$-algebra and $\Delta: C ( \mathbb{G} ) \rightarrow  C ( \mathbb{G} ) \otimes_{{\rm min}} C ( \mathbb{G} ) $ is a coassociative unital $C^*$-algebra homomorphism  such that 
$\Delta ( C ( \mathbb{G} ) ) ( 1 \otimes C ( \mathbb{G} ) ) $ and $\Delta (  C ( \mathbb{G} )  ) (   C ( \mathbb{G} ) \otimes 1 ) $ are both dense in $C ( \mathbb{G} ) \otimes_{{\rm min}} C ( \mathbb{G} ).$

We will assume familiarity with the theory of compact quantum groups for which we refer to \cite{woroleshouches}. However, we will be mostly using the notations of the monograph \cite{NeshveyevTuset} in this subsection. 

Woronowicz (\cite{woroleshouches}) proved an analogue of Peter-Weyl theorem for compact quantum groups in terms of irreducible unitary corepresentations. We refer to   \cite[Chapter 1]{NeshveyevTuset} for the relevant definitions. For a compact quantum group $ ( \mathbb{G}, \Delta ), $  $\mathbb{C}[ \mathbb{G}]$ will denote the dense $\ast$-subalgebra of $C ( \mathbb{G} ) $ spanned by the matrix coefficients of irreducible unitary corepresentations of $\mathbb{G}.$  Then $\mathbb{C}[ \mathbb{G} ]$ is a Hopf $\ast$-algebra. 

\begin{definition} \label{18thjan251}
A Hopf $\ast$-algebra $A$ is called a compact quantum group algebra if $A$ is isomorphic to $\mathbb{C}[ \mathbb{G} ]$ as Hopf $\ast$-algebras for some compact quantum group $\mathbb{G}.$
\end{definition}

Let $ ( \mathbb{G}, \Delta ) $ be a compact quantum group.  Following  \cite[Section 1.6]{NeshveyevTuset}, let $(\mathbb{C}[\mathbb{G}])^* $ denote the vector space dual of $ \mathbb{C}[\mathbb{G}]. $ Then $(\mathbb{C}[\mathbb{G}])^* $ is a $\ast$-algebra with multiplication of elements $f$ and $g$ given by the convolution product $ f \ast g $ 
and involution defined as
$$ f^\ast ( a ) = \overline{f ( S ( a )^\ast ) } $$
for all $ f \in (\mathbb{C}[\mathbb{G}])^* $ and $ a \in \mathbb{C}[ \mathbb{G} ].$ Moreover,  $(\mathbb{C}[\mathbb{G}] \otimes  \mathbb{C}[\mathbb{G}])^* $ is again a $\ast$-algebra with multiplicative identity given by $\epsilon \otimes \epsilon$ and $(\mathbb{C}[\mathbb{G}])^* \otimes  (\mathbb{C}[\mathbb{G}])^* $ is a proper subalgebra of   $(\mathbb{C}[\mathbb{G}] \otimes  \mathbb{C}[\mathbb{G}])^* $ in general. We have a $\ast$-homomorphism
\begin{equation} \label{31stjan261}
 \widehat{\Delta}: (\mathbb{C}[\mathbb{G}])^* \rightarrow (\mathbb{C}[\mathbb{G}] \otimes  \mathbb{C}[\mathbb{G}])^* ~ {\rm defined} ~ {\rm by} ~ \widehat{\Delta} ( f ) ( a \otimes b ) = f ( a b ).
 \end{equation}

Now we are in a position to make the following definition:

\begin{definition} (\cite[Section 3.1]{NeshveyevTuset})
If $ ( \mathbb{G}, \Delta )$ is a compact quantum group, a dual $2$-cocycle on $\mathbb{G}$ is an invertible element $\mathcal{F} \in (\mathbb{C}[\mathbb{G}] \otimes  \mathbb{C}[\mathbb{G}])^* $ such that 
$$ ( \mathcal{F} \otimes 1 ) ( \widehat{\Delta} \otimes {\rm id} ) ( \mathcal{F}  ) = ( 1 \otimes \mathcal{F} ) ( {\rm id} \otimes \widehat{\Delta}  ) ( \mathcal{F} ). $$
$\mathcal{F}$ is called unitary if $\mathcal{F}^\ast \mathcal{F} = \epsilon \otimes \epsilon.$
\end{definition}

Let us note the following well-known fact (see, for example \cite[Theorem 4.10]{SadeDeformSpecTrip}) :

\begin{remark} \label{18thdec24jb1}
If $A = \mathbb{C}[\mathbb{G}] $ is a compact quantum group algebra in the sense of  Definition \ref{18thjan251}, then a (unitary) $2$-cocycle on $A$ (in the sense of Definition \ref{18thjan252} and Definition \ref{18thjan253}) is nothing but a (unitary) dual $2$-cocycle on $\mathbb{G}.$    
\end{remark}

Now we collect some examples of cocycles which are well-known to the experts. 

\begin{example} \label{19thdec24jb1}
Consider the compact group $G = \mathbb{T}^n.$ Then all irreducible unitary representations of $G$ are one dimensional and indexed by the dual group $\mathbb{Z}^n.$ In what follows, we will denote an element $ ( m_1, \cdots m_n   ) \in \mathbb{Z}^n $ by the symbol $\underline{m}$ and the associated (one-dimensional) subspace of matrix coefficients by $\mathbb{C}[\mathbb{T}^n]_{\underline{m}}.$

We fix an $n \times n$ skew-symmetric matrix $\theta$ and define 
$$\mathcal{F} \in (\mathbb{C}[\mathbb{T}^n] \otimes  \mathbb{C}[\mathbb{T}^n])^* ~ {\rm  by} ~ \mathcal{F} ( u_{\underline{m}} \otimes u_{\underline{n}} ) = e^{2\pi \sqrt{-1} \langle \langle \theta \underline{m}, \underline{n} \rangle \rangle } $$
for all $u_{\underline{m}} \in \mathbb{C}[\mathbb{T}^n]_{\underline{m}} $ and $u_{\underline{n}} \in \mathbb{C}[\mathbb{T}^n]_{\underline{n}} $ and where $\langle \langle ~ , ~ \rangle \rangle$ denotes the usual Euclidean inner product.

Then $\mathcal{F}$ is a unitary dual $2$-cocycle on $\mathbb{T}^n$
and hence a unitary cocycle on the Hopf $\ast$-algebra $\mathbb{C}[\mathbb{T}^n]$ by Remark \ref{18thdec24jb1}. Since $\mathbb{C}[\mathbb{T}^n]$ is cocommutative, it is well-known (see Remark \ref{9thsep251})  that the cocycle deformation of the Hopf $\ast$-algebra $\mathbb{C}[\mathbb{T}^n]$  is again isomorphic to $\mathbb{C}[\mathbb{T}^n].$
\end{example}

\begin{example} \label{4thmarch253}
Suppose that $\mathbb{G}, \mathbb{H}$ be compact quantum groups such that there exists a surjective homomorphism of Hopf $\ast$-algebras. Then by  \cite[Example 3.1.7]{NeshveyevTuset}, a dual (unitary) $2$-cocycle on $H$ induces a dual (unitary) $2$-cocycle on $G.$

In particular, if $H = \mathbb{T}^n, $ then the cocycle $\mathcal{F}$ of Example \ref{19thdec24jb1} induces a unitary $2$-cocycle on $\mathbb{G}.$
\end{example}

\subsection{Cocycle deformations of a class of  K\"ahler  manifolds} \label{11thsep252}

In this subsection, we apply Theorem \ref{29thjan253} to the classical situation. The framework of this article is that of cocycles on Hopf algebras coacting on unital algebras. Therefore,  the natural classical setting to which our results apply is when there is an action of varieties $G \times X \rightarrow X, $ where $G$ is a smooth real affine algebraic group and $X$ a smooth real affine algebraic variety. With the only exception of Theorem \ref{thm:classical-deformed}, we do not claim any originality for the exposition in this subsection, since all the results are well-known to the experts. 

We will denote the set of all prime ideals of a ring $R$  by $\text{Spec} ( R ), $ while  the ideal generated by elements $x_1, \cdots x_r$ in a ring $R$ will be denoted by $(x_1, \cdots, x_r).$

Suppose that  $X  $ is  a real smooth affine variety. Thus, $ X = \text{Spec} ( O_{\mathbb{R}} ( X )  ) $ where $ O_{\mathbb{R}} ( X ) := \frac{\mathbb{R}[x_1, \cdots, x_n]}{(p_1, \cdots, p_r)} $ for polynomials $p_1, \cdots, p_r$ in  $\mathbb{R}[x_1, \cdots, x_n]$ and we assume that $O_{\mathbb{R}} ( X ) $ is an integral domain which is smooth, i.e, the localisation of $O_{\mathbb{R}} ( X ) $ at every maximal ideal is smooth (see \cite[Proposition 3.6]{uli}). 
For a real smooth affine variety $X$ as above, we let 
$$ X ( \mathbb{R} ):= \{ I \in X: \frac{O_{\mathbb{R}} ( X )}{I} \cong \mathbb{R} \} $$
denote the set of all real points of the variety $X.$ In what follows, we will always assume that $ X ( \mathbb{R} )  $ is non-empty. In this case, it is well-known (see \cite[page 105]{shafarevich}) that $ X ( \mathbb{R} ) $ is a smooth manifold. 

Thus, elements of $ O_{\mathbb{R}} ( X ) $ are in particular smooth functions on $ X ( \mathbb{R} ) $ and thus it makes sense to apply the de Rham differential $d$ to an element of $ O_{\mathbb{R}} ( X ). $ Therefore, we can define the $O_{\mathbb{R}} ( X )$-bimodule 
	$$ \Omega^1_{\mathbb{R}} ( X ( \mathbb{R} ) ):= \text{Span} \{ f dg: f, g \in O_{\mathbb{R}} ( X )  \}. $$
	It is well-known that $ \Omega^1_{\mathbb{R}} ( X ( \mathbb{R} ) ) $ is finitely generated and projective as an $O_{\mathbb{R}} ( X )$-module. If $\wedge$ denotes the classical wedge map and we define
$$ \Omega^k_{\mathbb{R}} ( X ( \mathbb{R} ) ) = {\rm Span} \{ b_0 db_1 \wedge db_2 \wedge \cdots \wedge db_k: b_i \in O_{\mathbb{R}} ( X )    \}, $$
then $ ( \Omega^\bullet_{\mathbb{R}} ( X ( \mathbb{R} ) ), \wedge, d ) $ is a differential calculus
on the real algebra $O_{\mathbb{R}} ( X ).$

 Now let $G = \text{Spec} ( O_{\mathbb{R}} ( G ) ) $ denote a real smooth affine algebraic group with the set of real points $G ( \mathbb{R} ). $ Once again, we assume that $G (\mathbb{R}) $ is non-empty. It is well-known (see \cite[page 9]{waterhouse}) that $ O_{\mathbb{R}} ( G ) $ is a Hopf algebra over the field of real numbers.  Now if we have an action of the variety $G$ on the variety $X$ (see \cite[Subsection 3.1]{waterhouse}), then we have a group action of $ G ( \mathbb{R} ) $ on $ X ( \mathbb{R} ) .$ Consequently, $ O_{\mathbb{R}} ( X ) $ is an $ O_{\mathbb{R}} ( G ) $ comodule algebra by \cite[Subsection 3.2]{waterhouse}. Moreover, it follows that  $ ( \Omega^\bullet_{\mathbb{R}} ( X ( \mathbb{R} ) ), \wedge, d )$ is an $ O_{\mathbb{R}} ( G ) $-covariant differential calculus.

\begin{prop} \label{14thaugust251}
If  we have  an action of varieties $G \times X \rightarrow X$ where $G$ and $X$ are as above and  $\widetilde{g_{\mathbb{R}}}: \Omega^1_{\mathbb{R}} ( X ( \mathbb{R} ) ) \otimes_{O_{\mathbb{R}} ( X )} \Omega^1_{\mathbb{R}} ( X ( \mathbb{R} ) ) \rightarrow O_{\mathbb{R}} ( X ) $ is a $G ( \mathbb{R} ) $-invariant Riemannian metric on $\Omega^1_{\mathbb{R}} ( X ( \mathbb{R} ) ),$ then there exists a unique  $O_{\mathbb{R}} ( G )$-covariant connection $\nabla_{\mathbb{R}}$ on $\Omega^1_{\mathbb{R}} ( X ( \mathbb{R} ) )$ which is torsionless and compatible with $\widetilde{g_{\mathbb{R}}}.$
 \end{prop}
 \begin{proof}
This is merely the fundamental theorem of Riemannian geometry. We only need to observe that the Levi-Civita connection restricts to a connection on the module $\Omega^1_{\mathbb{R}} ( X ( \mathbb{R} ) ). $ But this is again well-known and follows from our assumption that $\widetilde{g_{\mathbb{R}}}$ is $O_{\mathbb{R}} (X)$-valued. This can be verified for example by observing that the differential calculus $(\Omega^1_{\mathbb{R}} ( X ( \mathbb{R} ) ), \wedge, d ) $ satisfies the hypotheses of \cite[Theorem 6.1]{KoszulJDG}.
 
Secondly, since $\widetilde{g_{\mathbb{R}}}$ is $G ( \mathbb{R} ) $-invariant, the group $G ( \mathbb{R} ) $ acts by isometries on the Riemannian manifold $X ( \mathbb{R} ). $ Then it is well-known (see the first assertion of \cite[Proposition 5.6]{johnmlee}) that $\nabla_{\mathbb{R}}$ is $G ( \mathbb{R} ) $-equivariant, i.e, $\nabla_{\mathbb{R}}$ is $O_{\mathbb{R}} (G) $-covariant.
\end{proof}

Now we complexify this framework. Thus, we define the complex algebra 
$$ O_{\mathbb{C}} ( X ):= O_{\mathbb{R}} ( X ) \otimes_{\mathbb{R}} \mathbb{C}$$
and endow it  with the  complex antilinear involution $\ast$ defined as
   $$( f \otimes_{\mathbb{R}} \lambda )^\ast = f \otimes_{\mathbb{R}} \overline{\lambda} $$
 for all $f \in  O_{\mathbb{R}} ( X )$ and $ \lambda \in \mathbb{C}.$  
The involution  on $O_{\mathbb{C}} ( X )$ induces an involution (to be denoted by $\ast$ by an abuse of notation) on 
$$ \Omega^k ( X ( \mathbb{R} ) ):= \text{Span} \{ b_0 db_1 \wedge \cdots db_k: b_i \in  O_{\mathbb{C}} ( X ) \}    $$
so that $(\Omega^k ( X ( \mathbb{R} ) ), \wedge, d) $ becomes a $\ast$-differential calculus. 
	Similarly,  
    $$ O_{\mathbb{C}} ( G ) := O_{\mathbb{R}} ( G ) \otimes_{\mathbb{R}} \mathbb{C} $$
    is a  Hopf $\ast$-algebra over the field of complex numbers and  $ O_{\mathbb{C}} ( X ) $ is an $ O_{\mathbb{C}} ( G ) $ comodule $\ast$-algebra. 
The Riemannian metric $\widetilde{g_{\mathbb{R}}}$  of Proposition \ref{14thaugust251} canonically complexifies to an $O_{\mathbb{C}} (X) $-linear  map 
  $$\widetilde{g}: \Omega^1 ( X ( \mathbb{R} ) ) \otimes_{O_{\mathbb{C}} ( X )}  \Omega^1 ( X ( \mathbb{R} ) ) \rightarrow  O_{\mathbb{C}} ( X )$$
 defined as $\widetilde{g} ( ( \eta_1 \otimes_{\mathbb{R}} \lambda ) \otimes_{O_{\mathbb{C}} ( X )} (\eta_2 \otimes_{\mathbb{R}} \mu)  ) = \widetilde{g}_{\mathbb{R}} (\eta_1 \otimes_{O_{\mathbb{R}} ( X )} \eta_2) \lambda \mu $ for all $\eta_1, \eta_2 \in \Omega^1_{\mathbb{R}} (X(\mathbb{R})) $ and $\lambda, \mu \in \mathbb{C}.$
  It follows that
  \begin{equation} \label{13thaugust251}
\widetilde{g} ( \eta_1 \otimes_{O_{\mathbb{C}} ( X )} \eta_2 )^\ast = \widetilde{g} ( \eta_2^\ast \otimes_{O_{\mathbb{C}} ( X )} \eta_1^\ast )
\end{equation}
for all $\eta_1, \eta_2 \in \Omega^1 ( X ( \mathbb{R} ) ).$  

 Let us observe that  $\widetilde{g}$ canonically corresponds to a metric $(g, ( ~, ~ ) )$ in the sense of \cite{BeggsMajid:Leabh}. Indeed, consider the map
$$ V: \Omega^1 ( X ( \mathbb{R} ) ) \rightarrow \hm{O_{\mathbb{C}} ( X )}{\Omega^1 ( X ( \mathbb{R} ) )}{O_{\mathbb{C}} ( X )}  ~ {\rm defined} ~ {\rm by} ~ V ( \eta_1 ) ( \eta_2 ) = \widetilde{g} ( \eta_1 \otimes_{O_{\mathbb{C}} ( X )} \eta_2 ). $$
Then $V$ is an $O_{\mathbb{R}} ( X )$-linear $O_{\mathbb{R}} ( G ) $-covariant isomorphism. 
We define
$$ ( ~, ~ ): \Omega^1 ( X ( \mathbb{R} ) ) \otimes_{O_{\mathbb{C}} ( X )} \Omega^1 ( X ( \mathbb{R} ) ) \rightarrow O_{\mathbb{C}} ( X ) ~ {\rm as} ~ ( \eta_1, \eta_2 )  = V ( \eta_2 ) ( \eta_1 ). $$
Next, as $\Omega^1 ( X ( \mathbb{R} ) ) $ is finitely generated and projective as an $O_{\mathbb{C}} ( X )$-module, there exists a natural number $n$ and elements   $e^i\in \Omega^1 ( X ( \mathbb{R} ) ) , e_i \in \prescript{}{O_{\mathbb{C}} ( X )}{\Hom(\Omega^1 ( X ( \mathbb{R} ) ),O_{\mathbb{C}} ( X ))}$ for all $1\le i\le n$,  such that:
$$\eta =\sum_i e_i(\eta) e^i ~ \text{and} ~  f =\sum_i e_i f(e^i)$$
for all $ \eta \in \Omega^1 ( X ( \mathbb{R} ) ) $ and for all $ f\in \prescript{}{O_{\mathbb{C}} ( X )}{\Hom(\Omega^1 ( X ( \mathbb{R} ) ),O_{\mathbb{C}} ( X ))}$.
Now we can define 
$$ g = \sum^n_{i = 1} V^{-1} ( e_i ) \otimes_{O_{\mathbb{C}} ( X )} e^i. $$
It is easy to verify that $ ( g, ( ~, ~ ) ) $ is an $O_{\mathbb{C}} ( G ) $-covariant metric on $ \Omega^1 ( X ( \mathbb{R} ) ). $ Moreover, since $\widetilde{g}$ satisfies the equation \eqref{13thaugust251},
 it follows that the metric $ ( g, ( ~, ~ ) ) $ is actually real. 
Thus, we can state the following:

\begin{theorem} \label{18thaugust251}
If  we have  an action of varieties $G \times X \rightarrow X$ where $G$ and $X$ are as above and  $\widetilde{g_{\mathbb{R}}}: \Omega^1_{\mathbb{R}} ( X ( \mathbb{R} ) ) \otimes_{O_{\mathbb{R}} ( X )} \Omega^1_{\mathbb{R}} ( X ( \mathbb{R} ) ) \rightarrow O_{\mathbb{R}} ( X ) $ is a $G ( \mathbb{R} ) $-invariant Riemannian metric on $\Omega^1_{\mathbb{R}} ( X ( \mathbb{R} ) ),$ then the   connection $\nabla_{\mathbb{R}}$ of Proposition \ref{14thaugust251} complexifies to a connection $\nabla$ which is the  unique  $O_{\mathbb{C}} (G)$-covariant torsionless connection  on $\Omega^1 ( X ( \mathbb{R} ) )$  compatible with the real metric $(g, ( ~, ~ ) ).$ 
\end{theorem}

At this stage, we introduce almost complex structures. We will continue with the notations introduced before and  say that $X(\mathbb{R})$ has an almost complex structure if there exists an $O_{\mathbb{R}}(X)$-linear map $$I: \Omega^1(X(\mathbb{R})) \to \Omega^1(X(\mathbb{R})) \quad  \text{such that } I^2=-\id. $$ 
The map $I$ extends to an $O_\mathbb{C}(X)$-linear map (to be denoted by $I$ again) from $\Omega^1(X(\mathbb{R}))$to itself.  We define $\Omega^{(1,0)}(X(\mathbb{R}))$ and $\Omega^{(0,1)}(X(\mathbb{R }))$ to be the eigenspaces of $I$ for the eigenvalues $\sqrt{-1}$ and $-\sqrt{-1}$ respectively. Hence, 
$$\Omega^1(X(\mathbb{R}))= \Omega^{(1,0)} (X(R))\oplus\Omega^{(0,1)}(X(\mathbb{R})).$$
 We will denote the  projection maps from  $\Omega^1(X(\mathbb{R}))$ onto $\Omega^{(1,0)}(X(\mathbb{R}))$ and $\Omega^{(0,1)}(X(\mathbb{R}))$ by the symbols $\pi^{(1,0)}$ and $\pi^{(0,1)}$ respectively. Thus we have the maps 
$$\partial= \pi^{(1,0)}\circ d: O_{\mathbb{C}} (X) \rightarrow \Omega^{(1,0)} (X(\mathbb{R})) ~ \text{and} ~ \overline{\partial}={\pi^{(0,1)}}\circ d: O_{\mathbb{C}} (X) \rightarrow \Omega^{(0,1)} (X(\mathbb{R})). $$
Let us define the higher $(p,q)$ forms as
$$\Omega^{(p,q)}(X(\mathbb{R}))= \text{span}_\bbC\{ b_0 \partial b_1 \wedge\cdots \wedge \partial b_p\wedge \overline{\partial}b_{p+1} \wedge \cdots \wedge \overline{\partial}b_{p+q}: b_i \in O_{\bbC}(X) \}$$
which leads to  the decomposition (\cite[Corollary 2.6.8]{DanHuybrechts}) 
$$\Omega^k(X(\mathbb{R}))= \oplus_{p+q=k} \Omega^{(p,q)}(X(\mathbb{R})).$$
  The maps $\partial$ and $\delbar{}$ extend to the higher forms  by defining  
  $\partial= \pi^{(p+1,q)}\circ d: \Omega^{p+q} (X(\mathbb{R})) \rightarrow \Omega^{(p + 1, q)} (X(\mathbb{R})) $ and $   \delbar{}= \pi^{(p,q+1)}\circ d: \Omega^{p+q} (X(\mathbb{R})) \rightarrow \Omega^{(p, q + 1)} (X(\mathbb{R})).$

 If $I$ is an almost complex structure on $X(\mathbb{R})$ and $\widetilde{g_{\mathbb{R}}}: \Omega^1_{\mathbb{R}} (X(\mathbb{R})) \otimes_{O_{\mathbb{R}} (X)} \Omega^1_{\mathbb{R}} (X(\mathbb{R})) \rightarrow O_{\mathbb{R}} (X) $ is  a Riemannian metric on $X (\mathbb{R}),$ then $(\widetilde{g_{\mathbb{R}}}, I)$ is said to be an Hermitian structure  (\cite[Definition 3.1.1]{DanHuybrechts}) on $X(\mathbb{R}) $  if 
$$\widetilde{g_\mathbb{R}} \big(I(\eta_1) \ot_{O_{\mathbb{R}}(X)} I(\eta_2)\big) = \widetilde{g_\mathbb{R}}\left(\eta_1\ot_{O_{\mathbb{R}}(X)} \eta_2 \right). $$
If $\{f_i, f^i:i=1,2, \cdots ,n\}$ is a dual basis of the finitely generated and projective $O_{\mathbb{R}} (X)$-module $\Omega^1_\mathbb{R}(X(\mathbb{R})),$ then the fundamental form $\omega$ for the pair $(\widetilde{g_\mathbb{R}},I)$ is the element 
\begin{equation} \label{eq:omega}
    \omega= \sum_i I^{-1}V^{-1}(f_i)\wedge f^i \in \Omega^2_{\mathbb{R}} (X(\mathbb{R})) \subseteq \Omega^2 (X(\mathbb{R})). 
\end{equation}
It is well known that $\omega \in\Omega^{(1,1)}(X(\mathbb{R})) $ (\cite[Section 1.3]{DanHuybrechts}). We have seen that $\widetilde{g}_{\mathbb{R}}$ complexifies to an $O_{\mathbb{C}} (X) $-linear map $\widetilde{g}: \Omega^1 (X(\mathbb{R})) \otimes_{O_{\mathbb{C}} (X)} \Omega^1 (X(\mathbb{R})) \rightarrow O_{\mathbb{C}} (X) $ and then we get a real metric $\metric$ on $\Omega^1 (X(\mathbb{R})).$ Since $(\widetilde{g_{\mathbb{R}}}, I)$ is an Hermitian structure, it is well-known and easy to check that 
\begin{equation} \label{6thsep251}
(\eta_1, \eta_2)=0 ~ \text{if} ~ \eta_1,\eta_2 \in \Omega^{(1,0)}(X(\mathbb{R})) ~ \text{or if} ~ \eta_1,\eta_2 \in \Omega^{(0,1)}(X(\mathbb{R})).
\end{equation}
 Indeed, if $\eta_1,\eta_2 \in \Omega^{(1,0)}(X(\mathbb{R})),$ then
$$(\eta_1, \eta_2)= \widetilde{g}(\eta_1,\eta_2)= \widetilde{g}\big(I(\eta_1),I(\eta_2)\big)= \widetilde{g}\big(\sqrt{-1}\eta_1,\sqrt{-1}\eta_2\big)= - \widetilde{g}\big(\eta_1,\eta_2\big)= -(\eta_1, \eta_2)$$ 
as $I(\eta) =\sqrt{-1}\eta$ for $\eta \in  \Omega^{(1,0)}(X(\mathbb{R}))$. Similarly, if $\eta_1,\eta_2 \in \Omega^{(0,1)}(X(\mathbb{R}))$, then $(\eta_1, \eta_2)=0.$

\begin{definition} \label{def:kahler-manifold} 
Suppose that $ \widetilde{g_{\mathbb{R}}}:  \Omega^1_{\mathbb{R}} (X(\mathbb{R})) \otimes_{ O_{\mathbb{R}} (X)}  \Omega^1_{\mathbb{R}} (X(\mathbb{R})) \rightarrow  O_{\mathbb{R}} (X)  $ is a $G(\mathbb{R})$-invariant Riemannian metric on the space of one-forms  $\Omega^1_\R(X(\R)) $ of the $O_{\mathbb{C}} (G) $-covariant differential calculus $(\Omega^\bullet (X(\mathbb{R})), \wedge, d ).$ If $( \widetilde{g_{\mathbb{R}}}, I)$ is an Hermitian structure as above and $\omega$ is the fundamental form as in \eqref{eq:omega}, then $\big(X(\mathbb{R}), I, \widetilde{g_\mathbb{R}}\big)$ is said to be an  $O_{\mathbb{C}} (G) $-covariant 
K\"{a}hler structure if the following properties are satisfied:
\begin{enumerate}
\item $d = \partial + \overline{\partial}$ as maps on $\Omega^\bullet (X(\mathbb{R})) $ (i.e, $X(\mathbb{R}) $ is a complex manifold),
\item $\Omega^{(p,q)} (X(\mathbb{R})) $ is an $O_{\mathbb{C}} (G) $ comodule for all $(p,q),$
\item $d \omega = 0$ and $\omega$ is  $O_{\mathbb{C}} (G) $ coinvariant.
\end{enumerate}
\end{definition}



\begin{theorem} \label{thm:classical-deformed} 
Let $\big(X(\mathbb{R}), I, \widetilde{g_\mathbb{R}}\big)$ be an $O_{\mathbb{C}} (G) $-covariant   K\"{a}hler manifold as in Definition \ref{def:kahler-manifold}. If $\nabla$ denotes the $O_{\mathbb{C}} (G) $-covariant Levi-Civita connection for the real metric $\metric$ as in Theorem \ref{18thaugust251} and if $\gamma$ is a unitary $2$-cocycle on $O_\bbC(G)$, then the unique $\big(O_\bbC(G)\big)_\gamma$-covariant Levi-Civita connection for the metric $(g_\gamma, (~,~)_\gamma)$ on $\dctwisted$ coincides with  $ \nabla_{\text{Ch}, \Omega^{(1, 0)}_\gamma} \oplus \nabla_{\text{Ch}, \Omega^{(0, 1)}_\gamma}. $
\end{theorem}

\begin{proof}
We need to prove that the hypotheses of Theorem \ref{29thjan253} are satisfied for the real metric $\metric$ and the  $O_\bbC(G)$-covariant complex structure $(\Omega^{(\bullet, \bullet)}, \wedge, \partial, \overline{\partial})$ on the $O_\bbC(G)$ comodule $\ast$-algebra $O_{\mathbb{C}} (X). $

To begin with, \eqref{6thsep251} implies that \eqref{diamond} is satisfied and    since $X(\mathbb{R})$ is a complex manifold, it is well-known that the complex structure is factorizable.  Moreover,  the Levi-Civita connection on $\Omega^1 (X(\mathbb{R})) $ is a bimodule connection for the $O_\bbC(X)$-bilinear $O_\bbC(G)$-covariant map 
$$ \text{flip}: \Omega^1 ( X(\mathbb{R}) )  \otimes_{O_{\mathbb{C}} (X)} \Omega^1 ( X(\mathbb{R}) ) \rightarrow \Omega^1 ( X(\mathbb{R}) ) \otimes_{O_{\mathbb{C}} (X)} \Omega^1 ( X(\mathbb{R}) ), $$
where $ \text{flip} ( \eta_1 \otimes_{O_{\mathbb{C}} (X)} \eta_2 ) = \eta_2 \otimes_{O_{\mathbb{C}} (X)} \eta_1. $ Thus, we are left to verify \eqref{1stmarch251}.

This is again well-known (see \cite[Subsection 4.A]{DanHuybrechts}) but we provide the details for the sake of completeness. Let $\kH_g: \overline{\Omega^1(X(\mathbb{R}))} \rightarrow \prescript{}{O_{\mathbb{C}}(X)}{\text{Hom}} (\Omega^1(X(\mathbb{R})),O_{\mathbb{C}}(X)) $ be the Hermitian metric obtained from the real metric $\metric$ as in Proposition \ref{27thdec242jb}. Since $\nabla$ is compatible with $\metric,$ it can be easily checked that $\nabla$ is compatible with $\kH_g$ in the sense of \eqref{8thsep251}. Therefore, since $X(\mathbb{R})$ is K\"{a}hler, \cite[Theorem 4.3]{KobayashiVol2} implies that  
$$\nabla\circ I(\eta)= (\id \ot_{O_\bbC(X)}I)\nabla(\eta)$$
for all $\eta \in \Omega^1(X(\mathbb{R}))$.

Thus, $\nabla$ preserves the eigenspaces $\Omega^{(1,0)} (X(\mathbb{R})) $ and $\Omega^{(0,1)} (X(\mathbb{R}))$ inducing   two connections $\nabla_1$ and $\nabla_2$ on $\Omega^{(1,0)}(X(\mathbb{R}))$ and $\Omega^{(0,1)}(X(\mathbb{R}))$ respectively. Since \eqref{diamond} is satisfied for $\metric,$  Proposition \ref{lem:27thdec241} implies that 
$\kH_g= \kH_1 \oplus \kH_2$ for Hermitian metrics $\kH_1$ and $\kH_2$  on $\Omega^{(1,0)}(X(\mathbb{R}))$ and $\Omega^{(0,1)}(X(\mathbb{R})).$ 
 Moreover, the connections $\nabla_1$ and $\nabla_2$ are compatible with $\kH_1$ and $\kH_2$ respectively.  Thus, using the notations and formulas from Theorem \ref{28thfeb251}, for $\eta\in \Omega^{(1,0)}(X(\mathbb{R}))$, we obtain 
\begin{align*}
    \delbar{\Omega^{(1,0)}(X(\mathbb{R}))}(\eta) &=\theta^{(1,1)}_l \circ \delbar{}(\eta)\\
    &= \theta^{(1,1)}_l \pi^{(1,1)}(\partial+ \delbar{}) (\eta) \\
    &= \theta^{(1,1)}_l \pi^{(1,1)} d(\eta) \\
    &= \theta^{(1,1)}_l \pi^{(1,1)} \wedge\circ \nabla(\eta) \quad\big(\text{since } \nabla \text{ is torsionless}\big)\\
    &= \theta^{(1,1)}_l  \wedge (\pi^{(0,1)} \ot_{O_\bbC(X)}\id)\circ \nabla_1(\eta) \quad \big( \text{as } \eta \in  \Omega^{(1,0)}(X(\mathbb{R}))\big)\\
    &= (\pi^{(0,1)} \ot_{O_\bbC(X)}\id)\circ \nabla_1(\eta).
\end{align*}
 
Thus, the $(0,1)$-part of the connection $\nabla_1$ on $\Omega^{(1,0)}(X(\mathbb{R}))$ is $\delbar{\Omega^{(1,0)}(X(\mathbb{R}))}$ and $\nabla_1$ is compatible with $\kH_1.$ Therefore, $\nabla_1$ coincides with the Chern connection for the Hermitian metric $\kH_1$. Similarly, $\nabla_2$ coincides with the Chern connection (for the opposite complex structure $-I$ ) for the Hermitian metric $\kH_2.$

Since $\nabla = \nabla_1 \oplus \nabla_2$,    \eqref{1stmarch251} is satisfied  and the result follows from Theorem \ref{29thjan253}.
\end{proof}


As a sub-case, we can apply our results for the toric deformations of K\"ahler manifolds. Such deformations were  considered recently in \cite{meslandrennie1}. We will be using the notations, definitions and results of Subsection \ref{4thmarch254} regarding dual $2$-cocycles on compact quantum group algebra. In particular, we will be using the identification made in  Remark  \ref{18thdec24jb1}.

\begin{example}[Toric deformation] \label{8thmarch257}
Let us assume that $G(\mathbb{R}) = \mathbb{T}^n, $ the $n$-torus in the setup of  Theorem \ref{thm:classical-deformed}. 
If $n \geq 2$ and $\theta$ an $n \times n$ skew-symmetric matrix, then we have nontrivial cocycles on $\mathbb{T}^n$ of the form $\mathcal{F}$ as defined in Example \ref{19thdec24jb1}. Hence,  the $\ast$-algebra $O_{\mathbb{C}} (X) $ can be cocycle deformed to an algebra ${O_{\mathbb{C}} (X)}_{\mathcal{F}}.$

The algebra ${O_{\mathbb{C}} (X)}_{\mathcal{F}}$ is called a toric deformation of $O_{\mathbb{C}} (X).$ In the $C^*$-algebraic framework, such deformations were considered in \cite{rieffel93},  \cite{conneslandi01} and \cite{connesdubois02}.

Let us remind the reader about two simplifications that occur in this case. Firstly, since $O_{\mathbb{C}} (\mathbb{T}^n) $ is cocommutative, we already know (see Remark \ref{9thsep251}) that $(O_{\mathbb{C}} (\mathbb{T}^n))_{\mathcal{F}} $ is isomorphic to $O_{\mathbb{C}} (\mathbb{T}^n) $ as Hopf $\ast$-algebras. Secondly, the involution on $O_{\mathbb{C}}(X)$ coincides with the twisted involution $\ast_{\mathcal{F}}$ on $(O_{\mathbb{C}} (X))_{\mathcal{F}} $ defined in \eqref{4thfeb251}.  Indeed, let $\underline{m} = (m_1, \cdots, m_n) \in \mathbb{Z}^n $ and $b \in O_{\mathbb{C}} (X)$ belong to the $\underline{m}$-th spectral subspace, i.e, the coaction $\delta$ of $O_{\mathbb{C}} (\mathbb{T}^n)$ on $O_{\mathbb{C}} (X)$ satisfies
$$ \delta (b) = u_{\underline{m}} \otimes b $$ 
for  $ u_{\underline{m}} \in (O_{\mathbb{C}} (\mathbb{T}^n))_{\underline{m}}. $ Then in the notations of Example \ref{19thdec24jb1},
$$b^{\ast_{\mathcal{F}}} =  \overline{V} ( u^{\ast}_{\underline{m}}) b^\ast = e^{- 2 \pi \sqrt{-1} \langle\langle \theta ( - \underline{m}  ), \underline{m} \rangle\rangle} b^\ast = b^\ast,$$  
where we have used the fact that $\langle\langle \theta ( - \underline{m}  ), \underline{m} \rangle\rangle = 0$ by the skew-symmetry of $\theta.$
 
\end{example}

\subsection{Cocycle deformations of the Heckenberger-Kolb calculi}

We end this section by showing that Theorem \ref{29thjan253} can be applied to cocycle deformations of the Heckenberger-Kolb calculus on quantized irreducible flag manifolds. 

Let $G$ be a compact connected simply connected simple Lie group with complexified Lie algebra $\mathfrak{g}.$ Suppose that $ \pi = \{ \alpha_1, \cdots ,\alpha_n  \} $ is the set of simple roots of $\mathfrak{g}$ and $ S = \pi \setminus \{ \alpha_i \}, $ where $\alpha_i$ appears with coefficient at most $1$ in the highest root of $\mathfrak{g}.$ We will use the symbol $\mathfrak{l}_S$ to denote the Levi subalgebra for the standard parabolic subalgebra associated to the set $S.$  If $\widetilde{G}$ and $L_S \subseteq \widetilde{G} $ denote the connected (complex) Lie groups corresponding to $\mathfrak{g}$ and $\mathfrak{l}_S,$ then the homogeneous space $G/{(L_S \cap G)}$ is called the irreducible flag manifold for $G$ and the particular choice of $S.$ 

For $0 < q < 1,$ the symbol $U_q ( \mathfrak{g} ) $ will denote the Drinfeld-Jimbo quantized universal enveloping algebra of  \cite{DrinfeldICM,Jimbo1986} while $ \mathcal{O}_q ( G ) $ will denote the coordinate algebra of the dual compact quantum group $G_q.$ Moreover, we have a  quantum homogeneous space $ \mathcal{O}_q ( G/L_S ), $ called the algebra of functions on the corresponding quantized irreducible flag manifold. For more details, we refer to \cite{HKdR} and references therein. Moreover, Heckenberger and Kolb constructed (\cite{HKdR}) a unique $\mathcal{O}_q ( G ) $-covariant differential calculus $\dc$ of classical dimension on the algebra $ \mathcal{O}_q ( G/L_S ). $ This calculus is  known as the Heckenberger-Kolb calculus.

In \cite[Proposition 5.8]{MATASSA2019103477}, Matassa showed that for all but finitely many $q \in ( 0, 1 ),$ the Heckenberger-Kolb calculus admits a K\"ahler   structure.  However,  the following theorem holds for all $q \in ( 0, 1 ).$

\begin{theorem} \label{11thmarch251}
Let $\dc$ be the $A:= \mathcal{O}_q ( G ) $-covariant Heckenberger-Kolb calculus on $B:= \mathcal{O}_q ( G/L_S ). $ If $\gamma$ is a unitary $2$-cocycle on $A$ and $\metric$ an $A$-covariant real metric on $\Omega^1,$ then the conclusion of Theorem \ref{29thjan253} holds for the Levi-Civita connection on $\Omega^1_\gamma$ for the metric $\metrictwisted.$  
\end{theorem}
\begin{proof}
To begin with, the existence of the Levi-Civita connection follows from    \cite[Theorem 6.14]{LeviCivitaHK}.

The differential calculus $\dc$ has an $A$-covariant complex structure as shown in \cite {MATASSA2019103477} while from \cite[Proposition 3.11]{HKdR}, we know that this complex structure is factorizable. Moreover, \cite[Proposition 3.7]{LeviCivitaHK} proves that the metric $\metric$ satisfies \eqref{diamond} and the equation (41) of \cite{LeviCivitaHK} verifies \eqref{1stmarch251} for the Levi-Civita connection $\nabla$ on $\dc.$
Thus, all the conditions of Theorem \ref{29thjan253} are satisfied which proves the theorem.
\end{proof}  

We end this section by giving two classes  of examples of nontrivial unitary $2$-cocycles on $\mathcal{O}_q ( G ).$ 

\begin{example} \label{4thmarch252}
We have a surjective Hopf $\ast$-algebra morphism $\pi_q: \mathcal{O}_q ( G ) \rightarrow \mathbb{C} [\mathbb{T}^n],$ where $\mathbb{T}^n$ denotes the maximal torus of the compact Lie group $G.$ Whenever $n \geq 2,$ we have unitary dual $2$-cocycles $\mathcal{F}$ on $\mathbb{C}[\mathbb{T}^n]$ as in Example \ref{19thdec24jb1}.  Therefore, by Example \ref{4thmarch253}, $\mathcal{F}$ induces a unitary dual $2$-cocycle $\gamma$ on $\mathcal{O}_q ( G ).$ More concretely, for all $x, y \in \mathcal{O}_q ( G ),$  
$$\gamma(x \otimes y) := \mathcal{F}\big( \pi_q(x) \otimes \pi_q(y)\big).$$

We point out that in this case, the $\gamma$-deformed algebra structure and the $\ast$-structures on $\mathcal{O}_q (G/L_S)$ coincide with those obtained from $\mathcal{F}.$

Observe that the morphism $\pi_q$ induces a $\mathbb{C}[\mathbb{T}^n]$-coaction on $\mathcal{O}_q (G/L_S)$ given by 
$${}^{\mathcal{O}_q(G/L_S)}\delta:= (\pi_q\otimes \id) \circ \Delta, $$ 
where $\Delta: \mathcal{O}_q (G/L_S) \to \mathcal{O}_q(G) \otimes \mathcal{O}_q(G/L_S)$ is the canonical coaction. Therefore, it makes sense to deform $\mathcal{O}_q (G/L_S)$ by the cocycle $\mathcal{F}.$

If we use Sweedler's notation to write $\Delta(b)= \mone{b}\otimes \zero{b}$ for $b \in \mathcal{O}_q (G/L_S),$  then ${}^{\mathcal{O}_q(G/L_S)}\delta(b)= \pi_q(\mone{b}) \otimes\zero{b}.$

Thus,  for  $b, c\in \mathcal{O}_q(G/L_S),$ we write
\begin{align*}
    b\cdot_{{\gamma}} c =  {\gamma}\big( \mone{b}\otimes \mone{c}\big) \zero{b} \zero{c} = 
    {\mathcal{F}}\big(\pi_q( \mone{b})\otimes \pi_q(\mone{c})\big) \zero{b}\zero{c}= b\cdot_\mathcal{F} c
\end{align*}
and
\begin{align*}
    b^{*_{{\gamma}}}= \bar{{\gamma}}\big(\mone{b}^* \otimes S^{-1} (\mtwo{b})\big) \zero{b}^* = \bar{\mathcal{F}} \big(\pi_q (\mone{b}^*) \otimes \pi_q(S^{-1} (\mtwo{b}))\big) \zero{b}^* = b^{*_\mathcal{F}}.
\end{align*}
This shows that the $\gamma$-deformation of $\mathcal{O}_q(G/L_S)$ coincides with the $\mathcal{F}$-deformation. 
\end{example}

\begin{example} \cite[Chapter 3]{NeshveyevTuset}
 Consider the connected  simple Lie group $G =  \text{SO}(4n) $ so that its complexified Lie algebra $\mathfrak{g} = {\text so} (4n, \mathbb{C}). $ In this case, the proof of \cite[Corollary 3.4.2]{NeshveyevTuset} shows that we have a nontrivial dual cocycle $c$ on the center $Z(G)$ of $G.$ The proof utilizes the fact that the dual group of $Z(G)  $ can be identified with the quotient of the weight lattice by the root lattice.

 Then from \cite[Example 3.1.11]{NeshveyevTuset} shows that $c$ induces a dual unitary $2$-cocycle (to be denoted by $\mathcal{F}_c$ ) on $\mathcal{O}_q (G). $ Moreover, 
if $\lambda, \mu$ are positive integral weights of $\mathfrak{g}$ with weight spaces $V_\lambda, V_\mu$ respectively,  then
$$  \mathcal{F}_c (v_\lambda \otimes v_\mu) = c(\lambda, \mu)  v_\lambda \otimes v_\mu   $$
for all   $v_\lambda \in V_\lambda, v_\mu \in V_\mu.$ 

 Note that since $  \mathcal{F}_c \in  (\mathcal{O}_q (G) \otimes  \mathcal{O}_q (G))^*,  $ it is determined by its action on the subspaces $V_\lambda \otimes V_\mu.$  

The cocycle $ \mathcal{F}_c $ is invariant, in the sense that 
 $  \mathcal{F}_c \ast \widehat{\Delta} (f)   = \widehat{\Delta} (f) \ast  \mathcal{F}_c  $
for all $f \in  (\mathbb{C}[\mathbb{G}])^*, $ where $\widehat{\Delta}$ is as defined in \eqref{31stjan261}. In fact, it is proven that any unitary dual invariant $2$-cocycle  on $ \mathcal{O}_q (G) $ (up to coboundary) is obtained by this construction. Unfortunately, if $\mathfrak{g} \neq \text{so} (4n, \mathbb{C}), $ then there are no nontrivial dual $2$-cocycles on $Z(G).$

\end{example}

\appendix
\section{Some cocycle identities} \label{8thmarch255}

In this section, we start by deriving some identities regarding cocycles which have been crucially used in several computations throughout the article.  Throughout this section, $A$ will stand for a Hopf algebra.

Let us begin by recalling  the following  identity proved in \cite[Lemma 3.2]{TwistPAschieriMain}.
\begin{lem}
	Suppose that $\gamma$ is a $2$-cocycle on a Hopf algebra $A$. Then for all $h \in A, k \in A$,  we have the following equation:
	\begin{align} \label{eq:28thnov243}
		U(\one{h}) \coin{S(\two{h})}{k} = \co{\one{h}}{S(\two{h})k}.
	\end{align} 
\end{lem}

The  identity proved in the following lemma plays a key role in the article. It has been used  in the proof of many of the results, including  Proposition \ref{prop:5thdec241}, Theorem \ref{thm:twisted-hermitian}, Lemma \ref{14thfeb251}.

\begin{lem} 
	If $A$ is a Hopf $*$-algebra and $\gamma$ is a two cocycle on $A$ with convolution inverse $\bar{\gamma},$ then the following equation holds:
	\begin{equation} \label{2ndfeb262}
		\bar{V}(\one{k}^*)\bar{V}(\one{h}^*) \co{\two{k}^*}{\two{h}^*} =\coin{S(\one{h})^*}{S(\one{k})^*}~ \bar{V}(\two{k}^*\two{h}^*). 
        \end{equation}
		\end{lem}

\begin{proof}
 By repeated use of   \eqref{19thsept20245}, we have 
	\begin{align*}
		&\bar{V}(\one{k}^*)\bar{V}(\one{h}^*) \co{\two{k}^*}{\two{h}^*}\\
		& =\bar{V}(\one{k}^*)\coin{\two{h}^*}{{S^{-1}(\one{h}^*)}} \co{\two{k}^*}{\three{h}^*}\\
		&=\bar{V}(\one{k}^*)\coin{\two{h}^*}{{S(\one{h})^*}} \co{\two{k}^*}{\three{h}^*} \\
		&=\bar{V}(\one{k}^*)~ \co{\two{k}^*}{\three{h}^*S(\two{h})^*} \coin{\three{k}^*\four{h}^*}{S(\one{h})^*}~\text{(using \eqref{iv} of Lemma \ref{lem:formula})}\\
		&= \bar{V}(\one{k}^*)~ \co{\two{k}^*}{1} \coin{\three{k}^*\two{h}^*}{S(\one{h})^*}\\
		&= \bar{V}(\one{k}^*)  \coin{\two{k}^*\two{h}^*}{S(\one{h})^*} ~\left(\text{as $\gamma$ is unital.}\right)\\
		&= \coin{\two{k}^*}{S(\one{k})^*} ~ \coin{\three{k}^*\two{h}^*}{S(\one{h})^*}\\
		&=  \coin{\two{k}^*\three{h}^* S(\two{h})^*}{S(\one{k})^*} ~ \coin{\three{k}^*\four{h}^*}{S(\one{h})^*}\\
		&=\coin{S(\one{h})^*}{S(\one{k})^*}  \coin{\three{k}^*\three{h}^*}{S(\two{h})^*S(\two{k})^*}~~\text{(using \eqref{ii} of Lemma \ref{lem:formula})}\\
		&= \coin{S(\one{h})^*}{S(\one{k})^*}  \coin{\three{k}^*\three{h}^*}{S^{-1}(\two{h}^*)S^{-1}(\two{k}^*)}\\
		& =\coin{S(\one{h})^*}{S(\one{k})^*}  \coin{\three{k}^*\three{h}^*}{S^{-1}(\two{k}^*\two{h}^*)}\\
		&= \coin{S(\one{h})^*}{S(\one{k})^*} \bar{V}(\two{k}^*\two{h}^*).
	\end{align*}
	\end{proof}

As an application of \eqref{2ndfeb262}, we prove the following result which has been used in the proof of Proposition \ref{realmetric}, Lemma \ref{lem:17thdec242} and Theorem \ref{theorem:21stnov242}.
\begin{cor}
If $\gamma$ is a  $2$-cocycle on $A,$ then for all $h, k \in A,$ the following equation holds:
\begin{equation} \label{2ndfeb263}
\co{S(\one{h})^*}{S(\one{k})^*}
 \bar{V} ( k^\ast_{(2)} ) \bar{V} ( h^\ast_{(2)} ) = \bar{V} ( k^\ast_{(1)} h^\ast_{(1)} ) 
 \coin{ k^\ast_{(2)}}{h^\ast_{(2)}}
\end{equation}
\end{cor}
\begin{proof}
As $\gamma \ast \bar{\gamma} = \epsilon \otimes \epsilon,$ we can write 
\begin{eqnarray*}
&&\co{S ( h_{(1)} )^\ast  }{ S ( k_{(1)} )^\ast} \bar{V} ( k^\ast_{(2)} ) \bar{V} ( h^\ast_{(2)} )\\
&=& \gamma \big( S ( h_{(1)} )^\ast \otimes S ( k_{(1)} )^\ast \big) \bar{V} ( k^\ast_{(2)} ) \bar{V} ( h^\ast_{(2)} ) \gamma \big( k^\ast_{(3)} \otimes h^\ast_{(3)}  \big) \bar{\gamma} \big( k^\ast_{(4)} \otimes h^\ast_{(4)}   \big)\\
&=& \gamma \big( S ( h_{(1)} )^\ast \otimes S ( k_{(1)} )^\ast \big) \bar{\gamma} \big( S ( h_{(2)} )^\ast \otimes S ( k_{(2)} )^\ast \big)  \bar{V} \big(  k^\ast_{(3)} h^\ast_{(3)}  \big) \bar{\gamma} \big(  k^\ast_{(4)} \otimes h^\ast_{(4)}  \big)  ~ {\rm (}  {\rm by} ~ \eqref{2ndfeb262}   {\rm )}\\
&=& \epsilon ( S ( h_{(1)} )^\ast  )  \epsilon ( S ( k_{(1)} )^\ast  ) \bar{V} \big( k^\ast_{(2)}  h^\ast_{(2)} \big) \bar{\gamma} \big(  k^\ast_{(3)} \otimes h^\ast_{(3)}  \big).
\end{eqnarray*}
Now, by the definition of $\bar{V}$ and \eqref{19thsept20245}, we have
\begin{equation} \label{4thfeb252}
\bar{V} \big( k^\ast  h^\ast \big) = \bar{\gamma} \big( k^\ast_{(2)}  h^\ast_{(2)} \otimes S ( h_{(1)} )^\ast S ( k_{(1)} )^\ast  \big)
\end{equation}
and hence
\begin{eqnarray*}
&& \epsilon ( S ( h_{(1)} )^\ast  )  \epsilon ( S ( k_{(1)} )^\ast  ) \bar{V} \big( k^\ast_{(2)}  h^\ast_{(2)} \big) \bar{\gamma} \big(  k^\ast_{(3)} \otimes h^\ast_{(3)}  \big)\\
&=& \epsilon ( S ( h_{(1)} )^\ast  )  \epsilon ( S ( k_{(1)} )^\ast  ) \bar{\gamma} \big( k^\ast_{(3)}  h^\ast_{(3)} \otimes S ( h_{(2)} )^\ast S ( k_{(2)} )^\ast  \big) \bar{\gamma} \big(  k^\ast_{(4)} \otimes h^\ast_{(4)}  \big)\\
&=& \bar{\gamma} \big( k^\ast_{(2)}  h^\ast_{(2)} \otimes S ( h_{(1)} )^\ast S ( k_{(1)} )^\ast  \big) \bar{\gamma} \big(  k^\ast_{(3)} \otimes h^\ast_{(3)}  \big)\\
&=& \bar{V} \big( k^\ast_{(1)} h^\ast_{(1)} \big) \bar{\gamma} \big( k^\ast_{(2)} \otimes h^\ast_{(2)} \big)
\end{eqnarray*}
by \eqref{4thfeb252}. This completes the proof.
\end{proof}

We end this subsection by proving Lemma \ref{26thsept20241} which has been used in the proof of Theorem \ref{thm:twisted-hermitian}.

\begin{lem} \label{14thfeb252}
	Suppose that $\gamma$ is a $2$-cocycle on a Hopf algebra $A$  and $B$  a left $A$-comodule algebra. If $\cE$ and $\cF$ are objects of $\cat$ such that $\cE$ is finitely generated and projective as a left $B$-module, then  $\mathfrak{S}: \Gamma (\hm{B}{\mathcal{E}}{\mathcal{F}}) \to \hm{B_\gamma}{\Gamma(\mathcal{E})}{\Gamma(\mathcal{F})}$ defined in \eqref{fS} 
	 is an isomorphism in $\tcat$.
\end{lem}
\begin{proof}
	    Since \cite[Proposition 3.17]{TwistPAschieriMain} proves that $\fS$ is a vector space isomorphism, we  are left to show the  $B_\gamma$-bilinearity and $A_\gamma$-covariance of $\fS$.
	
	Let $b \in B_\gamma$ and $ f\in \hm{B}{\mathcal{E}}{\mathcal{F}}$. Then for all $v\in \Gamma(\cE)$, we compute 
	\begin{align*}
		&\mathfrak{S}(b\cdot_\gamma f)(v)\\
		&= \gamma(\mone{b}\otimes \mone{f}) \mathfrak{S}(\zero{b}\zero{f})(v)\\
		&= \gamma(\mone{b}\otimes \mone{f})~ \gamma \left(\mtwo{v}\otimes S(\mone{v})\mone{\left[(\zero{b}\zero{f})(\zero{v})\right]}\right)  \zero{\left[(\zero{b}\zero{f})(\zero{v})\right]}\\
		&= \gamma(\mone{b}\otimes \mone{f})~ \gamma \left(\mtwo{v}\otimes S(\mone{v})\mone{\left[\zero{f}(\zero{v}\zero{b})\right]}\right)  \zero{\left[\zero{f}(\zero{v}\zero{b})\right]}\\
		&=  \gamma\left(\mtwo{b}\otimes S\left(\mone{v}\mone{b}\right)  \mtwo{\left[f\left(\zero{v}\zero{b}\right)\right]}\right) \gamma \left(\mthree{v}\otimes S(\mtwo{v})\mone{\left[{f}(\zero{v}\zero{b})\right]}\right)  \zero{\left[{f}(\zero{v}\zero{b})\right]}\\
		& \mathrm{~(by ~\eqref{eq:27thnov241})}\\
		&=\gamma\left(\mthree{b} \otimes S(\mone{b})S\left(\mone{v}\right) \mtwo{\left[f\left(\zero{v}\zero{b}\right)\right]}\right)~
		\gamma \left(\mthree{v}\otimes \epsilon({\mtwo{b}}) S(\mtwo{v})\mone{\left[{f}(\zero{v}\zero{b})\right]}\right)\\
		& \zero{\left[{f}(\zero{v}\zero{b})\right]}(\text{by $\epsilon(\mtwo{b})\mone{b}=\mone{b}$})\\
		&=\gamma\left(\mfour{b} \otimes S(\mone{b})S\left(\mone{v}\right) \mtwo{\left[f\left(\zero{v}\zero{b}\right)\right]}\right)~\\
		& \gamma \left(\mthree{v}\otimes {\mthree{b}}S(\mtwo{b}) S(\mtwo{v})\mone{\left[{f}(\zero{v}\zero{b})\right]}\right)  \zero{\left[{f}(\zero{v}\zero{b})\right]}\\
		&= \co{\mthree{v}}{\mthree{b}} \co{\mtwo{v}\mtwo{b}}{S(\mone{v}\mone{b})\mone{\left[f(\zero{v}\zero{b})\right]}} \zero{\left[f(\zero{v} \zero{b})\right]} \left(\text{by \eqref{25thaug24}}\right)\\
		&= \mathfrak{S}(f)(\zero{v}\zero{b}) \co{\mone{v}}{\mone{b}}\\
		&= \mathfrak{S}(f)(v \cdot_\gamma b)= (b\cdot_\gamma \mathfrak{S}(f))(v).
	\end{align*}
	So, $\mathfrak{S}$ is a left $B_\gamma$-linear. Similarly, the right $B_\gamma$-linearity of $\mathfrak{S}$ can be proved.

     Finally, to prove the $A_\cot$-covariance of $\fS$, we will denote the the $A_\gamma$-coaction on the comodule $\hm{B_\cot}{\Gamma(\cE)}{\Gamma(\cF)}$ by $\del{\cot}$ and the $A$-coaction on $\Gamma(\hm{B}{\cE}{\cF})$ by $\delta$.
	So, for $e\in \Gamma(\cE),$ we have
	\begin{align*}
		&\del{\cot}(\fS(f))(e)\\
		&= \mone{\fS(f)} \ot \zero{\fS(f)}(e)\\
		&= S_\cot(\mone{e}) \cdot_\cot \mone{[\fS(f)(\zero{e})]} \ot \zero{[\fS(f)(\zero{e})]}~ \text{ ( by \eqref{eq:27thnov241} ) }\\
		&= U(e_{(-5)}) S(\mfour{e}) \bar{U}(\mthree{e}) \cdot_\cot \mone{[f(\zero{e})]} U(\mtwo{e}) \coin{S(\mone{e})}{\mtwo{[f(\zero{e})]}}\ot \zero{[f(\zero{e})]}\\
		&  (\text{by the definition of $S_\gamma $ in Proposition \ref{prop/defn:twisted-Hopf-algebra} and \eqref{eq:28thnov241}})\\
		&= U(e_{(-3)}) S(\mtwo{e})  \cdot_\cot \mone{[f(\zero{e})]}  \coin{S(\mone{e})}{\mtwo{[f(\zero{e})]}}\ot \zero{[f(\zero{e})]} ~ \text{(as $ \bar{U} \ast U = \epsilon $)} \\
		&=  U(e_{(-5)}) \co{S(\mtwo{e})}{\mthree{[f(\zero{e})]}} S(\mthree{e}) \mtwo{[f(\zero{e})]}\\
		& \quad  \quad \coin{S(\mfour{e})}{\mone{[f(\zero{e})]}} \coin{S(\mone{e})}{\mfour{[f(\zero{e})]}}\ot  \zero{[f(\zero{e})]}\\
		&= U(\mthree{e}) \coin{S(\mtwo{e})}{\mone{[f(\zero{e})]}} S(\mone{e}) \mtwo{[f(\zero{e})]} \ot \zero{[f(\zero{e})]} ~ \text{ (as $\bar{\gamma} \ast \gamma = \epsilon \otimes \epsilon $) } \\
		&= \co{\mthree{e}}{S(\mtwo{e})\mone{[f(\zero{e})]}} S(\mone{e}) \mtwo{[f(\zero{e})]}\ot  \zero{[f(\zero{e})]} (\text{by \eqref{eq:28thnov243}})\\
        &=  S(\mone{e})\mtwo{[f(\zero{e})]} \co{\mthree{e}}{S(\mtwo{e})\mone{[f(\zero{e})]}}\ot \zero{[f(\zero{e})]}\\
        &= \mone{f} \co{\mtwo{e}}{S(\mone{e})\mone{[\zero{f}(\zero{e})]}} \ot \zero{[\zero{f}(\zero{e})]} ~ \text{ (by \eqref{eq:27thnov241}) }\\
        &= \left(\mone{f}\otimes \fS(\zero{f})\right) (e)\\
        &= (\id \ot \fS )\delta(f)(e).
	\end{align*}
	Hence,  we have that $\fS$ is $A_\gamma$-covariant.
\end{proof}

\section{Bar Categories and cocycles} \label{B}

In this section, we derive some results relating to the interaction of bar categories with unitary cocycles. We begin by proving the following lemma  which was used in the proof of Theorem \ref{thm:twisted-hermitian}.

\begin{lem} \label{14thfeb251}
	Let $\gamma$ be a unitary 2-cocycle on  a Hopf $*$-algebra $A$ and $B$  a $A$-comodule $*$-algebra.  Then $\fN : \mathrm{bar}\circ \Gamma \to \Gamma \circ \mathrm{bar}$  defined in \eqref{5thdec242jb}
	is a natural isomorphism.
	\end{lem}

\begin{proof}
		We start by verifying that $\fN_\cE$ is left $A_\gamma$-covariant. Suppose that $e\in \Gamma (\cE)$. Then 
	\begin{align*}
		\del{\Gamma(\overline{\cE})}(\fN_\cE(\overline{\Gamma(e)})) &= \bar{V}(\mone{e}^*)~~ \del{\Gamma(\overline{\cE})}(\Gamma(\overline{\zero{e}}))\\
		&= \bar{V}(\mtwo{e}^*) (\mone{e}^* \otimes \Gamma(\overline{\zero{e}})) \\
		&= \bar{V}(\mfour{e}^*)~ \mthree{e}^* ~{V}(\mtwo{e}^*) \bar{V}(\mone{e}^*) \otimes  \Gamma(\overline{\zero{e}}) ~\text{(by Remark \ref{25thnov241})}\\
		&= \bar{V}(\mfour{e}^*)~ \mthree{e}^* ~{V}(\mtwo{e}^*)  \otimes \bar{V}(\mone{e}^*) \Gamma(\overline{\zero{e}})\\
		&= {\mtwo{e}^{*_\gamma}} \otimes \bar{V}\left( \mone{e}^*\right) \Gamma\left(\ol{\zero{e}}\right)\\
		&= (\id \otimes \fN_\cE) \left(\del{\overline{\Gamma(\cE)}}(\overline{\Gamma(e)})\right).
	\end{align*} 
	Therefore, $\fN_\cE$ is $A_\gamma$-covariant. Now for $b_\gamma \in B_\gamma, e  \in \Gamma( \cE)$, 
	\begin{align*}
		&\fN_\cE\left(\Gamma(b) \cdot_\gamma \overline{\Gamma(e)}\right)\\
		&= \fN_\cE \left( \overline{ \Gamma(e) \cdot_\cot \Gamma(b)^{*_\cot} }\right)\\
		&= \fN_\cE \left( \overline{ \bar{V}(\mone{b}^*)~ \Gamma(e) \cdot_\cot \Gamma(\zero{b}^{*}) }\right)\\
		&=  \overline{ \bar{V}(\mtwo{b}^*)}~ \fN_\cE \left(\overline{ \co{\mone{e}}{\mone{b}^*}\Gamma(\zero{e}\zero{b}^{*})}\right)\\
		&=   \overline{ \bar{V}(\mthree{b}^*)}~\coin{S(\mtwo{e})^*}{S(\mtwo{b}^*)^*}~\bar{V}(\mone{b}\mone{e}^*)  ~\Gamma(\overline{\zero{e}\zero{b}^{*}}) ~ (\text{by \eqref{19thSept20241} and \eqref{5thdec242jb}})\\
		&=  {V}(\mthree{b})~ \bar{V}(\mtwo{b})  \bar{V}(\mtwo{e}^*) ~\co{\mone{b}}{\mone{e}^*} \Gamma(\zero{b}\overline{\zero{e}})  ~ \text{(by using \eqref{2ndfeb262} and Lemma \ref{lem:22ndnov241})}\\
		&=\bar{V}(\mtwo{e}^*) ~\co{\mone{b}}{\mone{e}^*} ~ \Gamma(\zero{b}\overline{\zero{e}}) ~~~ \text{(by using Remark \ref{25thnov241})}\\
		& = \Gamma(b)\cdot_\gamma ~ \bar{V}(\mone{e}^*) \Gamma(\overline{\zero{e}})= \Gamma(b) \cdot_\gamma \fN_\cE(\overline{\Gamma(e)})
	\end{align*}
	and thus, $\fN_\cE$ is left $B_\gamma$-linear. Similarly, we can show that $\fN_\cE$ is right $B_\gamma $-linear. 	
	
	We also note that 
	$\fN^{-1}_\cE: \Gamma(\overline{\mathcal{E}}) \to \overline{\Gamma(\mathcal{E})}$ is given by  
	\begin{equation}
		\label{11thdec24jb1} \fN^{-1}_\cE(\Gamma(\overline{e}))= V(\mone{e}^*)~ \overline{\Gamma(\zero{e})}.
	\end{equation}
	Indeed, this can be easily checked by using that $\bar{V}$ is the convolution inverse of $V$. The fact that $\fN$ is natural can be easily verified. 
\end{proof}

We end this subsection with the following lemma:

\begin{lem} \label{10thmarch252}
	Suppose that $A$ is Hopf $*$-algebra and $\gamma$ a unitary $2$-cocycle on $A$. Let $\dc$ be an $A$-covariant $*$-differential calculus on an $A$-covariant $*$-algebra $B$. Then the space of twisted one-forms $\Omega^1_\gamma$ on $B_\gamma$ is a star object in $\tcat$ via the map 
	\begin{equation} \label{eq:18thdec241}
		\star_{\Omega^1_\gamma}: \Omega^1_\gamma \to \ol{\Omega^1_\gamma}, \text{ defined by } \star_{\Omega^1_\gamma}(\Gamma(\omega)):= \ol{\Gamma(\omega)^{*_\gamma}}.
	\end{equation} Moreover, this map satisfies the following: 
	\begin{equation} \label{eq:21stnov245}
		\fN_{\Omega^{1}}^{-1}\Gamma(\star_{\Omega^1})= \star_{\Omega^1_\gamma}. 
	\end{equation}	
\end{lem}

\begin{proof}
	By virtue of Proposition \ref{prop:5thdec241}, $\dctwisted$ is an $A_\gamma$-covariant  $*$-differential calculus on $B_\gamma$. Using Example \ref{5thdec241jb}, we conclude that  $\Omega^1_\gamma$ is a star object in the category $\tcat$. For the second claim, if $\omega\in \Omega^1$, then by \eqref{11thdec24jb1}, we have 
	\begin{align*}
		\fN^{-1}_{\Omega^1} \Gamma(\star_{\Omega^1})(\Gamma(\omega))= \fN^{-1} _{\Omega^1} (\Gamma(\ol{\omega^*}))= V(\mone{\omega}) \ol{\Gamma(\zero{\omega}^*)} = \ol{\bar{V}(\mone{\omega}^*) \Gamma(\zero{\omega}^*)}= \ol{\Gamma(\omega)^{*_\gamma}}.
	\end{align*}
	This finishes the proof.
\end{proof}

\subsection{Cocycle deformation as a bar functor}
Suppose that $B$ is a left $A$-comodule $*$-algebra and $\gamma$ is a $2$-cocycle on $A$. If $\gamma$ is a unitary $2$-cocycle,  $B_\gamma$ is a left $A_\gamma$-comodule $*$-algebra by Proposition \ref{10thmarch251} and so the monoidal category $\tcat$ is a bar category by Example \ref{5thdec241jb}. The goal of this subsection is to prove that the monoidal equivalence $\Gamma: \cat \to \tcat $ is a bar functor in the sense of the following definition. Although we have not used this fact anywhere, we feel that it is worthwhile to point this out.

\begin{definition}\cite[Definition 2.5]{BarCategory} \label{3rdmarch251}
	Let  $({\cC}, \ot_\cC, 1_\cC)$ and $({\cD}, \ot_\cD, 1_\cD)$ be  two bar categories. A bar functor $(F,\varphi, f^1)$ from
	${\cC}$ to ${\cD}$ is a monoidal functor $F:{\cC}\to{\cD}$ together with:
	
	\noindent(1)\quad  a natural equivalence $\mathrm{fb}$ between the functors $\mathrm{bar}\circ F$ and $F\circ\mathrm{bar}$. In particular,  
	$\mathrm{fb}_Y:\overline{F(Y)}\to F(\overline{Y})$ is an isomorphism for any object $Y$ in $\cC$.
	
	\noindent(2)\quad 
	$\mathrm{fb}_{1_{\cC}} \circ \ol{f^1} \circ \star_{\cD} = F(\star_{\cC} )\circ f^1$, where
	$f^1: 1_\cD\simeq{F(1_\cC)}$ is an isomorphism,
	
	\noindent(3)\quad $F(\mathrm{bb}_Y)=
	\mathrm{fb}_{\bar Y}\circ \overline{\mathrm{fb}_Y}\circ \mathrm{bb}_{F(Y)}$ and

	\noindent(4)\quad 
	if $\varphi_{X,Y}:F(X)\ot_\mathcal{D} F(Y)\to F(X\ot_\cC Y)$ 
	is the natural equivalence associated to $F$, then the following diagram commutes:
	\begin{equation} \label{11thdec24jb2}
		\begin{tikzcd}
			\overline{F(X\ot_\cC Y)} & \overline{F(X)\ot_\cD F(Y)} & \overline{F(Y)}\ot_\cD \overline{F(X)} \\
			F(\overline{X \ot_\cC Y}) & F(\overline{Y}\ot_\cC\overline{ X}) & F(\overline{Y})\ot_\cD F(\overline{ X})
			\arrow["\overline{\varphi_{X,Y}^{-1}}", from=1-1, to=1-2]
			\arrow["\mathrm{fb}_{X\ot_\cC Y}"', from=1-1, to=2-1]
			\arrow["\Upsilon_\cD", from=1-2, to=1-3]
			\arrow["\mathrm{fb}_Y \ot_\cD \mathrm{fb}_X ", from=1-3, to=2-3]
			\arrow["F(\Upsilon_\cC)", from=2-1, to=2-2]
			\arrow["\varphi^{-1}_{\overline{Y}, \overline{X}}", from=2-2, to=2-3].
		\end{tikzcd}
	\end{equation}	
\end{definition}

\begin{theorem} \label{theorem:21stnov242}
	If $\gamma$ is a unitary $2$-cocycle on $A,$ then the monoidal functor  $\Gamma: \cat \to \tcat $ is a bar functor.
\end{theorem}
\begin{proof}
	We have already seen  (in \eqref{21stnov243}) that $\Gamma: \cat \to \tcat $ is a monoidal equivalence. Moreover, by Lemma \ref{14thfeb251}, $\fN: \mathrm{bar}\circ \Gamma \to \Gamma \circ \mathrm{bar} $ is a  natural isomorphism. Thus, the candidate for $\mathrm{fb}_{\cE}$ is $\fN_{\cE}.$
    
    In our case, $\Gamma(B)= B_\gamma$ and so $f^1=\id$. Now by a verbatim adaptation of the proof of \eqref{eq:21stnov245}, we have $\star_\gamma = \fN_B^{-1} \Gamma(\star).$ This proves the second condition. 
	
	    Next, for $e\in \cE$,
	\begin{align*}
		\left(\fN_{\overline{\cE}}\circ  \overline{\fN_\cE} \circ \mathrm{bb_{\Gamma(\cE)}}\right)\Gamma(e)&= \fN_{\overline{\cE}} \left( \overline{\fN_\cE\left(\overline{\Gamma(e)}\right)}\right) = \fN_{\overline{\cE}}\left(\overline{\bar{V}(\mone{e}^*) \Gamma(\overline{\zero{e}})}\right)\\
        &= \fN_{\overline{\cE}}\left(\overline{ \Gamma(\overline{\zero{e}})}\right) \overline{\bar{V}(\mone{e}^*)}
		= V(\mtwo{e}) \bar{V}(\mone{e})\Gamma(\overline{\overline{\zero{e}}}) = \Gamma(\overline{\overline{e}})\\
		&= \Gamma(\mathrm{bb_\cE}) \Gamma(e),
	\end{align*}	
    where we have used Lemma \ref{lem:22ndnov241} and Remark \ref{25thnov241}.
	Finally, we need to check the commutativity of the diagram \eqref{11thdec24jb2}. This will follow from a computation in which the natural equivalence $\Upsilon$ for the category  $\tcat$ will be denoted by $\Upsilon_\gamma$. Suppose that $\cE$ and $\cF$ are objects in $\cat$. Then for  $e\in \cE$ and $f \in \cF$, we have 
	\begin{align*}
		&(\fN_\cF \ot_{B_\cot} \fN_\cE)\Upsilon_\gamma   (\overline{\varphi_{\cE, \cF}^{-1}})\overline{\Gamma(e\ot_{B}f)}\\
		&= (\fN_\cF \ot_{B_\cot} \fN_\cE)\Upsilon_\gamma \overline{\coin{\mone{e}}{\mone{f}} \Gamma(\zero{e}) \ot_{B_\cot} \Gamma(\zero{f})}  \\ 
		&= \overline{\coin{\mone{e}}{\mone{f}}} (\fN_\cF \ot_{B_\cot} \fN_\cE)  \overline{\Gamma(\zero{f})} \ot_{B_\cot}\overline{ \Gamma(\zero{e})}\\
		&= \co{S(\mtwo{e})^*}{S(\mtwo{f})^*} \bar{V}(\mone{e}^*) \bar{V}({\mone{f}^*})  \Gamma(\overline{\zero{f}}) \ot_{B_\gamma} \Gamma(\overline{\zero{e}}) \\
				&= \bar{V}(\mtwo{f}^*\mtwo{e}^*) \coin{\mone{f}^*}{\mone{e}^*}  \Gamma(\overline{\zero{f}}) \ot_{B_\gamma} \Gamma(\overline{\zero{e}}) ~\text{\big(by \eqref{2ndfeb263}\big)}\\
		&= \bar{V}(\mone{f}^*\mone{e}^*) \varphi^{-1}_{\overline{\cF},\overline{\cE}}  \Gamma(\overline{\zero{f}} \ot_{B} \overline{\zero{e}}) \text{ (by \eqref{eq:29thnov243})} \\
		&=  \bar{V}((\mone{e} \mone{f})^*) \varphi^{-1}_{\overline{\cF},\overline{\cE}} \Gamma(\Upsilon) \Gamma(\overline{\zero{e}\ot_{B} \zero{f}} )\\ 
		&=   \varphi^{-1}_{\overline{\cF},\overline{\cE}} \Gamma(\Upsilon) \fN_{\cE \ot_{B} \cF} \Gamma(\overline{{e}\ot_{B} {f}} ).
	\end{align*}
	This completes the proof that $\Gamma$ is a bar functor.
\end{proof}

\vspace{.5cm}

\textbf{Data Availability:} This article has no associated data.

\textbf{Conflict of interest:}  The authors declare that there is no conflict of interest.


\bibliographystyle{alpha} 
\bibliography{references}

\end{document}